\documentclass[12pt]{article}

\usepackage{amsthm}
\usepackage{amssymb}
\usepackage{amsmath}
\usepackage{amsfonts}
\usepackage{times}
\usepackage{mathtools}
\usepackage{color}
\usepackage[colorlinks]{hyperref}
\usepackage{tikz}
\usetikzlibrary{matrix}

\usepackage[paper=a4paper, top=2.5cm, bottom=2.5cm, left=2.5cm, right=2.5cm]{geometry}

\def\qed{\hfill $\vcenter{\hrule height .3mm
		\hbox {\vrule width .3mm height 2.1mm \kern 2mm \vrule width .3mm
			height 2.1mm} \hrule height .3mm}$ \bigskip}

\def \Sph{\mathbb{S}^{n-1}}
\def \RR {\mathbb R}

\def \EE {\mathbb E}
\def \ZZ {\mathbb Z}
\def \Var {\mathrm{Var}}

\def \PP {\mathbb P}
\def \eps {\varepsilon}

\def \Id {\mathbf{I}_n}
\def \COV {\mathrm{Cov}}
\def \KL {\mathrm{D_{KL}}}
\def \b {\mathbf{b}}
\def \tilt {\mathcal{T}}
\def \res {\mathcal{R}}
\def \pin {\mathcal{R}}
\def \ones {\vec{\mathbf{1}}}
\def \SG {\mathfrak{gap}}
\def \EF {\mathrm{\rho_{LS}}}
\def \Ent {\mathrm{Ent}}
\def \sgn {\mathrm{Sgn}}
\def \OP {\mathrm{OP}}
\def \COR {\mathrm{Corr}}

\def \zoom {\mathcal{S}}
\def \SI {\mathbf{\Psi}}
\def \COR {\mathbf{Cor}}
\def \ee {\mathbf{e}}

\def \TTT {\mathfrak{a}}

\def \FF {\mathcal{F}}

\def \M {\mathcal{M}}

\long\def\symbolfootnotetext[#1]#2{\begingroup
	\def\thefootnote{\fnsymbol{footnote}}\footnotetext[#1]{#2}\endgroup}

\DeclarePairedDelimiterX{\infdivx}[2]{(}{)}{%
  #1\;\delimsize\|\;#2%
}

\newtheorem{theorem}{Theorem}
\newtheorem{lemma}[theorem]{Lemma}

\newtheorem{fact}[theorem]{Fact}
\newtheorem{claim}[theorem]{Claim}

\newtheorem{proposition}[theorem]{Proposition}
\newtheorem{corollary}[theorem]{Corollary}
\theoremstyle{definition}
\newtheorem{definition}[theorem]{Definition}
\theoremstyle{remark}
\newtheorem{remark}[theorem]{Remark}

\long\def\symbolfootnotetext[#1]#2{\begingroup
\def\thefootnote{\fnsymbol{footnote}}\footnotetext[#1]{#2}\endgroup}

\DeclareMathOperator{\diag}{diag}

\DeclareMathOperator{\Cov}{Cov}

\DeclareMathOperator{\sign}{sign}

\DeclareMathOperator*{\argmin}{arg\,min}

\newcommand{\real}{\ensuremath{\mathbb{R}}}

\newcommand{\Exs}{\ensuremath{{\mathbb{E}}}}
\newcommand{\Prob}{\ensuremath{{\mathbb{P}}}}

\newcommand{\brackets}[1]{\left[ #1 \right]}
\newcommand{\parenth}[1]{\left( #1 \right)}

\newcommand{\braces}[1]{\left\{ #1 \right \}}
\newcommand{\abss}[1]{\left| #1 \right |}
\newcommand{\angles}[1]{\left\langle #1 \right \rangle}

\newcommand{\floors}[1]{\left\lfloor #1 \right \rfloor}

\newcommand{\vecnorm}[2]{\left\| #1\right\|_{#2}}

\begin{document}
\title{Localization schemes: A framework for proving mixing bounds for Markov chains}
\author{Yuansi Chen\thanks{Duke University.}~~and Ronen Eldan\thanks{Microsoft Reseach. This work was partially supported by a European Research Council Grant no. 803084. and by NSF grant no. DMS-1926686. }}
\date{}
\maketitle
\begin{abstract}
Two recent and seemingly-unrelated techniques for proving mixing bounds for Markov chains are: (i) the framework of Spectral Independence, introduced by Anari, Liu and Oveis Gharan, and its numerous extensions, which have given rise to several breakthroughs in the analysis of mixing times of discrete Markov chains and (ii) the Stochastic Localization technique which has proven useful in establishing mixing and expansion bounds for both log-concave measures and for measures on the discrete hypercube. In this paper, we introduce a framework which connects ideas from both techniques. Our framework unifies, simplifies and extends those two techniques. In its center is the concept of a ``localization scheme'' which, to every probability measure on some space $\Omega$, assigns a martingale of probability measures which ``localize'' in space as time evolves. As it turns out, to every such scheme corresponds a Markov chain, and many chains of interest appear naturally in this framework. This viewpoint provides tools for deriving mixing bounds for the dynamics through the analysis of the corresponding localization process. Generalizations of concepts of Spectral Independence and Entropic Independence naturally arise from our definitions, and in particular we recover the main theorems in the spectral and entropic independence frameworks via simple martingale arguments (completely bypassing the need to use the theory of high-dimensional expanders). We demonstrate the strength of our proposed machinery by giving short and (arguably) simpler proofs to many mixing bounds in the recent literature. In particular, we: (i) Give the first $O(n \log n)$ bound for mixing time of the hardcore-model (of arbitrary degree) in the tree-uniqueness regime, under Glauber dynamics, (ii) Give the first optimal mixing bounds for Ising models in the uniqueness regime under any external fields, (iii) Prove a KL-divergence decay bound for log-concave sampling via the Restricted Gaussian Oracle, which achieves optimal mixing under any $\exp\left (n\right )$-warm start, (iv) Prove a logarithmic-Sobolev inequality for near-critical Ferromagnetic Ising models, recovering in a simple way a variant of a recent result by Bauerschmidt and Dagallier.
\end{abstract}
\newpage
\tableofcontents
\section{Introduction}
Suppose that we would like to sample from a measure $\nu$ on some set $\Omega$. For the sake of discussion, suppose that either $\Omega = \{-1,1\}^n$ is the Boolean hypercube or $\Omega = \RR^n$. A common algorithm is to find a Markov chain whose stationary distribution is $\nu$ and which exhibits good \emph{mixing bounds}. 

In the case $\Omega=\{-1,1\}^n$, a very useful Markov chain associated to a measure $\nu$ is the Glauber dynamics, defined as follows: Given $x \in \{-1,1\}^n$, the transition kernel from $x$, $P_{x \to ~\cdot}$, is the law which describes the point $Y$ picked according to following random procedure: Pick uniformly a coordinate $i \in [n]$ and then take $Y$ according to the law $\nu$ conditioned on the event $\left \{Y_j = x_j, ~ \forall j \in [n] \setminus \{i\} \right \}$.

For a Markov chain $(X_t)_t$ in a state space $\Omega$ which has a unique stationary measure $\nu$, a mixing bound typically asserts that for every $\eps > 0$ there is a time $t(\eps)$ such that for all measurable $A \subset \Omega$ and all $t>t(\eps)$, one has $\left | \PP(X_{t} \in A) - \nu(A) \right | \leq \eps$. See \cite{MCbook} for an extensive account of this subject.

In recent years, there appeared two seemingly-unrelated new techniques which were used to establish mixing bounds through functional inequalities:
\begin{itemize}
\item
The work \cite{ALO-SI} put forth the notion of \emph{Spectral independence} and developed a framework which relies on those notions in order to establish mixing bounds for measures on the set of subsets of $[n]$. This framework relies on the theory of high-dimensional expanders. Some follow-up works which extended this technique are \cite{EI1, CLV, EI2, CFYZ21-rapid, feng2021rapid, blanca2022mixing, liu2021coupling,AAFractionally}.
\item
The \emph{stochastic localization} technique, introduced by the second author in \cite{Eldan-SL}, is the central ingredient used in the proofs of several functional inequalities, both in the continuous setting where $\Omega=\RR^n$ and $\nu$ is a \emph{logarithmically-concave} measure and in the setting of the discrete hypercube. Most notably, the technique gives the state-of-the-art bounds, due to the first author (\cite{chen2021almost}), for the so-called Kannan-Lov\'asz-Simonovits conjecture (\cite{KLS95}) and Bourgain's slicing problem (see \cite{KM-Slicing}). Some follow-up works based on this technique are \cite{Eldan-taming,Eldan-Shamir,Klartag-SL,LV-KLS,chen2021almost}.
\end{itemize}

In this work, we both unify and expand these two techniques towards a new framework which can be used to establish mixing bounds in various settings, showing that the same principles govern in a wide variety scenarios.

One of the main principles underlying both techniques is that concentration bounds on a measure can be deduced from bounds on the covariance structure of a certain family of measures which are transformations of the original measure: In the spectral independence framework, a sufficient condition for a spectral gap is the boundedness of the \emph{influence} matrices of restrictions of the measure (the influence matrix has a simple correspondence with the covariance matrix), and in the stochastic localization framework a spectral gap is implied by the boundedness of the covariance matrix along a certain stochastic process which is associated with the measure.

This work shows that with the correct point of view, those two reductions follow from the exact same argument. By approaching the notions of spectral and entropic independence from this point of view, we will be able to:
\begin{enumerate}
\item[(i)]
Generalize the theory, giving rise to a natural family of Markov chains together with a toolbox of ingredients that can be used to prove mixing bounds for those chains.
\item[(ii)]
Simplify the proofs in the foundations of the framework of spectral and entropic independence, and in particular completely bypass the need to use the theory of high-dimensional expanders.
\item[(iii)]
Provide a self-contained and (arguably) simpler proofs for many of the expansions of the spectral/entropic independence machinery, and in particular reprove in a more general context several main theorems which appear in \cite{ALO-SI,CLV,EI1,EI2,CFYZ21-rapid,CFYZ21-treeuniqueness,RGO20,RGO21}.
\end{enumerate}
\subsubsection*{Summary of applications}
To demonstrate the strength of our machinery, we apply it in several settings:
\begin{itemize}
\item
We derive the optimal mixing rate for Glauber dynamics on the hardcore model of any degree in the tree-uniqueness regime, showing that it mixes in time $O(n \log n)$. A similar mixing rate was obtained in \cite{EI2} for a different Markov chain which was tailored for the hardcore model. Additionally, our framework (arguably) allows us to significantly simplify the argument.
\item
We give the first optimal mixing bounds for graphical Ising models in the uniqueness regime under any external fields, improving the results in \cite{CFYZ21-treeuniqueness} in the sense that there is no dependence of the bound on the external field. In this case as well, the proof is significantly simpler.
\item
Provide a very simple proof of a KL-divergence decay bound for log-concave sampling via the so-called Restricted-Gaussian-Oracle introduced in \cite{RGO20}. Our bound works under any $\exp(n)$-warm start, resolving an issue raised in \cite{RGO21}.
\item
Give a self-contained and simpler proof of mixing for the Glauber dynamics for Ising models whose interaction matrix has operator norm bounded by $1$, which in particular gives optimal mixing for Glauber dynamics on the Sherrington-Kirkpatrick model in high enough temperature, recovering the result derived from \cite{EKZ,EI2}.
\item
We recover, in a simple way, a variant of a recent result by Bauerschmidt and Dagallier~\cite{bauerschmidt2022log}, proving a logarithmic-Sobolev inequality for Ferromagnetic Ising models in terms of the model's susceptibility.
\end{itemize}
\subsubsection*{Concurrent work}
Shortly before submitting this manuscript, we were informed by Chen, Feng, Yin and Zhang of a manuscript in preparation~\cite{chen2022optimal} which independently proves parts of the results regarding antiferromagnetic Ising models that we prove here, and in particular gives the same mixing bound for the hardcore model, using a different proof technique.

\subsubsection*{Mixing via functional inequalities}
Consider a reversible transition operator $P = P_{x \to y}$ with stationary measure $\nu$ on a state space $\Omega$. Its \textbf{spectral gap} is defined as the quantity
$$
\SG(P) := \inf_{\varphi:\Omega \to \RR} \frac{ \int_{\Omega \times \Omega} (\varphi(x) - \varphi(y))^2 d P_x (y) d \nu(x) }{ 2 \Var_\nu[\varphi]}.
$$
Next, define
$$
\Ent_{\nu}[f] = \int f(x) \log f(x) d \nu(x) - \int f(x) d \nu(x) \log \left ( \int f(x) d \nu(x) \right ).
$$
We define the \textbf{modified log-Sobolev Inequality (MLSI)} coefficient as
$$
\EF(P) := 1 - \sup_{f:\Omega \to [0, \infty)} \frac{\Ent_\nu[P f]}{  \Ent_\nu[f] }.
$$
\begin{remark}
The above definition is only valid if the Markov transition kernel $P$ is reversible (or self adjoint). If it is not reversible, $Pf$ needs to be replaced by $\frac{d P^* (f \nu)}{d \nu}$. In this work we only discuss the reversible case, so we stick to the simpler notation.
\end{remark}

It is standard to deduce mixing bounds for the associated Markov chain from the above quantities. Given an initial distribution $\mu$ which is absolutely continuous with respect to $\nu$, consider the total-variation mixing time
$$
t_{\mathrm{mix}}(P, \eps; \mu) = \min \left \{t >0; ~~ \left |P^t[\mu](A) - \nu(A) \right | \leq \eps, ~~ \forall A \subset \Omega \right  \}
$$
Moreover, define
$$
t_{\mathrm{mix}}(P, \eps) = \max_{x \in \Omega} t_{\mathrm{mix}}(P, \eps, \delta_x).
$$
The following fact is standard (see e.g. \cite[Theorem 12.4]{MCbook} and \cite[Fact 3.5]{CLV}).
\begin{fact} \label{fact:mixing}
Suppose that for all $x \in \Omega$ one has $\nu(\{x\}) \geq \eta$. Then,
\begin{align*}
	t_{\mathrm{mix}}(P, \eps) &\leq C \SG(P)^{-1} \left (\log(1/\eta) + \log(1/\eps) \right ), \text{ and } \\
	t_{\mathrm{mix}}(P, \eps) &\leq C \EF(P)^{-1} \left (\log \log(1/\eta) + \log(1/\eps) \right ).
\end{align*}
\end{fact}
The main aim of our machinery is to provide tools which give lower bounds for the spectral gap and modified log-Sobolev coefficient for a family of Markov chains that arise naturally in our framework.
\subsection{Review of the framework and ideas}
We now give a brief overview of the framework constructed in the next sections.

\begin{itemize}
\item
The key definitions in the framework are a localization process and a localization scheme. A \textbf{localization process} is a stochastic process $(\nu_t)_{t \geq 0}$ where each $\nu_t$ is a probability measure on some space $\Omega$, having the property that for every subset $A \subset \Omega$, the stochastic process $t \to \nu_t(A)$ is a martingale which satisfies $\lim_{t \to \infty} \nu_t(A) \in \{0,1\}$. A \textbf{localization scheme} is a mapping which, to every probability measure $\nu$ on $\Omega$, assigns a localization process $(\nu_t)_t$ with $\nu_0 = \nu$. Thus, a localization scheme can be thought of as a way to interpolate between a given measure $\nu$ and a (random) Dirac measure, via a martingale on the space of measures.
\item
Given a localization process $(\nu_t)_t$ and a time $\tau > 0$, there is a reversible Markov chain on $\nu$, whose stationary measure is $\nu = \nu_0$, which is naturally associated to the localization process. This Markov chain is defined by the formula
$$
P_{x \to y} = \EE \left [ \frac{\nu_\tau(x) \nu_\tau(y)}{\nu(x)} \right ].
$$
It turns out that many Markov chains arise via naturally defined localization schemes. In particular, the Glauber dynamics on $\left (\{-1,1\}^n, \nu \right )$ can be derived from what we call the coordinate-by-coordinate localization scheme. This scheme is defined by taking $(k_1,...,k_n)$ to be a uniformly random permutation of $[n]$, taking $X \sim \nu$ and setting $\nu_t$ to be the law of $X$ conditioned on $X_{k_1},\dots,X_{k_t}$. Some other chains that arise via this framework (using other localization schemes) are the hit-and-run walk, the up-down walk, the Restricted-Gaussian dynamics and the field dynamics.
\item
Next, we will see that there is a simple way to analyze the spectral gap and MLSI coefficient of the Markov chain associated to a localization scheme: In order to give a bound on the spectral gap, one needs to give a lower bound to the quantity $\EE \left [\frac{\Var_{\nu_\tau}[\varphi]}{\Var_{\nu}[\varphi]}\right ]$ for an arbitrary test function $\varphi: \Omega \to \RR$. Similarly, for MLSI, one needs to lower bound the quantity $\EE \left [\frac{\Ent_{\nu_\tau}[f]}{\Ent_{\nu}[f]}\right ]$ for $f:\Omega \to [0, \infty)$. Intuitively speaking, the measure-valued process $\nu_t$ ``zooms in'' on smaller and smaller portions of $\Omega$, and we want to establish that there is still some variance (or entropy) left all the way up to time $\tau$, while the measure $\nu_\tau$ is already focused on a small portion of the space.
\item
In light of the above intuition, it makes sense to analyze the time differentials of the stochastic processes $t \to \log(\Var_{\nu_t}[\varphi])$ and $t \to \log(\Ent_{\nu_t}[f])$ (which can be in either continuous or discrete time, depending on the localization scheme). Roughly speaking, we want to obtain lower bounds for their drifts. By integrating those bounds with respect to time, we could then obtain bounds on the spectral gap and MLSI. When a measure $\nu$ satisfies the lower bound $\frac{d \Var_{\nu_t}[\varphi]}{ \Var_{\nu_t}[\varphi]} \geq - \alpha dt + \mbox{martingale}$, we say that it satisfies $\alpha$-\textbf{approximate conservation of variance} bound, and similarly when $ \frac{d\Ent_{\nu_t}[f]}{\Ent_{\nu_t}[f]} \geq - \alpha dt + \mbox{martingale}$, we say that it satisfies $\alpha$-\textbf{approximate conservation of entropy}.
\item
It turns out that there is an efficient way to obtain approximate conservation of variance and entropy bounds for a family of localization schemes which we call linear-tilt localizations. Roughly speaking, these are schemes where $\frac{\nu_{t+h}}{\nu_t}$ is a linear function up to $o(h)$. We will see that many localization schemes of interest, including the coordinate-by-coordinate scheme and the stochastic localization scheme of \cite{Eldan-SL} can be described this way.

\tikzset{every picture/.style={line width=0.75pt}} 
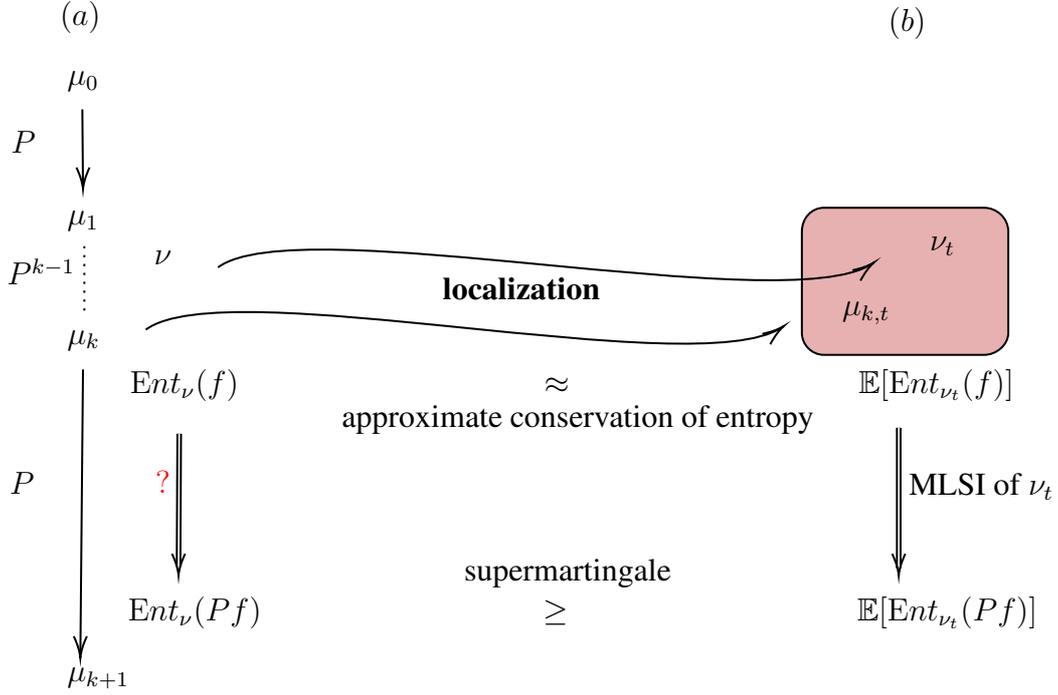
\begin{figure}[ht]
\centering
\begin{tikzpicture}[x=0.75pt,y=0.75pt,yscale=-1,xscale=1]

\draw  [fill={rgb, 255:red, 231; green, 177; blue, 177 }  ,fill opacity=1 ] (476,182) .. controls (476,175) and (481,170) .. (487,170) -- (567,170) .. controls (573,170) and (579,175) .. (579,182) -- (579,233) .. controls (579,239) and (573,244) .. (567,244) -- (487,244) .. controls (481,244) and (476,239) .. (476,233) -- cycle ;
\draw    (117,120.5) -- (117.32,158.67) ;
\draw [shift={(117.33,160.67)}, rotate = 270] [color={rgb, 255:red, 0; green, 0; blue, 0 }  ][line width=0.75]    (10.93,-3.29) .. controls (6.95,-1.4) and (3.31,-0.3) .. (0,0) .. controls (3.31,0.3) and (6.95,1.4) .. (10.93,3.29)   ;
\draw  [dash pattern={on 0.84pt off 2.51pt}]  (118,189.33) -- (118,224) ;
\draw    (117.33,251.33) -- (116.02,394.5) ;
\draw [shift={(116,396.5)}, rotate = 270] [color={rgb, 255:red, 0; green, 0; blue, 0 }  ][line width=0.75]    (10.93,-3.29) .. controls (6.95,-1.4) and (3.31,-0.3) .. (0,0) .. controls (3.31,0.3) and (6.95,1.4) .. (10.93,3.29)   ;
\draw    (184.67,200) .. controls (224.27,170.3) and (467.07,226.36) .. (510.74,198.37) ;
\draw [shift={(512,197.5)}, rotate = 143.13] [color={rgb, 255:red, 0; green, 0; blue, 0 }  ][line width=0.75]    (10.93,-3.29) .. controls (6.95,-1.4) and (3.31,-0.3) .. (0,0) .. controls (3.31,0.3) and (6.95,1.4) .. (10.93,3.29)   ;
\draw    (148.67,231.33) .. controls (188.27,201.63) and (422.25,258.35) .. (465.74,230.37) ;
\draw [shift={(467,229.5)}, rotate = 143.13] [color={rgb, 255:red, 0; green, 0; blue, 0 }  ][line width=0.75]    (10.93,-3.29) .. controls (6.95,-1.4) and (3.31,-0.3) .. (0,0) .. controls (3.31,0.3) and (6.95,1.4) .. (10.93,3.29)   ;
\draw  [double]  (524.33,280.67) -- (524.33,351.33) ;
\draw [shift={(524.33,355)}, rotate = 270] [color={rgb, 255:red, 0; green, 0; blue, 0 }  ][line width=0.75]    (10.93,-3.29) .. controls (6.95,-1.4) and (3.31,-0.3) .. (0,0) .. controls (3.31,0.3) and (6.95,1.4) .. (10.93,3.29)   ;
\draw [double] (165.33,283.67) -- (165.01,350.5) ;
\draw [shift={(165,355)}, rotate = 270] [color={rgb, 255:red, 0; green, 0; blue, 0 }  ][line width=0.75]    (10.93,-3.29) .. controls (6.95,-1.4) and (3.31,-0.3) .. (0,0) .. controls (3.31,0.3) and (6.95,1.4) .. (10.93,3.29)   ;

\draw (108,100) node [anchor=north west][inner sep=0.75pt]    {$\mu _{0}$};
\draw (108,170) node [anchor=north west][inner sep=0.75pt]    {$\mu _{1}$};
\draw (80,130) node [anchor=north west][inner sep=0.75pt]    {$P$};
\draw (79.33,301.33) node [anchor=north west][inner sep=0.75pt]    {$P$};
\draw (108,230) node [anchor=north west][inner sep=0.75pt]    {$\mu _{k}$};
\draw (108,400) node [anchor=north west][inner sep=0.75pt]    {$\mu _{k+1}$};
\draw (78,194) node [anchor=north west][inner sep=0.75pt]    {$P^{k-1}$};
\draw (151.33,190.33) node [anchor=north west][inner sep=0.75pt]    {$\nu $};
\draw (494.67,215.67) node [anchor=north west][inner sep=0.75pt]    {$\mu _{k,t}$};
\draw (538.33,183) node [anchor=north west][inner sep=0.75pt]    {$\nu _{t}$};
\draw (295.33,203.33) node [anchor=north west][inner sep=0.75pt]   [align=left] {\textbf{localization}};
\draw (139.33,249.67) node [anchor=north west][inner sep=0.75pt]    {$\mathrm{E} nt_{\nu }( f) \ \ \ \ \ \ \ \ \ \ \ \ \ \ \ \ \ \ \ \ \ \ \ \ \ \ \ \ \ \ \ \ \ \ \ \ \ \approx \ \ \ \ \ \ \ \ \ \ \ \ \ \ \ \ \ \ \ \ \ \ \ \ \ \ \ \ \ \ \ \ \ \ \ \mathbb{E}[\mathrm{E} nt_{\nu _{t}}( f)] \ \ $};
\draw (138.33,364.33) node [anchor=north west][inner sep=0.75pt]    {$\mathrm{E} nt_{\nu }( Pf) \ \ \ \ \ \ \ \ \ \ \ \ \ \ \ \ \ \ \ \ \ \ \ \ \ \ \ \ \ \ \ \ \ \ \geq \ \ \ \ \ \ \ \ \ \ \ \ \ \ \ \ \ \ \ \ \ \ \ \ \ \ \ \ \ \ \ \ \ \ \ \mathbb{E}[\mathrm{E} nt_{\nu _{t}}( Pf)] \ \ $};
\draw (528,302) node [anchor=north west][inner sep=0.75pt]   [align=left] {MLSI of $\displaystyle \nu _{t}$};
\draw (104.33,64.67) node [anchor=north west][inner sep=0.75pt]    {$( a)$};
\draw (517.67,67.67) node [anchor=north west][inner sep=0.75pt]    {$( b)$};
\draw (245,268) node [anchor=north west][inner sep=0.75pt]   [align=left] {approximate conservation of entropy};
\draw (306,345) node [anchor=north west][inner sep=0.75pt]   [align=left] {supermartingale};
\draw (152,301) node [anchor=north west][inner sep=0.75pt]    {$\textcolor[rgb]{1,0,0}{?}$};
\end{tikzpicture}
\caption{$(a)$. Probability measures encountered during the iterates of a Markov chain $P$ with target distribution $\nu$ and initial distribution $\mu _{0}$. We wish to establish the entropy decay (or MLSI) for each iteration of the Markov chain. $(b)$. Probability measures generated via a localization scheme. If a stochastic scheme is used, the red box indicates that one may generate multiple random instances of $( \nu _{t} ,\mu _{k,t})$ starting from $(\nu ,\ \mu _{k})$. The localization scheme is designed such that it is easier to establish the entropy decay for the new measures on $(b)$.
 With $f = \frac{\mu_k}{\nu}$, the entropy decay properties of the measures on $(b)$  can be related to those on $(a)$ thanks to the martingale properties of the localization. }
\label{fig:MLSI_analysis_sketch}
\end{figure}


\item
For linear-tilt localization scheme, we see that a simple argument which uses the Cauchy-Schwartz inequality shows that approximate variance conservation is related to the process of covariance matrices $t \to \COV(\nu_t)$. For the case of coordinate-by-coordinate localization this argument recovers the spectral independence framework of Anari, Liu and Oveis Gharan in a very simple way that completely bypasses the need to use high-dimensional expanders.
\item
Approximate conservation of entropy turns out to follow from a natural condition that we call \textbf{entropic stability}, which roughly says that the center of mass of the measure cannot move too much under perturbations with small relative entropy. This can be thought of as a weak form of \textbf{transportation-entropy inequalities}. In the context of the coordinate-by-coordinate localization we will see that it amounts to a very simple property of the logarithmic Laplace transform which is very similar to fractional log-concavity. This would give us a way to recover a variant of the entropic-independence framework introduced in \cite{EI1}.
\item
Another ingredient in our framework will be the concept of \textbf{annealing} via a localization scheme. Given two localization schemes on the same space, we can ``concatenate'' them by running one scheme up to some time $t$ and then running the other on the measure $\nu_t$ produced by the first scheme. What we will see is that we can produce spectral gap and MLSI bounds, with respect to the dynamics associated with the second localization scheme, by showing that the first localization scheme does not contract the variance / entropy and ``anneals'' the measure in a way that it outputs a measure to which we can ensure mixing for the second scheme. Such annealing procedure will be useful in many of the applications.
\end{itemize}

In Figure~\ref{fig:MLSI_analysis_sketch}, we summarize the conceptual diagram used to analyze the MLSI coefficient of a Markov chain via a localization scheme.

\subsection{Related techniques}

\subsubsection*{Local-to-global theorems and high-dimensional expanders}
The framework of spectral independence was originally based on \emph{local-spectral-expansion} inequalities in high-dimensional expanders, which was defined and derived in works of Alev, Dinur, Kaufmann, Lau, Mass and Oppenheim  (\cite{AlevLau, Op18, KO20, KD17, KM17}). These inequalities allow us to compare between the spectral gap of the $\ell$-down-up-walk and that of the $\ell+1$-up-down-walk on spectral expanders. This principle was first used towards establishing bounds on the spectral gap by Anari, Liu, Oveis Gharan and Vinzant \cite{LCP2}. Our framework could be thought as an alternative way to prove ``local-to-global'' theorems. 
\subsubsection*{Pathwise analysis}
Our technique can be seen as a manifestation of the pathwise analysis method (see \cite{eldan2022analysis}). This method is closely-related to the so-called \emph{semigroup technique} where, in order to prove an inequality regarding a measure, one considers an evolution of the measure via a semigroup (usually the heat semigroup) which provides an interpolation between the measure in hand and a simpler one, and the analysis amounts to showing inequalities regarding the evolution (e.g., monotonicity of certain quantities). Pathwise-analysis considers a random evolution rather than a deterministic one which allows arguments to be carried out path-by-path rather than on an average sense.
 
\subsubsection*{The logarithmic Laplace transform}
The logarithmic Laplace transform is central to our framework: Upper bounds on either itself or its Hessian turn out to be naturally related to functional inequalities of the underlying measure, and it is also used associate between inequalities regarding the center of mass and the covariance structure and relative entropies. In the context of measures on the discrete cube, it was suggested that the log-Laplace transform performs a natural role in concentration inequalities in a paper by Shamir and the second author \cite{Eldan-Shamir}. In the context of log-concave measures on $\RR^n$ the log-Laplace transform was long known to perform a central role in concentration inequalities, see \cite{Klartag-pert,KM-LL,KE-ThinSlice}.

\subsubsection*{Related techniques in Markov chains}
Our technique is also related to Markov chain decomposition techniques (\cite{MCdecomp1, MCdecomp2}) where the analysis of mixing times of a Markov chain is carried out by splitting the chain into simple-to-analyze components via restriction and projection-like operations. Another common technique in Markov chains is \emph{coupling}, of paths. The coupling technique is also based on path-by-path analysis but otherwise is very different in spirit from our technique, which does not consider paths of the random walk but rather of the measure-evolution process.

\subsubsection*{Localization techniques}
The idea of proving concentration inequalities by considering a certain scheme which converts the measures into simpler ones by localizing on space goes back to Gromov and Milman \cite{GM-LOC} and Kannan, Lov\'asz and Simonivits (\cite{KLS93,KLS95}). In these works, a measure is effectively decomposed into one-dimensional components called ``needles''. More recently, needle-decompositions of measures were also used in Riemannian geometry \cite{Klartag-Needle,Milman-QCD}.

\subsection*{Summary of notation}
\begin{itemize}
\item $\EF(\cdot)$ - Modified log-Sobolev (MLSI) coefficient.
\item $\SG(\cdot)$ - Spectral gap.
\item
$\Ent_\nu[f] := \int f(x) \log f(x) d \nu(x) - \int f d \nu \log \left ( \int f d \nu \right )$.
\item
$\KL(\mu || \nu) := \int \log \frac{d \mu}{d \nu}(x) d \mu(x)$. 
\item 
For a matrix $A \in \mathbb{M}_{n \times n}$, $\|A\|_{\OP}$ - operator norm; $\rho(A)$ - spectral norm.
\item For a measure $\nu$ on $\RR^n$:
\begin{itemize}
\item 
$\b(\nu) := \int x d \nu(x)$ - The \textbf{center of mass}.
\item $\COV(\nu) = \int (x - \b(\nu))^{\otimes 2} d \nu(x)$ - The \textbf{covariance matrix}.
\item
$\COR(\nu) := \mathrm{diag}(\COV(\nu))^{-1/2} \COV(\nu) \mathrm{diag}(\COV(\nu))^{-1/2}$, where $\diag(\cdot)$ is the diagonal matrix obtained by setting all the off-diagonal entries to $0$, called the \textbf{correlation matrix}.
\item
$\SI(\nu) := \COV(\nu) \mathrm{diag}(\COV(\nu))^{-1}$ the \textbf{influence matrix}.
\item
$\tilt_v \nu$ where $v \in \RR^n$ is defined as $\frac{d\tilt_v \nu(x)}{d \nu(x)} := \frac{e^{\langle v, x \rangle} d \nu(x)}{ \int e^{\langle v, z \rangle} d \nu(z) }$ (\textbf{exponential tilt} of a measure).
\item
$(\mathbf{e}_1,\dots,\mathbf{e}_n)$ is the standard basis of $\RR^n$.
\item $\mathbb{S}^{n-1}$ is the unit sphere in $\RR^n$.
\end{itemize}
\item
For a measure $\nu$ on $\{-1,1\}^n$,
\begin{itemize}
\item
For $u \in \{-1,0,1\}^n$ define $S_u := \bigl \{x \in \{-1,1\}^n; ~ x_i u_i \geq 0, ~ \forall i \in [n] \bigr \}$.
\item  
$\pin_u \nu$ is the normalized restriction of $\nu$ to $S_u$ (\textbf{pinning} of $\nu$ according to $u$).
\item
For $A \subset [n]$, $\res_A \nu = \pin_{\mathbf{1}_A} \nu$ (by abuse of notation, pinning of coordinates in $A$ to $+1$).
\item
$P^{\mathrm{GD}}(\nu)$ is the transition kernel for the Glauber-dynamics on $\nu$.
\end{itemize}
\end{itemize}

\subsubsection*{Acknowledgements} 
R.E. would like to thank Kuikui Liu and Bo'az Klartag for enlightening discussions. We would also like to thank Kuikui Liu, Mark Sellke and Eric Vigoda for comments, and Nima Anari and Vishesh Jain for spotting mistakes in an early version of this paper. This work was carried out when R.E. was on a fellowship at the Institute for Advanced Study, Princeton.

\section{Localization schemes and their associated dynamics}
In this section, first we introduce the central object in our framework: a \textbf{localization scheme}. Second we discuss the dynamics associated with a localization scheme. Finally, we present several examples of localization schemes.

We fix a set $\Omega$ equipped with a $\sigma$-algebra $\Sigma$. For simplicity the reader may assume that either $\Omega$ is $\RR^n$ equipped the $\sigma$-algebra of Lebesgue measurable sets, or $\Omega$ is $\{-1,1\}^n$ equipped with the discrete $\sigma$-algebra. Denote by $\M(\Omega)$ the space of probability measures on $\Omega$.
	
\begin{definition}\label{def:localization_process}
A \textbf{localization process} on $\Omega$ is a measure-valued stochastic process $(\nu_t)_{t \geq 0}$ which satisfies the following properties:
\begin{itemize}
\item[(P1)] 
Almost surely, $\nu_t$ is a probability measure on $\Omega$ for all $t$.
\item[(P2)] 
For every measurable $A \subset \Omega$, the process $t \to \nu_t(A)$ is a martingale. 
\item[(P3)] 
For any measurable $A \subset \Omega$, the process $\nu_t(A)$ almost surely converges to either $0$ or $1$ as $t \to \infty$.
\end{itemize}
\end{definition}
\begin{remark}
In the following, we deal with both \textbf{continuous-time} localization processes (namely the time parameter $t$ lives in $[0, \infty)$) and \textbf{discrete-time} localization processes (with the discrete time parameter $t=0,1,2,...$). In order to keep the notation simple, we always assume that $t$ is real-valued, but in the latter case we simply assume that the process is constant in intervals of the form $[k,k+1)$ for $k \in \ZZ$. 
\end{remark}	
	
\begin{definition}
A \textbf{localization scheme} on $\Omega$ is a mapping that assigns to each probability measure $\nu \in \M(\Omega)$ a localization process $(\nu_t)_{t \geq 0}$ which satisfies $\nu_0 = \nu$.
\end{definition}
If $\mathcal{L}$ is a localization scheme and $(\nu_t)_{t\geq 0} = \mathcal{L}(\nu)$, then we say that \emph{$(\nu_t)_{t\geq 0}$ is the localization process associated with $\nu$ via the localization scheme $\mathcal{L}$}.
	
\subsection{Simple examples of localization schemes}
To better understand the definition of localization schemes, we give a few simple examples.
\subsubsection{Example 1: Coordinate-by-coordinate localization} \label{sec:coord}
Let $\nu$ be a probability measure on $\{-1,1\}^n$. Define a process $\nu_1,...,\nu_n$ in the following way: Let $(k_1,...,k_n)$ be a random variable drawn from the uniform distribution over all permutations of $[n]$. Let $X$ be a random point sampled from $\nu$, independent of $(k_1,...,k_n)$. For $t \geq 0$, define $\nu_t$ to be the law of $X$ conditioned on $X_{k_1},...,X_{k_i}$ for $i = \min(\lfloor t \rfloor, n)$.

Note that, by definition, the sequence $(\nu_t)_{t \geq 0}$ is a martingale, the measures $\nu_t$ are probability measures and almost surely $\nu_t$ is a Dirac measure for all $t \geq n$. Thus, $(\nu_t)_t$ is a localization process. The coordinate-by-coordinate localization scheme is the scheme that assigns the process $(\nu_t)_{t \geq 0}$ to the measure $\nu$.
	
\subsubsection{Example 2: Random-subspace localization} \label{sec:subspace}
Let $\nu$ be a probability measure on $\RR^n$ and let $(u_1,...,u_n)$ be a random orthogonal basis in $\RR^n$, distributed according to the Haar measure on the orthogonal group $\mathrm{O}(n)$. Sample $Z$ a random point from $\nu$ independently of $(u_1,...,u_n)$. For all $t \geq 0$, let $\nu_t$ be the law of $Z$ conditioned on $\langle Z, u_1 \rangle, \dots, \langle Z, u_i \rangle$ for $i = \min(\lfloor t \rfloor, n)$. Then $\nu_t$ is a Doob martingale of probability measures and $\nu_n$ is a Dirac point measure supported on $Z$. Thus, $(\nu_t)_{t\geq 0}$ is a localization process.

\subsubsection{Example 3: Halfspace-bisection localization} \label{sec:halfspace}
Let $\nu$ be an absolutely continuous probability measure on $\RR^n$. Let $\theta_1,\theta_2,\dots$ be a sequence of independent points generated from the Haar measure on $\Sph$ and let $\epsilon_1,\epsilon_2,\ldots$ be a sequence of independent $\pm 1$ Bernoulli random variables. Set $\nu_0 = \nu$ and for all $i \in \mathbb{Z}$, define $\nu_i$ inductively as follows. Set
$$
t_i = \min \left \{s \in \RR; ~~ \nu_{i-1} ( \{x; ~ x \cdot \theta_i \leq s\}  ) = 1/2 \right \},
$$
define $H_i = \{x \in \RR^n; ~ (x \cdot \theta_i - t_i) \epsilon_i \geq 0 \}$. For all measurable $A \subset \RR^n$, set $\nu_{i}(A) := 2\nu_{i-1}(A \cap H)$, the normalized restriction of $\nu_{i-1}$ to $H$. In words, $\nu_{i}$ is generated by choosing a uniformly random direction in the sphere and picking one of the two half-spaces that bisect $\nu_{i-1}$ into two parts of equal mass in that direction. Define $\nu_t = \nu_{\min(\lfloor t \rfloor, n)}$. It is not hard to check that this process is a martingale and thus (P1) and (P2) are satisfied. It can be shown that (P3) is also satisfies, and so the process $(\nu_t)_t$ is a localization process, but we omit the proof since this process is only given for the sake of the example and is not used in the rest of the paper. 

\subsubsection{Example 4: Gaussian channel localization} \label{sec:gaussloc}
Consider the case $\Omega = \RR^n$. Let $\nu$ be a probability measure on $\Omega$. We construct the process $(\nu_t)_t$ as follows. Let $X \sim \nu$ and let $B_t$ be a standard Brownian motion on $\RR^n$ independent of $X$. Define $\nu_t$ to be the law of $X$ conditioned on $t X + B_t$. This process clearly satisfies (P1) and (P2). Under mild conditions we have that $\frac{tX+B_t}{t} = X + \frac{B_t}{t}$ converges to $X$ and that $(\nu_t)_t$ is a localization process. More details about the Gaussian channel localization are provided in section \ref{sec:lintilt}.

\subsubsection{Example 5: Subset simplicial-complex localization} \label{sec:subsetloc}
Let $k$ be an integer less than or equal to $n$. Let $\Omega$ be the set of all subsets of size of $k$ of $[n]$, and let $\nu$ be a probability measure on $\Omega$. We construct a sequence of measures as follows: Let $X \sim \nu$ and let $j_1,...j_k$ be a random permutation of $[k]$, independent of $X$. Write $X=\{x_1,..,x_k\}$, where $x_1 < \cdots < x_k$. For $i \in \braces{1, \ldots, k}$, define $\nu_i$ to be the law of $X$ conditioned on $X_{j_1},\dots,X_{j_i}$. Define $\nu_t = \nu_{\lfloor t \rfloor \wedge k}$.
	
\subsection{Dynamics associated with a localization scheme}
Given a localization scheme, there is a natural way to introduce a sampling algorithm associated with it. Specifically, starting from a localization process associated with $\nu$ with a localization scheme, we can define a Markov chain whose stationary measure is $\nu$.

\begin{definition} \label{def:mc1} (\textbf{Markov chain associated with a localization process}) Let $(\nu_t)_t$ be a localization process on $\Omega$ and let $\tau > 0$. The dynamics associated with $\nu_t$ and $\tau$ is the Markov chain whose transition kernel is defined by the formula
\begin{equation}\label{eq:MC}
P_{x \to A} = \EE \left [\frac{\nu_{\tau}(x) \nu_{\tau}(A)}{ \nu(x)} \right ], ~~ \forall x \in \Omega, ~~ A \subset \Omega.
\end{equation}
\end{definition}
\begin{remark}
The above definition only makes sense under the condition that $\nu_t$ is almost surely absolutely-continuous with respect to $\nu$, in which case the quantity $\frac{\nu_t(x)}{\nu(x)}$ is well-defined for $\nu$-almost every $x$. Below, we will see how to remedy this issue if this is not the case.
\end{remark}

\begin{fact}
The operator defined in \eqref{eq:MC} is the transition kernel of a reversible Markov chain whose stationary measure is $\nu$.
\end{fact}
\begin{proof}
By property (P2) of Definition~\ref{def:localization_process}, we have $\EE[\nu_{\tau}] = \nu$, which implies that for $\nu$-almost every $x \in \Omega$, 
$$
\EE \left [\frac{\nu_{\tau}(x) \nu_{\tau}(\Omega)}{ \nu(x)} \right ] = \EE \left [\frac{\nu_{\tau}(x)}{ \nu(x)} \right ] = 1,
$$ 
so that $P_{x \to \cdot}$ is indeed a probability measure. It is evident from the definition that for all $A,B \subset \Omega$,
\begin{align*}
\int_A P_{x \to B} d \nu(x) & = \int_A \EE \left [ \frac{d \nu_\tau(x)}{d \nu(x)} \nu_\tau(B) \right ] d \nu(x) \\
& = \EE \left [ \nu_\tau(A) \nu_\tau(B) \right ] = \int_B P_{y \to A} d \nu(y),
\end{align*}
hence the Markov chain is reversible and has stationary measure $\nu$.
\end{proof}
\begin{remark}
Instead of taking the time $\tau$ to be deterministic, we can take it to be a \emph{stopping time}. It is not hard to verify that we still get a reversible Markov chain. 
\end{remark}

Another interpretation of the Markov chain associated to localization process is the following: consider a random variable $(X,Y) \in \Omega \times \Omega$ defined by
\begin{equation} \label{eq:XYalt}
\PP(X \in A, ~ Y \in B) = \EE\left [ \nu_\tau(A) \nu_\tau(B) \right ].
\end{equation}
Define a transition kernel by the formula
$$
P_{x \to A} = \PP(Y \in A ~ | X = x).
$$
If $\nu_t$ is absolutely continuous with respect to $\nu$, the conditioning is well-defined for $\nu$-almost every $x$. If that is the case, it is not hard to check that the above transition kernel coincides with the one in definition~\ref{def:mc1}. 

Definition \ref{def:mc1} introduces a Markov chain with a localization process and a parameter $\tau$. Consequently, every localization scheme gives rise to a sampling algorithm.
\begin{definition} (\textbf{Sampling algorithm associated to a localization scheme}). Given a localization scheme $\mathcal{L}$ on a space $\Omega$ and given $\tau > 0$, we define $P^{(\mathcal{L}, \tau)}$ as the mapping from $\M(\Omega)$ to the space of Markov kernels on $\Omega$ where $P^{(\mathcal{L}, \tau)}(\nu)$ is given by equation \eqref{eq:MC}.
\end{definition}

\subsubsection*{Example 1: The coordinate-by-coordinate scheme and Glauber dynamics}
Let us show that the dynamics associated with the coordinate-by-coordinate localization scheme from subsection \ref{sec:coord}, with the choice $\tau=n-1$, is the Glauber dynamics, defined as follows.

\begin{definition} (Glauber dynamics)
For a measure $\nu$ on $\{-1,1\}^n$, the Glauber dynamics is the Markov chain whose transition kernel $P^{\mathrm{GD}}(\nu)$ is
\begin{align*}
	P^{\mathrm{GD}}_{x \to y}(\nu) = \frac{1}{n} \frac{\nu(y)}{\nu(x) + \nu(y)} \mathbf{1}_{\{\|x-y\|_1 = 1 \}} + p \mathbf{1}_{\{x=y\}},
\end{align*}
with $p= \frac{1}{n} \sum_{i=1}^n \frac{\nu(x)}{\nu(x) + \nu(x \oplus \ee_i)}$ and $x\oplus \ee_i$ denotes the vector of $\braces{-1, 1}^n$ obtaining by flipping the $i$-th coordinate of $x$.
\end{definition}
	
Consider the dynamics obtained from the formula \eqref{eq:MC} with $\tau=n-1$. First, for $y = x \oplus \ee_i$, we have
\begin{align*}
&\quad \EE \left [\frac{\nu_{\tau}(x) \nu_{\tau}(y)}{ \nu(x)} \right ] \\
&=  \frac{1}{n} \sum_{j \in [n]} \EE \left [ \left . \frac{\nu_{n-1}(x) \nu_{n-1}(y)}{ \nu(x)} \right | k_n = j \right ] \\
& = \frac{ \PP(\mathrm{supp}(\nu_{n-1}) = \{x,y\} )  }{n} \EE \left [ \left . \frac{\nu_{n-1}(x) \nu_{n-1}(y)}{ \nu(x)} \right | k_n = i, \mathrm{supp}(\nu_{n-1}) = \{x,y\} \right ] \\
& = \frac{\nu(x) + \nu(y)}{n \nu(x)} \frac{\nu(x)}{\nu(x) + \nu(y)} \frac{\nu(y)}{\nu(x) + \nu(y)} \\
& = \frac{1}{n} \frac{\nu(y)}{\nu(x) + \nu(y)}.
\end{align*}
Second, if $\|x-y\|_1 > 1$ then $P^{\mathrm{GD}}_{x \to y}(\nu) = 0$, which establishes the fact that the Markov chain associated with the coordinate-by-coordinate process is Glauber dynamics.

\begin{remark} \label{rmk:lGD}
We can also consider the Markov chain obtained by choosing $\tau = n-\ell$, for some $\ell > 1$. The transition kernel of this Markov chain can be described as follows: Given that the current state is some $x \in \{-1,1\}^n$, the next state will be generated by first choosing a subset of coordinates $A \subset [n]$ with $|A|=\ell$ uniformly at random, and then choosing a point according to the restriction of $\nu$ to the subcube defined by $\{y; ~ y_i = x_i, ~~ \forall i \in [n] \setminus A \}$. We refer to this Markov chain as the \textbf{$\ell$-Glauber-dynamics}.
\end{remark} 

\subsubsection*{Example 2: The subspace localization scheme and the hit-and-run walk} \label{sec:hitandrun}
Fix a measure $\nu$ on $\RR^n$ and let $(\nu_t)_t$ be the  subspace-localization process defined in Subsection \ref{sec:subspace}. Define $\tau=n-1$. Let $(X,Y)$ obey equation \eqref{eq:XYalt}. Given $x \in \Omega$, consider the expression 
\begin{align*}
P_{x \to A} & = \PP[Y \in A | X = x] \\
&  = \EE_{u_1,...,u_{n-1}, Z}[\PP(Y \in A | X = x, u_1,...,u_{n-1}, Z)].
\end{align*}
Note that, by definition, conditioned on $u_1,..,u_{n-1},Z$ we have that $X,Y$ have the law of the restriction of $\nu$ to $\{x; ~ \langle x-Z, u_i \rangle = 0, ~~ \forall i \in [n-1]\}$. Therefore, we can understand the conditioning on $X=x$ as conditioning on the event $E_x = \{\langle Z-x, u_i \rangle = 0, ~~ \forall i \in [n-1] \}$. In other words 
\begin{align*}
P_{x \to A} & = \EE_{u_1,...,u_n} \left [\PP\left (Z \in A | \langle Z-x, u_i \rangle = 0, ~ \forall i \in [n-1] \right ) \right ] \\
& = \EE_{u_n} \left [\PP\left (Z \in A | \mathrm{Proj}_{u_n^\perp }(Z-x) = 0 \right ) \right ].
\end{align*}
In words, given the point $x$, to generate the next point in the Markov chain, we first generate the direction $u_n$ uniformly from the Haar measure on the unit sphere and then generate a point from the restriction of $\nu$ to the fiber $\mathrm{Proj}_{u_n^\perp} x + \mathrm{Span}(u_n)$. This Markov chain is known as the \textbf{hit-and-run} walk (see e.g, \cite{HitandRun99}).
\subsubsection*{Other examples of Markov chains}
\begin{itemize}
\item 
In Subsection~\ref{sec:intro_RGO}, we show that the Gaussian-channel localization gives rise to the sampling algorithm introduced in \cite{RGO20}, where each step is a \textbf{restricted Gaussian oracle} step.
\item
The dynamics associated to the subset simplicial-complex localization in Subsection \ref{sec:subsetloc} with time $\tau = k-\ell$ leads to the $\ell$-\textbf{down-up walk} (see \cite{ALO-SI}).
\item
A continuous analogue of the coordinate-by-coordinate localization scheme gives rise to the \textbf{coordinate hit-and-run} algorithm (see \cite{CHNR}). 
\item
The dynamics associated with the negative-fields localization constructed in Subsection \ref{sec:NF} is the \textbf{field dynamics} introduced in \cite{CFYZ21-rapid}.
\end{itemize}

\subsection{Doob localization schemes}
In this subsection, we describe a class of localization schemes which generalizes several examples we have seen so far. Given a set $\Omega$ equipped with a $\sigma$-algebra $\Sigma$, we say that a filtration $\parenth{\FF_t}_{t \geq 0}$ of $\Omega$ is precise if $\sigma \left ( \bigcup_{t \geq 0} \FF_t \right ) = \Sigma$. Let $\mathbf{F}_\Omega$ be the space of probability measures over the set of precise filtrations on $\Omega$. 

Given a distribution $m \in \mathbf{F}_\Omega$, we can introduce a localization scheme using the filtration generated from $m$. For any measure $\nu$ on $\Omega$, we construct a localization process $(\nu_t)_t$ as follows: Let $Y$ be a $\nu$-distributed random variable and let $\FF_t$ be a precise filtration generated from $m$ independently of $Y$. $\nu_t$ is defined as follows
\begin{align*}
	\nu_t(A) = \EE \brackets{ \PP(Y \in A | \FF_t)}, ~~ \forall A \subset \Omega \mbox{ measurable }.
\end{align*}
It is evident that $\nu_t$ satisfies (P1) and (P2). Since $\FF_t$ is precise we also have that $\nu_t$ satisfies (P3). We call localization schemes constructed this way \textbf{Doob localization schemes}.

We can verify that coordinate-by-coordinate scheme and the random-subspace localization schemes are Doob localization schemes. For the coordinate-by-coordinate localization scheme, it suffices to take $\parenth{\FF_t}_{t\geq 0}$ to be $\sigma$-algebra generated by the maps $x \mapsto x_i$ for $i \in k_1,\ldots,k_{\floors{t}}$. For the random-subspace localization scheme, we take $\parenth{\FF_t}_{t\geq 0}$ to be $\sigma$-algebra generated by the maps $x \mapsto \angles{x, u_i}$ for $i \in k_1,\ldots,k_{\floors{t}}$. 


\subsection{Linear-tilt-localizations} \label{sec:lintilt}
In this subsection, we describe a family of localization schemes which arise from a process known in the literature under the name \textbf{stochastic localization}. To avoid confusion and to make it possible to include new localization schemes to this family, we refer to this family as linear-tilt localizations. The general idea behind these schemes is that, in order to obtain the measure $\nu_{t+dt}$ from the measure $\nu_t$, one multiplies the density by a random linear function. Our definitions in this section make sense whenever $\Omega$ can be naturally embedded in a linear space. However, we will mainly focus our attention on the cases $\Omega = \RR^n$ and $\Omega = \{-1,1\}^n \subset \RR^n$.

Given a measure $\nu$ on $\Omega$, denote the center of mass of $\nu$ as $\b(\nu) := \int_{\Omega} x \nu(dx)$.

\subsubsection{An alternative viewpoint on the coordinate-by-coordinate localization} \label{sec:tiltcoord}
We begin with an alternative description of the coordinate-by-coordinate localization scheme described in Subsection \ref{sec:coord}, which serves as our first example for a linear-tilt localization.

Fix a measure $\nu$ on $\Omega = \{-1,1\}^n$. Let $(k_1,...,k_n)$ be a uniform permutation of $[n]$ and let $U_1,\dots,U_n$ be an i.i.d. sequence of independent random variables drawn from the uniform distribution on $[-1,1]$.

Define $\nu_0 = \nu$. For $i=0,1,\dots,n$ define
\begin{equation}\label{eq:cbctilt}
\nu_{i+1} (x) = \nu_i(x) \bigl (1 + \left \langle x - \b(\nu_i), Z_i  \right \rangle  \bigr )	
\end{equation}
where $Z_i$ is defined by
\begin{equation} \label{eq:defZt}
Z_i = \mathbf{e}_{k_i} \times \begin{cases}
\frac{1}{1+ \b(\nu_i)_{k_i}} & \mbox{if}~~ \b(\nu_i)_{k_i} \geq U_i, \\
\frac{-1}{1 - \b(\nu_i)_{k_i}} & \mbox{if}~~ \b(\nu_i)_{k_i} \leq U_i.
\end{cases}
\end{equation}
and where $\mathbf{e}_1,\dots,\mathbf{e}_n$ is the standard basis of $\RR^n$ . Finally, to extend the process beyond times $1,\dots,n$, take $\nu_t = \nu_{\lfloor t \rfloor \wedge n}$.

To see that this is indeed a localization process, first observe that to obtain $\nu_{t+1}$, we multiply $\nu_t$ by a linear function which is equal to $1$ at its center of mass, meaning that $\nu_t(\Omega) = \nu_{t+1}(\Omega)$. It is easy to check that $\EE[Z_t] = 0$, which verifies (P2). Finally, a direct calculation shows that $\nu_{i+1}$ is in fact a pinning of the $k_i$-th coordinate of $\nu_{i}$. Indeed, if $x_{k_i} = - \sign(U_i - \b(\nu_i)_{k_i})$ then $\left \langle x - \b(\nu_i), Z_i  \right \rangle = -1$ in which case the right-hand side of \eqref{eq:cbctilt} vanishes.

It follows from the above that the definition in this subsection is equivalent to the one given in subsection \ref{sec:coord}. This point of view can be used as a motivation for the family of localization schemes considered in this section: Note that it is not evident at all from this definition that the described scheme is a \emph{Doob} localization scheme. However, in this case $(\nu_t)_t$ can be thought of as a Markov chain on the space of measures, where the transition corresponds to multiplication by linear functions with random slopes whose conditional expectation is $0$. In this example, the slopes are chosen in the unique way so that it will cause the measure $\nu_{t+1}$ to be a restriction of $\nu_t$ to a subcube. 


\subsubsection{Stochastic localization driven by a Brownian motion}
Next we describe the stochastic localization process initially constructed in \cite{Eldan-SL}. Let $\nu$ be a probability measure on $\RR^n$, and let $B_t$ be a standard Brownian motion in $\RR^n$ adapted to a filtration $\FF_t$. Let $(C_t)_t$ be an $\FF_t$-measurable process which takes values in the space of $n \times n$ positive-definite matrices. We define a process $(\nu_t)_t$ via the change of measure $\frac{d \nu_t}{d \nu}(x) = F_t(x)$, where the functions $F_t(x)$ solve the following system of stochastic differential equations:
\begin{equation}\label{eq:SL}
F_0(x) = 1, ~~ d F_t(x) = F_t(x) \langle x - \b(\nu_t), C_t d B_t \rangle, ~~ \forall x \in \RR^n.
\end{equation}
The above is an infinite system of differential equations (one for each $x \in \RR^n$), but it turns out that a slightly different point of view allows us to view it as a finite system. The uniqueness and existence of the solution of the system of stochastic differential equations are established in \cite{Eldan-SL,MR4152649}. The following fact gives a sufficient condition for the process to be a localization process. It was proved in \cite{Eldan-SL, MR4152649}, but we provide a sketch for completeness.
\begin{fact} 
If, for every unit vector $\theta \in \mathbb{S}^{n-1}$, one has $\int_0^\infty |C_t \theta|^2 dt = \infty$  almost surely, then the process $(\nu_t)_t$ is a localization process.
\end{fact}
\begin{proof}
A direct calculation shows that
\begin{align*}
d \nu_t(\RR^n) &= \int_{\RR^n} d F_t(x) \nu(dx) \\
& = \int_{\RR^n} \langle x - \b(\nu_t), C_t d B_t \rangle  F_t(x) \nu(dx) \\
& = \left \langle \int x d \nu_t - \b(\nu_t), C_t d B_t \right  \rangle = 0.
\end{align*}
Moreover, by Ito's formula
$$
d \log F_t(x) = \langle x - \b(\nu_t), C_t d B_t \rangle - \frac{1}{2} \left | C_t (x-\b(\nu_t)) \right |^2 dt
$$
which shows that $F_t(x)$ is non-negative almost surely for all $x,t$. The above establishes (P1). The martingale property (P2) follows directly from \eqref{eq:SL}. Property (P3) follows from a calculation carried out using Ito's formula which shows that $d \COV(\nu_t) = - \COV(\nu_t) C_t^2 \COV(\nu_t) dt + \mbox{ martingale }$, see \cite[Section 2]{Eldan-taming}.
\end{proof}

A very useful property of the stochastic localization process is the following.
\begin{fact}
For every $t \geq 0$, we have almost surely that $\nu_t$ attains the form,
\begin{equation} \label{eq:SLform} 
\frac{\nu_t(dx)}{\nu(dx)} = \exp \left (Z_t - \frac{1}{2}  \left \langle \Sigma_t  x, x \right \rangle + \langle y_t, x \rangle \right ),
\end{equation}
where $\Sigma_t = \int_0^t C_s^2 ds$, $y_t = \int_0^t \left (C_s d B_s + C_s^2 \b(\nu_s) ds \right )$ and $Z_t$ is a normalizing constant to ensure $\nu_t$ is a probability measure.
\end{fact}
This process turns out to be useful in proving concentration inequalities in both the continuous setting of log-concave measures and the discrete setting, see \cite{chen2021almost,Eldan-SL,Eldan-taming,Eldan-Shamir,LV-KLS,EKZ}.

\subsubsection{Stochastic localization and sampling via a restricted Gaussian oracle} \label{sec:intro_RGO}
With the specific choice $C_t = \Id$, the stochastic localization scheme gives rise, via equation \eqref{eq:MC}, to a natural reversible Markov chain associated with a measure on $\RR^n$, which happens to coincide with the restricted Gaussian dynamics given in \cite{RGO20,RGO21}. 

For a measure $\nu$ on $\RR^n$, $y \in \RR^n$ and $\eta > 0$, define the measure $\zoom_{y,\eta} \nu$ by
$$
\zoom_{y,\eta} \nu(dx) := \frac{ \nu(dx) \exp\left ( - \frac{1}{2 \eta} |x-y|^2 \right ) }{ \int_{\RR^n} \exp\left ( - \frac{1}{2 \eta} |z-y|^2 \right ) \nu(dz)  }
$$

\begin{definition} \label{def:RGD}(The restricted Gaussian dynamics)
For a measure $\nu$ on $\RR^n$ and $\eta > 0$, the restricted Gaussian dynamics for $\nu$ with parameter $\eta$ is the Markov chain defined by
$$
P^{\mathrm{RGD}_\eta}(\nu)_{x \to A} = \EE_{y \sim \mathcal{N}(x, \eta \Id)} \left [\zoom_{y, \eta} \nu (A) \right ].
$$
\end{definition}

The next proposition is based on an alternative point of view on stochastic localization which was suggested by El Alaoui and Montanari \cite{AM21SL}.
\begin{proposition}
The dynamics associated to the stochastic localization process with $C_t \equiv \Id$ and time $\tau$ is the restricted Gaussian dynamics of $\nu$ with $\eta = 1/\tau$.
\end{proposition}

\begin{proof}
It follows from equation \eqref{eq:SLform} that $\nu_t = \zoom_{y_t, 1/t} \nu$, where $y_t = B_t + \int_0^t \b(\nu_s) ds$. Define $X = \lim_{t \to \infty} \frac{1}{t} y_t$. It is shown in \cite[Section 2]{Eldan-taming} that $\b(\nu_t)$ almost-surely converges and that $\lim_{t \to \infty} \b(\nu_t)$ has the law $\nu$. Therefore,
$$
X = \lim_{t \to \infty} \frac{B_t}{t} + \frac{1}{t} \int_0^t \b(\nu_s) ds = \lim_{t \to \infty} \b(\nu_t) \sim \nu.
$$
An application of \cite[Theorem 2]{AM21SL} tells us that $(X, y_t)$ have the same joint law as $(X, t X + \sqrt{t} \Gamma)$ where $\Gamma \sim \mathcal{N}(0, \Id)$ is independent of $X$. Therefore, if we generate the measure $\nu_t$ according to the process and then generate $X'$ according to $\nu_t$, then the joint law of $(X', \nu_t)$ is the same as that of $(X', \zoom_{X+\Gamma/\sqrt{t}, 1/t} \nu)$. In light of equation \eqref{eq:XYalt}, this completes the proof.
\end{proof}
The last proposition also shows that the stochastic localization process with the choice $C_t = \Id$ is identical to the Gaussian channel localization from Subsection \ref{sec:gaussloc}.

\subsubsection{The Negative-Fields localization} \label{sec:NF}
In this subsection, we introduce a localization scheme which is only relevant for the setting $\Omega = \{-1,1\}^n$. Roughly speaking, the Negative-Fields localization process associated to a measure $\nu$ on $\{-1,1\}^n$ is the unique process $(\nu_t)_t$ where $\nu_t$ is almost surely an exponential tilt of the measure by some deterministic external field $v(t)$, and a random pinning of the measure.

Before we move on to the construction of the process, let us introduce some notation. For a measure $\nu$ on $\{-1,1\}^n$ and for $v \in \RR^n$ define
$$
\tilt_v \nu(x) = \frac{e^{\langle x, v \rangle} \nu(x)}{ \int e^{\langle y, v \rangle} d \nu(y) },
$$
the measure obtained by applying external field $v$ to $\nu$ (or the exponential tilt).

For $u \in \{-1,0,1\}$ we define the $u$-pinning $\pin_u \nu$ to be the restriction of $u$ to the subcube 
$$
S_u := \left \{x \in \{-1,1\}^n; ~ x_i u_i \geq 0, ~ \forall i \in [n]  \right \}
$$
By slight abuse of notation, for $A \subset [n]$ define $\res_A \nu(x) = \pin_{\mathbf{1}_A} \nu (x)$, namely the measure $\nu$ where the coordinates in $A$ are pinned to $+1$.

We first give a slightly informal account of the construction. Given an initial measure $\nu$ on $\Omega$, the negative-fields localization process is the unique martingale of probability measures, $(\nu_t)_t$, which has the form
$$
\nu_t = \tilt_{ -t \ones} \res_{A_t} \nu, ~~ \forall t > 0,
$$
where $\ones = (1,...,1)$ and $(A_t)_t$ is an almost-surely increasing process of subsets of $[n]$. 

A different point of view on the process is to express it as a linear-tilt localization by noting that both the operation of applying an infinitesimal external field and that of pinning amount to linear tilts. Hence, we suppose that for all $t\geq 0$
\begin{equation}\label{eq:SLjump1}
\nu_{t+h}(x) = \nu_t(x) \left ( 1 + \langle x - \b(\nu_t), Z(t+h) - Z(t) \rangle \right ) + o(h), ~~~ \forall h > 0,
\end{equation}
where
$$
Z_i(t) = -t + \frac{\mathbf{1} \left \{i \in A_{t+h} \setminus A_t \right \} }{1+\b(\nu_t)_i}
$$
and the process $(A_t)_t$ of increasing sets is uniquely defined by the equation
$$
\PP(i \in A_{t+h} \setminus A_t ~~ | A_t) = \mathbf{1}_{ \{i \notin A_t\} } \bigl(1 + \b(\nu_t)_i \bigr)h + o(h), ~~~ \forall h > 0.
$$
The two last equations render that
$$
\EE[Z(t+h) | A_t] = Z(t) + o(h), ~~~ \forall t \geq 0, h > 0
$$
which, as we will see, implies that the process is a martingale.
\begin{remark}
Roughly speaking, the measure $\nu_t$ evolves by applying increasingly stronger negative external fields to all sites, and occasionally pinning sites to the value $+1$ in a way that balances the effect of the external fields resulting in a martingale. The dynamics associated with this localization is called the \textbf{field dynamics}, and was introduced in \cite{CFYZ21-rapid}. Below, will not study the field dynamics itself, but we will be able to apply this localization to obtain the same results in a way that bypass the dynamics.
\end{remark}
The next proposition gives a rigorous account of the above and establishes the existence of a slightly generalized version of the process. 
\begin{proposition} \label{prop:negativefieldsloc}
Let $\nu$ be a probability measure on $\{-1,1\}^n$ and let $v:[0, \infty) \to \RR^n$ be a differentiable curve satisfying $v(0) = 0$. There exists a (unique) stochastic process $(u(t))_t$ with $u_t \in \{-1,0,1\}^n$, such that the process $(\nu_t)_t$ defined by $\nu_t := \res_{u(t)} \tilt_{v(t)} \nu$ is a martingale. This process is uniquely defined by
$$
\EE\bigl( \bigl. u(t+h)_i - u(t)_i ~ \bigr | u(t) \bigr ) = - \mathbf{1}_{ \{ u(t)_i = 0 \} } \bigl (1 - \sign(v_i'(t)) \b(\nu_t)_i\bigr ) v'_i(t) h + o(h), ~~ \forall t \geq 0, i \in [n].
$$
Moreover, we have for all $t \geq 0$ and $x \in \{-1,1\}^n$,
\begin{equation} \label{eq:dnft}
\nu_{t+h}(x) = \nu_t(x) \left ( 1+ \sum_{i \in [n]} (x_i-\b(\nu_t)_i) \left ( h v_i'(t)  +  \frac{u(t+h)_i - u(t)_i }{1 - \sign(v_i'(t)) \b(\nu_s)_i}  \right ) \right ) + o(h).
\end{equation}
If, in addition, $\lim_{t \to \infty} |v_i(t)| = \infty$ for all $i \in [n]$, then the above process is a localization process. 
\end{proposition}
The proof of this proposition is postponed to Appendix \ref{appendix:A}.

\subsubsection{General linear-tilt localizations}
Let us now consider a natural family of localization processes which generalizes the two previous examples. Let $\nu_t$ be a localization process on $\Omega \subset \RR^n$ (in general we can consider any linear space). Suppose that for every $t \geq 0$ and $h > 0$ we can write
\begin{equation} \label{eq:lineartilt}
\frac{d \nu_{t+h}}{d \nu}(x) = 1 + \langle x - \b(\nu_t), Z_{t,h} \rangle + o_x(h),
\end{equation}
such that $\EE[Z_{t,h} | \nu_t] = 0$. In other words, the measure $\nu_{t+h}$ is generated from $\nu_t$ via multiplication by a linear function whose slope is a conditionally-centered random variable. Note that the linear functions is equal to $1$ at the center of mass of $\nu_t$, which has to be the case given that $\nu_t$ remains a probability measure.

It is not hard to see that the family of localization processes which obey equation \eqref{eq:lineartilt} contains the coordinate-by-coordinate localization, the stochastic localization and the negative-fields localization.

\subsection{Spectral gap and modified log-Sobolev inequalities}
Consider a localization process $(\nu_t)_t$ obtained from $\nu$ via some localization scheme $\mathcal{L}$, and let $\tau > 0$. Let $P=P^{(\mathcal{L},\tau)}(\nu)$ be the Markov chain obtained by formula \eqref{eq:MC}. The following result establishes functional inequalities regarding the transition kernel $P$ through the analysis of the corresponding localization process.
\begin{proposition}
If the transition kernel $P$ is the one associated to the localization process $(\nu_t)_t$ via equation \eqref{eq:MC}, then
\begin{equation}\label{eq:SG}
\SG(P) = \inf_{\varphi:\Omega \to \RR} \frac{ \EE \left [ \Var_{\nu_\tau}[\varphi] \right ] }{ \Var_\nu[\varphi] }, \text{ and }
\end{equation}
\begin{equation}\label{eq:EF}
\EF(P) \geq \inf_{f:\Omega \to [0, \infty)} \frac{ \EE \left [ \Ent_{\nu_\tau}[f] \right ] }{ \Ent_\nu[f] }.
\end{equation}
\end{proposition}
\begin{proof}
For a test function $\varphi: \Omega \to \real$, we have
\begin{align*}
\int_\Omega \varphi(x) \left (\int_{\Omega} \varphi(y) d P_x(y) \right ) d \nu(x) & = \int_\Omega \varphi(x) \left (\EE \left [ \frac{\nu_\tau(x) \left ( \int_\Omega \varphi(y) d \nu_\tau(y) \right )}{\nu(x)} \right ] \right ) d \nu(x) \\
& = \EE \left [ \int_\Omega \varphi(x) d \nu_\tau(x) \left (\int_\Omega \varphi(y) d \nu_\tau(y) \right ) \right ] \\
& = \EE \left [ \left (\int_\Omega \varphi(x) d \nu_\tau(x) \right )^2 \right ].
\end{align*}
Consequently, 
\begin{align*}
	\frac{ \int_{\Omega \times \Omega} (\varphi(x) - \varphi(y))^2 d P_x (y) d \nu(x) }{ 2 \Var_\nu[\varphi]}
	& = \frac{ \int_\Omega \varphi(x)^2 d \nu(x) - \EE \left [ \left (\int_\Omega \varphi(x) d \nu_\tau(x) \right )^2 \right ] }{ \Var_\nu [\varphi]} \\
	& = \frac{\EE \left [  \int_\Omega \varphi(x)^2 d \nu_\tau(x) - \left ( \int_\Omega \varphi(x) d \nu_\tau(x) \right )^2 \right ]}{\Var_\nu [\varphi]}\\
	& = \frac{\Exs\brackets{\Var_{\nu_\tau}[\varphi]}}{\Var_\nu[\varphi]}.
\end{align*}
Taking the infimum over all $\varphi$, we obtain \eqref{eq:SG}. For the MLSI, observe that by Jensen's inequality we have for all $x \in \Omega$ and all $f:\Omega \to \RR_+$,
$$
\EE \left [ \frac{\nu_\tau(x)}{\nu(x)} \int_\Omega f(y) d \nu_\tau(y)  \right ] \log \EE \left [ \frac{\nu_\tau(x)}{\nu(x)} \int_\Omega f(y) d \nu_\tau(y)  \right ] $$
$$
\leq \EE \left [\frac{\nu_\tau(x)}{\nu(x)} \left ( \int_\Omega f(y) d \nu_\tau(y) \log \left ( \int_\Omega f(y) d \nu_\tau(y)\right ) \right )  \right ],
$$
where we used the fact that $\EE \left [ \frac{\nu_\tau(x)}{\nu(x)} \right ] = 1$. We therefore have
\begin{align*}
&\quad \int_\Omega (P f)(x) \log ((P f)(x)) d \nu(x) \\
& = \int_\Omega \EE \left [\frac{\nu_\tau(x)}{\nu(x)} \int_\Omega f(y) d \nu_\tau(y)  \right ] \log \EE \left [ \frac{\nu_\tau(x)}{\nu(x)} \int_\Omega f(y) d \nu_\tau(y)  \right ] d \nu(x) \\
& \leq \int_\Omega \EE \left [ \frac{\nu_\tau(x)}{\nu(x)} \left ( \int_\Omega f(y) d \nu_\tau(y) \log \left ( \int_\Omega f(y) d \nu_\tau(y)\right ) \right )  \right ]  d \nu(x) \\
& = \EE \left [ \left (\int_{\Omega} \nu_\tau(dx) \right ) \int_\Omega f(y) d \nu_\tau(y) \log \left ( \int_\Omega f(y) d \nu_\tau(y)\right ) \right ] \\
& = \EE \left [ \int_\Omega f(y) d \nu_\tau(y) \log \left ( \int_\Omega f(y) d \nu_\tau(y)\right ) \right ]. 
\end{align*}
Therefore,
\begin{align*}
\frac{\EE \left [ \Ent_{\nu_t}[f] \right ]}{\Ent_\nu[f]} & = \frac{ \EE \left [ \int_\Omega f(x) \log f(x) d \nu_\tau(x) \right ] - \EE \left [ \int_\Omega f(y) d \nu_\tau(y) \log \left ( \int_\Omega f(y) d \nu_\tau(y)\right ) \right ]  }{\Ent_\nu[f]} \\
& \leq \frac{\int_\Omega f(x) \log f(x) d \nu(x)}{\Ent_\nu[f]} - \frac{\int_\Omega (P f)(x) \log (P f)(x) d \nu(x)}{\Ent_\nu[f]} \\
& = 1 - \frac{\Ent_\nu[Pf]}{  \Ent_\nu[f] }.
\end{align*}
Equation \eqref{eq:EF} follows by taking infimum over all $f:\Omega \to \real_+$.
\end{proof}

\section{Approximate conservation of variance and entropy} \label{sec:contraction}
In this section we describe the main tools used to prove mixing bounds for the dynamics associated with a localization scheme. We introduce the notions of approximate conservation of variance and entropy which, in the context of the coordinate-by-coordinate localization are related in a simple way to spectral independence and entropic independence. In the context of stochastic localization, the concept of variance decay has been applied in order to prove a spectral gap for log-concave measures.

Here, we will see that the same underlying principle is relevant to both techniques, and the difference is in the localization scheme being considered: In both proof techniques, one of the main insights has to do with a connection between the rate of variance decay and the covariance structure of a certain class of measures. Thus our framework captures the relation between covariance structures and variance / entropy decay in both settings using the same derivation.

\begin{definition} (Approximate Conservation of Variance, discrete time).
We say that a localization process $(\nu_i)_i$ satisfies $(\kappa_1,\kappa_2,...)$-variance conservation up to time $t$, if for every test function $\varphi :\Omega \to \RR$ one has
$$
\EE[ \Var_{\nu_{i}}[\varphi] ~~|~~ \nu_{i-1} ] \geq (1-\kappa_i) \Var_{\nu_{i-1}}[\varphi], ~~~ \forall 1 \leq i \leq t.
$$
\end{definition}
The motivation behind this definition is given by the fact that
\begin{align*}
\frac{\EE [\Var_{\nu_t}[\varphi]]}{ \Var_\nu[\varphi] } & =  \EE \left [ \prod_{i=1}^t \frac{ \Var_{\nu_{i}}[\varphi] }{ \Var_{\nu_{i-1}}[\varphi] } \right ] \\
& = \EE \left [\prod_{i=1}^t \EE \left [ \left . \frac{ \Var_{\nu_{i}}[\varphi] }{ \Var_{\nu_{i-1}}[\varphi] } \right | \nu_{i-1} \right ] \right ] \geq \prod_{i=1}^t (1 - \kappa_i).
\end{align*}
In light of \eqref{eq:SG}, we immediately obtain the following simple theorem which relates between variance conservation and the spectral gap of the dynamics associated with the localization process.
\begin{proposition} \label{prop:varcontr}
If the localization process $(\nu_t)_t$ satisfies $(\kappa_1,\dots,\kappa_t)$-approximate variance conservation then the dynamics given equation \eqref{eq:MC} with $\tau=t$ has a spectral gap bounded below by $\prod_{i=1}^t (1-\kappa_i)$.
\end{proposition}

\subsection{Approximate conservation of variance for linear-tilt localizations and spectral independence}
In this section we will present very simple derivation, which gives a powerful tool for proving approximate variance conservation bounds for linear-tilt localization schemes. A special case of this derivation will recover (in a very simple way) the main theorem in the spectral independence framework of \cite{ALO-SI}.

Suppose that $\Omega \subset \RR^n$. Suppose that $(\nu_t)_t$ is a localization process on $\Omega$ whose evolution is given by the equation
\begin{equation}\label{key}
\nu_{t+1}(x) = \nu_t(x) \left (1+\langle x - \b(\nu_t), Z_t \rangle \right ), ~~ \forall x \in \Omega,
\end{equation}
where $Z_t$ is a random vector satisfying $\EE[Z_t | \nu_t] = 0$. 

Our first objective is to calculate the variance decay of a test function along this process. 
\begin{claim}
For a test function $\varphi:\Omega \to \RR$, we have
\begin{equation}\label{eq:vardecaytilt}
\EE \left [\left . \Var_{\nu_{t+1}} [\varphi] \right  | \nu_t \right ] - \Var_{\nu_t}[\varphi] = - \langle v_t, C_t v_t \rangle.
\end{equation}
where
$$
v_t := \int_\Omega (x - \b(\nu_t)) \varphi(x) \nu_t(dx) ~~~ \mbox{and} ~~~ C_t := \Cov(Z_t | \nu_t).
$$
\end{claim}
\begin{proof}
Fix $\varphi:\Omega \to \RR$. We calculate
\begin{align*}
\EE \left [\left . \Var_{\nu_{t+1}} [\varphi] \right  | \nu_t \right ] & = \EE \left [ \left . \int_\Omega \varphi(x)^2 \nu_{t+1}(dx) \right  | \nu_t \right ] - \EE \left [ \left . \left (\int_\Omega \varphi(x) \nu_{t+1}(dx) \right )^2 \right  | \nu_t \right ] \\
& \stackrel{(i)}{=} \EE_{\nu_t}[\varphi^2] - \EE \left [\left . \left ( \int_\Omega \left (1+ \langle x - \b(\nu_t), Z_t \rangle \right ) \varphi(x) \nu_{t}(dx) \right )^2 \right  |\nu_t \right ] \\
& \stackrel{(ii)}{=} \EE_{\nu_t}[\varphi^2] - \EE_{\nu_t}[\varphi]^2 - \EE \left [\left . \left ( \int_\Omega \langle x - \b(\nu_t), Z_t \rangle \varphi(x) \nu_{t}(dx) \right )^2 \right  |\nu_t \right ] \\
& = \Var_{\nu_{t}} [\varphi] - \Var \left [\left .  \left \langle \int_\Omega (x - \b(\nu_t)) \varphi(x) \nu_t(dx), Z_t \right \rangle \right | \nu_t \right ]
\end{align*}
where (i) uses (P2) and (ii) uses the fact that $\EE[Z_t | \nu_t] = 0$. By definition of $C_t$ and $v_t$, the right-hand side is equal to $\Var_{\nu_{t}} [\varphi] - \langle v_t, C_t v_t \rangle$, which completes the proof.
\end{proof}

In order to establish an approximate variance conservation bound, we need to give an upper bound on the expression $\langle v_t, C_t v_t \rangle$. We now arrive at what is perhaps the heart of the argument. We use the Cauchy-Schwartz inequality, obtaining
\begin{align} 
\langle v_t, C_t v_t \rangle & = \left | \int_\Omega C_t^{1/2} (x - \b(\nu_t) ) \varphi(x) \nu_t(dx)  \right |^2  \nonumber  \\
& = \sup_{|\theta| = 1} \left ( \int_\Omega \langle C_t^{1/2} (x - \b(\nu_t)), \theta \rangle \varphi(x)  \nu_t(dx) \right )^2 \nonumber \\
& \leq \sup_{|\theta| = 1} \int_\Omega \langle C_t^{1/2} (x - \b(\nu_t)), \theta \rangle^2 \nu_t(dx) \Var_{\nu_t}[\varphi] \nonumber \\
& = \|C_t^{1/2} \Cov(\nu_t) C_t^{1/2}\|_{\OP} \Var_{\nu_t}[\varphi]. \label{eq:CSlong}
\end{align}
Combining with \eqref{eq:vardecaytilt} yields that
\begin{equation} \label{eq:CS}
\frac{\EE \left [\left . \Var_{\nu_{t+1}} [\varphi] \right  | \nu_t \right ]}{ \Var_{\nu_t}[\varphi] } \geq  1 - \|C_t^{1/2} \Cov(\nu_t) C_t^{1/2}\|_{\OP}.
\end{equation}
The derivation above is one of the key insights in the framework presented here: A simple application of the Cauchy Schwartz inequality gives rise to the role of the covariance structure in variance conservation. 

Let us now examine how to apply \eqref{eq:CS} to the case of the coordinate-by-coordinate localization. A simple calculation, using the definition \eqref{eq:defZt} gives that
\begin{align*}
(n-t) \Cov[Z_t | \nu_t]_{i,i} & = \frac{1}{(1+\b(\nu_t)_i)^2} \frac{1+\b(\nu_t)_i}{2} + \frac{1}{(1-\b(\nu_t)_i)^2} \frac{1-\b(\nu_t)_i}{2} \\
& = \frac{1}{1-\b(\nu_t)_i^2} = (\Cov(\nu_t)_{i,i})^{-1},
\end{align*}
where the last equality is a consequence of the fact that $\int_{\Omega} x_i^2 \nu_t(dx) = 1$. Denoting $D_t$ to be the diagonal matrix whose entries coincide with the diagonal entries of $\Cov(\nu_t)$, we finally have
$$
\|C_t^{1/2} \Cov(\nu_t) C_t^{1/2} \|_{\OP} = \frac{1}{(n-t)} \| D_t^{-1/2} \Cov(\nu_t) D_t^{-1/2} \|_{\OP}.
$$
Define for a measure $\tilde \nu$ on $\{-1,1\}^n$,
\begin{equation} \label{eq:defPsi}
\COR(\tilde \nu) :=  \mathrm{diag}(\COV(\tilde \nu))^{-1/2} \COV(\tilde \nu) \mathrm{diag}(\COV(\tilde \nu))^{-1/2},
\end{equation}
where $\diag(\cdot)$ is the diagonal matrix obtained by setting all the off-diagonal entries to $0$. So $\COR(\tilde \nu)$ is just the matrix of \textbf{correlations} between different coordinates of $\tilde \nu$. We will also consider the closely-related \textbf{influence matrix of the measure}, defined as
\begin{equation}\label{eq:defInfmat}
\SI(\nu)_{i,j} := \EE_{X \sim \nu }[X_i | X_j = 1] - \EE_{X \sim \nu}[X_i | X_j=-1].	
\end{equation}
The following fact relates the correlation matrix and the influence matrix.  Its proof is postponed to Appendix \ref{appendix:A}.
\begin{fact} \label{fact:inf}
We have $\SI(\nu) = \Cov(\nu) \diag(\COV(\nu))^{-1}$ and
\begin{equation}\label{eq:SICOROP}
\|\COR(\nu) \|_\OP = \rho(\SI(\nu)) \leq \|\SI(\nu)\|_\OP.
\end{equation}
\end{fact}
In light of the above fact, Equation \eqref{eq:CS} becomes
$$
\frac{\EE \left [\left . \Var_{\nu_{t+1}} [\varphi] \right  | \nu_t \right ]}{ \Var_{\nu_t}[\varphi] } \geq  1 - \frac{1}{n-t} \rho(\SI(\nu_t)).
$$

For $u \in \{-1,0,1\}^n$ define the $u$-pinning of a measure, $\pin_u \tilde \nu$ to be the restriction of $\tilde \nu$ to the sub-cube $S_u$, where
\begin{equation} \label{eq:defSu}
S_u := \bigl \{x \in \{-1,1\}^n; ~ x_i u_i \geq 0, ~ \forall i \in [n] \bigr \}.
\end{equation}
Observe that under the coordinate-by-coordinate localization, $\nu_t$ is of the form $\pin_u \nu$ for some $u=u(t)$. The condition that $\|\SI(\pin_u \nu_t) \|_{\OP}$ is bounded uniformly in $u$ is called \textbf{spectral independence}. Plugging the above into \eqref{eq:CS} and applying Proposition \ref{prop:varcontr} recovers the main theorem in the spectral independence framework \cite{ALO-SI}:
\begin{theorem} (A reformulation of \cite[Theorem 1.3]{ALO-SI})
Suppose that $\nu$ is a measure on $\{-1,1\}^n$ such that for all $u \in \{-1,0,1\}^n$,
$$
\rho(\SI(\pin_u \nu)) \leq \eta_{ |u|_1 }.
$$
Then the spectral gap of the $k$-Glauber dynamics on $\nu$ is at least $\prod_{i=0}^{n-k-1} (1- \frac{\eta_i}{n-i})$.
\end{theorem}
\begin{remark}
There are some delicate differences between the above theorem and \cite[Theorem 1.3]{ALO-SI}. In the latter, $\tilde \eta_i$ is taken to be the operator norm of the matrix $\SI(\pin_i \nu) - \Id$ rather than that of the matrix $\SI(\nu)$. Precisely, their result states that spectral gap is bounded from below by the expression $\frac{1}{n} \prod_{i=0}^{n-k-1} (1- \frac{\tilde \eta_i}{n-i-1})$ where $\tilde \eta_i = \eta_i - 1$. So the extra factor $1/n$ is traded for the fact that $\eta_i$ can be replaced by $\tilde \eta_i = \eta_i - 1$.
\end{remark}

In our proof, we completely bypassed the need to use the notion of high-dimensional expanders or the up-down walk. The inequalities relating different levels of the up-down walk were replaced by a simple application of the Cauchy-Schwartz inequality.

\subsubsection{The case of stochastic localization; bounds for the spectral gap of log-concave measures}
The same derivation (almost verbatim) as above gives a similar bound for the variance decay of the stochastic localization process \eqref{eq:SL}. Here we briefly outline the argument which can be found in \cite{eldanlecture}.

Let $\nu_t$ be the stochastic localization process associated to a measure $\nu$ on $\RR^n$, via equation \eqref{eq:SL} with some driving matrix $C_t$. Fix a test function $\varphi:\RR^n \to \RR$. The continuous analogue of \eqref{eq:vardecaytilt} takes the form
$$
d \Var_{\nu_t}[ \varphi ] = - \left | \int C_t^{1/2} (x-\b(\nu_t)) \varphi(x) d \nu_t(x)  \right | dt + \mbox{martingale}.
$$
Equation \eqref{eq:CSlong} (or, in other words, an application of Cauchy Schwartz) gives
$$
\left | \int C_t^{1/2} (x-\b(\nu_t)) \varphi(x) d \nu_t(x)  \right | \leq \|C_t^{1/2} \COV(\nu_t) C_t^{1/2} \|_{\OP} \Var_{\nu_t}[\varphi]
$$
which implies that
$$
d \Var_{\nu_t}[\varphi] \geq - \|C_t^{1/2} \COV(\nu_t) C_t^{1/2} \|_{\OP} \Var_{\nu_t}[\varphi] dt + \mbox{martingale}
$$
By integration, we get the variance decay bound
\begin{equation}\label{eq:SLVardecay}
\frac{\EE[\Var_{\nu_t}[\varphi]]}{\Var_{\nu}[\varphi]} \geq \EE \left [ e^{\int_0^t \|C_s^{1/2} \COV(\nu_s) C_s^{1/2} \|_{\OP} ds }   \right ]^{-1} 
\end{equation}
Therefore, upper bounds on the process $t \to \|\COV(\nu_t)\|_{\OP}$ correspond to approximate conservation of variance bounds. 

A classical inequality by Brascamp and Lieb allows us to relate the Dirichlet form $\EE_{\nu}[|\nabla \varphi|^2]$ to $\EE[\Var_{\nu_t}[\varphi]]$. Namely, due to \eqref{eq:SLform} we have that the measure $\nu_t$ is $\alpha_t$-strongly log-concave with $\alpha_t$ being the smallest eigenvalue of $\Sigma_t = \int_0^t C_s^2 ds$, which implies that
$$
\EE[\Var_{\nu_t}[\varphi]] \leq \frac{1}{\alpha_t} \EE_{\nu}[|\nabla \varphi|^2].
$$
Combining this with \eqref{eq:SLVardecay} gives a way to obtain lower bounds for the spectral gap of a measure $\nu$ on $\RR^n$ via the analysis of the process $t \to \|\COV(\nu_t)\|_{\OP}$. This is one of the main ideas underlying the stochastic localization technique, which in particular led to the near-solution, due to the first author (\cite{chen2021almost}), of the Kannan-Lov\'asz-Simonovits conjecture (\cite{KLS95}) and Bourgain's slicing problem.

\subsection{Approximate conservation of entropy}
It turns out that there are also natural sufficient conditions regarding the conservation of the entropy along a localization process. These conditions are closely related to the notion of \textbf{entropic independence} put forth in \cite{EI1}.

\begin{definition} (Entropy conservation, discrete time).
We say that a localization process $(\nu_i)_i$ satisfies $(\kappa_1,\kappa_2,...)$-entropy conservation up to time $t$, if for every test function $f :\Omega \to \RR_+$ one has
$$
\EE[ \Ent_{\nu_{i}}[f] ~~|~~ \nu_{i-1} ] \geq (1-\kappa_i) \Ent_{\nu_{i-1}}[f], ~~~ \forall 1 \leq i \leq t.
$$
\end{definition}

A telescopic-product argument, completely analogous to the one in the proof of Proposition \ref{prop:varcontr}, together with equation \eqref{eq:EF}, yields the following.
\begin{proposition} \label{prop:entstab1}
If the localization process $(\nu_t)_t$ satisfies a $(\kappa_1,\dots,\kappa_t)$-entropy conservation then the dynamics given equation \eqref{eq:MC} with $\tau=t$ has a MLSI coefficient bounded below by $\prod_{i=1}^t (1-\kappa_i)$.
\end{proposition}
\begin{remark}
We refer the reader to \cite[Section 1.1]{EI1} for some examples where this bound can be applied. Usually, however, the expression $\prod_{i=1}^t (1-\kappa_i)$ will give optimal bounds. Below we will see better ways to extracting the power of approximate conservation of entropy.
\end{remark}

\subsubsection{Entropic stability}
Next, we discuss the case $\Omega \subset \RR^n$ (relevant in particular to $\Omega = \RR^n$ and $\Omega = \{-1,1\}^n$). We introduce a notion which turns out to be very useful in controlling the decay of entropy along linear-tilt localizations.

For a measure $\nu$ on $\RR^n$ and $v \in \RR^n$, define the exponential tilt $\tilt_v \nu$ by
$$
\frac{d\tilt_v \nu(x)}{d \nu(x)} := \frac{e^{\langle v, x \rangle} }{ \int e^{\langle v, z \rangle} d \nu(z) }.
$$
A central definition in our framework is the following.
\begin{definition} (Entropic stability).
For $\nu$ on $\Omega \subset \RR^n$, a function $\psi:\RR^n \times \RR^n \to \RR_+$ and $\alpha > 0$, we say that $\nu$ is $\alpha$-entropically stable with respect to $\psi$ if
$$
\psi(\b(\tilt_v \nu), \b(\nu)) \leq \alpha \KL(\tilt_v \nu || \nu), ~~ \forall v \in \RR^n.
$$
\end{definition}
In order for the definition to make sense, $\psi(x,y)$ needs to vanish when $x=y$. For the sake of intuition, we may think of the case $\psi(x,y) = |x-y|^2$. Roughly speaking, entropic stability amounts to the fact that the center of mass does not move much when tilting the measure in terms of the relative entropy of the tilt. 

What follows is a very useful observation, based the \textbf{principle of maximum entropy}. This observation is essentially due to \cite{EI1}, where it was used in the context of entropic independence, which can be understood as a specific case of entropic stability.
\begin{fact}\label{fact:tilt_is_sufficient}
For every measure $\nu$ on $\Omega$ and every function $g:\RR^n \to \RR$, one has
\begin{equation}\label{eq:pme}
\inf_{\mu} \frac{\KL(\mu||\nu)}{g(\b(\mu))} = \inf_{v \in \RR^n} \frac{\KL(\tilt_v \nu ||\nu)}{g(\b(\tilt_v \nu))}.
\end{equation}
\end{fact}
\begin{proof}
By the maximum entropy principle, the minimum over measures $\mu$ of the relative entropy $\KL(\mu || \nu)$ under the linear constraint $\b(\mu) = b_0$ is attained for a measure such that $\frac{d \mu}{d \nu} \propto \exp(L(\cdot))$ for some linear function $L:\RR^n \to \RR$, or in other words the optimal $\mu$ is of the form $\tilt_v \nu$ for some $v \in \RR^n$.
\end{proof}
An immediate corollary is the following,
\begin{lemma} \label{lem:maxent}
Suppose that a measure $\nu$ is $\alpha$-entropically stable with respect to $\psi$. Then, for every measure $\mu$ which is absolutely continuous with respect to $\nu$, we have
$$
\psi(\b(\mu), \b(\nu)) \leq \alpha \KL(\mu || \nu).
$$
\end{lemma}
\begin{remark}
In view of the last lemma, we see that the notion of entropic stability is closely related to \textbf{transportation-entropy} bounds introduced by Talagrand (see \cite{Gozlan-Survey} for a survey). While entropic stability refers to the fact that the center of mass does not move under change of measure with small relative entropy, a transportation-entropy inequality alludes to the transportation distance being small under such change of measure. Thus, in a sense entropic stability bounds are a weaker version of transportation-entropy bounds.
\end{remark}

One last ingredient which will be useful to us is the following formula which can be obtained via the logarithmic Laplace transform.
\begin{lemma} (\cite[Lemma 1]{BE-Entropic}). \label{lem:llent}
Let $\nu$ be a measure on $\RR^n$ such that $\COV(\nu)$ is invertible. Denote by $K \subset \RR^n$ the interior of the convex hull of the support of $\nu$. Then there exists a unique function $v:K \to \RR^n$ such that 
$$
\b(\tilt_{v(x)}\nu) = x, ~~ \forall x \in K.
$$
Moreover, denoting $g(x) := \KL(\tilt_{v(x)} \nu || \nu )$, we have 
\begin{equation}\label{eq:denttilt}
\nabla g(\b(\nu)) = 0, ~~~ \nabla^2 g(x) = \COV(\tilt_{v(x)}\nu)^{-1}, ~~ \forall x \in K.
\end{equation}
Finally, if $L_\nu(\theta) := \log \int e^{\langle x, \theta\rangle} d \nu(x)$ is the Logarithmic Laplace transform of $\nu$, then we have that
$$
g(x) = \sup_{\theta \in \RR^n} \langle x, \theta \rangle - L_\nu(\theta)
$$
hence $g$ is the Legendre dual of $L_\nu$.
\end{lemma}
\begin{remark}
The role of the logarithmic Laplace transform in concentration inequalities on the discrete hypercube was suggested in \cite{Eldan-Shamir}. The boundedness of its hessian is closely related to the notion of fractional log-concavity of the characteristic polynomial, introduced in \cite{AAFractionally}.
\end{remark}

\subsubsection{The case of the coordinate-by-coordinate localization: Entropic independence}
As a warm-up, in this section we essentially recover the results of \cite{EI1}, deriving a natural criterion for approximate conservation of entropy for $\Omega = \{-1,1\}^n$ with the coordinate-by-coordinate localization process. The proof here boils down to ideas which are quite similar to the ones that appear in \cite{EI1}, with two crucial differences which will (arguably) make our argument cleaner: 1. We directly analyze measures on the hypercube rather than measures on $\binom{n}{k}$. 2. The notion of \emph{fractional log-concavity} is replaced by a definition which involves the logarithmic Laplace transform of the measure, which seems to arise naturally in this context, and due to which the proof involves much simpler formulas.

For $x, y \in [-1,1]^n$, define
\begin{equation}\label{eq:defh}
H(x,y) = \sum_{i=1}^n \mathbf{1}_{|y_i| < 1} \left (\frac{1+x_i}{2} \log \left ( \frac{1+x_i}{1+y_i} \right ) + \frac{1-x_i}{2} \log \left ( \frac{1-x_i}{1-y_i} \right ) \right ),
\end{equation}
It turns out that in the context of the coordinate-by-coordinate scheme, the key will be to consider $\alpha$-entropic stability with respect to the function $\psi(x,y) = H(x,y)$. With this choice, entropic stability is essentially the same is \textbf{entropic independence} introduced in \cite{EI1}, and the following proposition shows that it implies approximate conservation of entropy under the coordinate-by-coordinate localization.

\begin{proposition} \label{prop:cbcent}
Let $\nu_t$ be the coordinate-by-coordinate localization process for some measure $\nu$ on $\{-1,1\}^n$. For every $t$, if the measure $\nu_t$ is $\kappa$-entropically stable with respect to the function $\psi(x,y) = H(x,y)$ then we have the approximate conservation of entropy bound
$$
\EE \left [ \Ent_{\nu_{t+1}}[f] ~| \nu_t \right ] \geq \left (1- \frac{\kappa}{n-t}\right ) \Ent_{\nu_t}[f],
$$
for all $f:\Omega \to \RR_+$ and $t \in [n-1]$. Consequently, if $\pin_u \nu$ is $\kappa$-entropically stable with respect to the same function, for all $u \in \{-1,0,1\}^n$, then the $\ell$-Glauber dynamics $P=P^{\mathrm{GD},\ell}(\nu)$ has MLSI coefficient
$$
\EF \left (P \right ) \geq \prod_{i=0}^{n-\ell-1} \left (1- \frac{\kappa}{n-i} \right ).
$$
\end{proposition}
\begin{proof}
Fix a probability measure $\nu$ on $\Omega$ and a function $f:\Omega \to \RR_+$. Let $(\nu_t)_t$ be the coordinate-by-coordinate localization process of $\nu$. The derivation obtained in Subsection \ref{sec:tiltcoord} gives the following: Conditioned on $\nu_t$, we have
\begin{align}
\int f(x) d \nu_{t+1}(x) & \stackrel{\eqref{eq:cbctilt}}{=} \int f(x) \left (1 + \langle x - \b(\nu_t), Z \rangle \right ) d \nu_{t}(x) \nonumber \\
& = \int f(x) d \nu_t(x) \left ( 1 + \left \langle \left (\frac{\int x f(x) d\nu_t(x)}{\int f(x) \nu_t(dx)} - \b(\nu_t) \right ), Z \right \rangle \right ) \nonumber \\
& = \nu_t(f) \left ( 1 + \langle V, Z \rangle \right ) \label{eq:EI1}
\end{align}
where
$$
V := \frac{\int x f(x) \nu_t(dx)}{\int f(x) \nu_t(dx)} - \b(\nu_t)
$$
and 
$$
Z = \mathbf{e}_{k_{t+1}} \times \begin{cases}
\frac{1}{1+ \b(\nu_t)_{k_{t+1}}} & \mbox{with probability}~~ \frac{1+\b(\nu_t)_{k_{t+1}}}{2}, \\
\frac{-1}{1 - \b(\nu_t)_{k_{t+1}}} & \mbox{with probability}~~ \frac{1 - \b(\nu_t)_{k_{t+1}}}{2}.
\end{cases}
$$
We have
\begin{align*}
\EE \left [ ( 1 + \langle V, Z \rangle ) \log (1+\langle V, Z \rangle) | \nu_t \right ] = 
& \frac{1}{n-t} \sum_{i \in [n] \setminus \{k_1,\dots,k_t\}} q(t;i)
\end{align*}
where
$$
q(t;i) = \EE \left [ ( 1 + \langle V, Z \rangle ) \log (1+\langle V, Z \rangle) | \nu_t, k_{t+1}=i \right ].
$$
Now, writing $\frac{d\mu_t}{d \nu_t}(x) = \frac{f(x)}{\int f(z) d \nu_t(z)}$, we have $V = \b(\mu_t) - \b(\nu_t)$ and 
\begin{align*}
q(t;i) & =  \frac{1+b(\nu_t)_i}{2} \left (1+ \frac{V_i}{1+\b(\nu_t)_i} \right ) \log  \left (1+ \frac{V_i}{1+\b(\nu_t)_i} \right ) \\
& ~~~~+ \frac{1-b(\nu_t)_i}{2} \left (1- \frac{V_i}{1-\b(\nu_t)_i} \right ) \log  \left (1- \frac{V_i}{1-\b(\nu_t)_i} \right ) \\
& = \frac{1+\b(\mu_t)_i}{2} \log \left ( \frac{1+\b(\mu_t)_i}{1+\b(\nu_t)_i} \right ) + \frac{1-\b(\mu_t)_i}{2} \log \left ( \frac{1-\b(\mu_t)_i}{1-\b(\nu_t)_i} \right ).
\end{align*}
Combining the last three displays, we have
\begin{equation} \label{eq:EI2}
\EE \left [ ( 1 + \langle V, Z \rangle ) \log (1+\langle V, Z \rangle) | \nu_t \right ] = \frac{1}{n-t} H(\b(\mu_t), \b(\nu_t)).
\end{equation}
We finally get
$$
\EE[ \left .\Ent_{\nu_{t+1}}[f] \right  | ~ \nu_t] = \Ent_{\nu_t}[f] - \frac{1}{n-t} \nu_t(f) H(\b(\mu_t), \b(\nu_t)).
$$
The assumption that $\nu_t$ is $\kappa$-entropically stable with respect to $h$, combined with Lemma \ref{lem:maxent}, amounts to
$$
H(\b(\mu_t), \b(\nu_t)) \leq \kappa \KL(\mu || \nu) = \frac{\kappa}{\nu_t(f)} \Ent_{\nu_t}[f].
$$
Combining the last two displays gives
$$
\EE \left [ \left . \Ent_{\nu_{t+1}}[f] \right | \nu_t \right ] \geq \Ent_{\nu_t}[f] - \frac{\kappa}{n-t} \Ent_{\nu_t}[f],
$$
By definition of approximate entropy conservation, this completes the proof of the first part. The second part follows immediately by the use of Proposition \ref{prop:entstab1}.
\end{proof}

In the paper \cite{EI1}, the authors show that entropic independence is naturally related to the log-concavity of a power of the characteristic polynomial, referred to as \emph{fractional log-concavity}. Next, we show that a somewhat analogous, very simple derivation, gives that entropic stability with respect to the function $h$ is implied by the fact that exponential tilts of the measure have a bounded correlation matrix. The proof of the next lemma is found in Appendix \ref{appendix:A}
\begin{lemma} \label{lem:loglaplace}
Suppose that a measure $\nu$ on $\{-1,1\}^n$ satisfies 
\begin{equation} \label{eq:SItilts}
\| \COR (\tilt_v \nu) \|_{\OP} \leq \alpha , ~~ \forall v \in \RR^n
\end{equation}
where $\COR(\cdot)$ is defined in \eqref{eq:defPsi}. In this case, $\nu$ is $\alpha$-entropically stable with respect to $\psi(x,y) = H(x,y)$.
\end{lemma}
\begin{remark} \label{rmk:SIEI}
Fact \ref{fact:inf} shows that $\alpha$-spectral independence is equivalent to the fact that $\COV(\nu) \preceq (\alpha+1) ~ \mathrm{diag}( \COV(\nu))$. Therefore, the condition given in the lemma can be thought of as spectral independence for all tilts.
\end{remark}
\begin{remark}
The condition \eqref{eq:SItilts} is the exact condition given in \cite{Eldan-Shamir}, where it is shown that it implies nontrivial concentration of Lipchitz functions.
\end{remark}

\subsubsection{Entropic decay for stochastic localization}
Let $\nu$ be a measure on $\RR^n$ and consider the stochastic localization process $(\nu_t)_t$ defined by \eqref{eq:SL}. Fix a non-zero, measurable function $f:\Omega \to \RR_+$. For every $t>0$, define a probability measure $\mu_t$ by the equation 
$$
\frac{d \mu_t}{d \nu_t}(x) = \frac{f(x)}{\int f(x) \nu_t(dx)}.
$$
Consider the martingale
$$
M_t := \nu_t(f) = \int f(x) F_t(x) \nu(dx).
$$
Using equation \eqref{eq:SL}, we can calculate
\begin{align*}
d M_t & = \int f(x) F_t(x) \langle x - \b(\nu_t), C_t d B_t \rangle \nu(dx) \\
& = M_t \bigl \langle C_t (\b(\mu_t) - \b(\nu_t)), d B_t \bigr \rangle.
\end{align*}
Using It\^o's formula, we have
\begin{align*}
d (M_t \log M_t) & = \frac{d[M]_t}{2M_t} + \mbox{ martingale} \\
& = \frac{1}{2} M_t |C_t (\b(\mu_t) - \b(\nu_t))|^2 dt + \mbox{ martingale}.
\end{align*}
We finally obtain
\begin{align}
d \Ent_{\nu_t}[f] &= d \int f(x) \log f(x) \nu_t(dx) - d(M_t \log M_t) \nonumber \\
&= - \frac{1}{2} \nu_t(f) \left |C_t \left (\b(\mu_t) - \b(\nu_t) \right ) \right |^2 dt + \mbox{martingale}. \label{eq:SLentdecay}
\end{align}
We arrive at the following,
\begin{proposition} \label{prop:SLentdecay}
For a fixed $T>0$, suppose that, almost surely for all $t \in [0, T]$ that $\nu_t$ is $\alpha_t$-entropically stable with respect to the function $\psi(x,y) = \frac{1}{2} |C_t(x-y)|^2$. Then we have the approximate entropic conservation bound
\begin{equation}\label{eq:SLentcons}
\EE \left [ \Ent_{\nu_{T}}[f] \right ] \geq  e^{-\int_0^T \alpha_t dt } \Ent_{\nu}[f].
\end{equation}
\end{proposition}
\begin{proof}
The entropic stability assumption combined with Lemma \ref{lem:maxent} and equation \eqref{eq:SLentdecay} yields that
$$
d \Ent_{\nu_t}[f] \geq -\alpha_t \Ent_{\nu_t}[f]dt + \mbox{martingale}, ~~ \forall t \in [0,T].
$$
Consequently, by applying Ito's formula, we have that the process $t \to e^{  \int_0^t \alpha_s ds} \Ent_{\nu_t}[f]$ is a submartingale. Therefore,
$$
\EE \left [  e^{\int_0^T \alpha_t dt } \Ent_{\nu_T}[f] \right ] \geq \Ent_{\nu}[f],
$$
which completes the proof.
\end{proof}

The following lemma is useful for establishing entropic stability with respect to quadratic functions. 
\begin{lemma} \label{lem:SLentstab}
Let $\nu$ be a measure on $\RR^n$ and $C,A$ be positive-definite matrices. Suppose that for every $v \in \RR^n$ one has
\begin{equation}\label{eq:entstabtilt}
\COV(\tilt_v \nu) \preceq A,
\end{equation}
Then $\nu$ is $\alpha$-entropically stable with respect to the function $\psi(x,y) = \frac{1}{2} |C(x-y)|^2$, for $\alpha = \|CAC\|_{\OP}$.
\end{lemma}
\begin{proof}
Define $v(x)$, $g(x) := \KL(\tilt_{v(x)} \nu || \nu )$ and $K$ as in Lemma \ref{lem:llent}. We have, 
$$
\nabla^2 g(x) \stackrel{\eqref{eq:denttilt}}{=}  \COV(\tilt_{v(x)}\nu)^{-1} \stackrel{ \eqref{eq:entstabtilt} }{\succeq} A^{-1}, ~~ \forall x \in K,
$$
and $\nabla g(\b(\nu)) = 0$. Define $h(x) = \frac{1}{2} |C(x-\b(\nu))|^2$. Then $\nabla h(\b(\nu)) = 0$ and, for all $x$, $\nabla^2 h(x) = C^2$. Therefore, for all $x \in K$, we have
$$
\|CAC \|_{\OP} \nabla^2 g(x) \succeq \|CAC \|_{\OP} A^{-1} \succeq C^2 = \nabla^2 h(x).
$$
Since $g(\b(\nu)) = h(\b(\nu)) = 0$ and $\nabla g(\b(\nu)) = \nabla h(\b(\nu)) = 0$, $g$ and $h$ coincide up to first order Taylor expansion around $\b(\nu)$.  It follows that $\|CAC \|_{\OP} g(x) \geq h(x)$. Together with Fact~\ref{fact:tilt_is_sufficient}, we complete the proof.
\end{proof}

\subsubsection{Entropic decay for the negative-fields localization}
It turns out that the entropy decay for the negative-fields localization is also governed by the function $H(x,y)$ in equation \eqref{eq:defh}.
\begin{proposition} \label{prop:NFentdecay}
	Consider a localization process $\nu_t$ obtained via the negative-fields localization scheme (with the choice $v(t) = - t \ones$). Let $f:\{-1,1\}^n \to \RR$ and let $\mu_t$ be the measure obtained via the formula $\frac{d\mu_t(x)}{d \nu_t(x)} = f(x)$. Suppose that $\nu_t$ is $\alpha$-entropically stable with respect to the function $H(x,y)$ defined in equation \eqref{eq:defh}, then we have the approximate entropy conservation bound
	$$
	\EE[ \Ent_{\nu_{t+h}}[f] | \nu_t ] \geq \Ent_{\nu_t}[f] (1- 4 h \alpha) + o(h), ~~~ \forall t,h \geq 0.
	$$
\end{proposition}
The proof of the above proposition is just a calculation, in the same spirit as the derivation of entropic decay for the coordinate-by-coordinate process and stochastic localization. The proof is found in Appendix \ref{appendix:A}.

\subsection{Entropic stability via spectral independence}
Lemma \ref{lem:loglaplace} and Lemma \ref{lem:SLentstab} both rely on the logarithmic Laplace transform to show that entropic-stability can be deduced from bounds on the covariance matrix of different tilts of the measure. On the other hand, the weaker notion of spectral independence only requires a corresponding bound on the covariance matrix itself, and in order to get a spectral gap via the spectral independence framework, we need a bound on the influence matrix for all pinnings rather than all tilts (see remark \ref{rmk:SIEI}).

It turns out that requiring all pinnings of a measure $\nu$ to be spectrally-independent is sufficient for entropic stability with respect to a quadratic function. Recall the definition of $\SI(\nu)$ from equation \eqref{eq:defPsi} and recall that for all $u \in \{-1,0,1\}^n$, $\pin_u \nu$ is defined to be the restriction of $\nu$ to the subcube $S_u$ (defined in \eqref{eq:defSu}).

\begin{theorem} \label{thm:EISI}
	Let $\nu$ be a probability measure on $\{-1,1\}^n$ and let $\alpha \geq 1$. Suppose that
	\begin{align} \label{eq:condSIpinnings0}
	\vecnorm{\COR(\pin_u \nu)}{\OP} \leq \alpha, ~~~  \forall u \in \braces{-1,0,1}^n.
	\end{align}
	Then $\nu$ is $8\alpha$-entropically stable with respect to $\psi(x,y) = \frac{1}{2}|x-y|^2$. \\
	Furthermore, if for some constants $K \geq 1, C \geq 1$, for every $i \in [n]$ and for every $u, w \in \{-1,0,1\}^n$ with $\text{supp}(u) \cap \text{supp}(w) = \emptyset$ with $u_i = w_i = 0$, we have
	\begin{align}\label{eq:smallmarginals0}
		\frac{1 + \b_i(\pin_w \pin_u \nu)}{1-\b_i(\pin_w \pin_u \nu)} &\leq K \frac{1 + \b_i(\pin_u \nu)}{1- \b_i(\pin_u \nu)}, \text{ and } \\
		\label{eq:marginals_away_from_one0}
		1 - \b_i(\pin_u \nu) &\geq \frac{1}{C},
	\end{align}
	then $\nu$ is $768 \alpha K^3C$-entropically stable with respect to the function $H(x,y)$ as defined in equation \eqref{eq:defh}.
\end{theorem}
The proof is found in Section \ref{sec:CLV}. 

Recall the $\ell$-Glauber dynamics (see remark \ref{rmk:lGD} above), which is associated to the coordinate-by-coordinate localization scheme with stopping time $\tau=n-\ell$. Denote its transition kernel by $P^{\ell-\mathrm{GD}}(\nu)$. 

The above theorem will allow us to recover the extension of the spectral independence framework due to Chen, Liu and Vigoda (\cite{CLV}), and show that spectral independence for all pinnings implies an MLSI for the $\ell$-Glauber dynamics with $\ell = \Omega(n)$. This is true under the extra condition that the marginals of the measure are bounded away from $-1$ and $1$ under pinnings:

\begin{definition} (bounded marginals).
We say that a measure $\nu$ on $\RR^n$ has $b$-marginally bounded if for all $i \in [n]$ and all $u \in \{-1,0,1\}^n$ with $u_i = 0$, we have
$$
|\b(\res_u \nu)_i| \leq 1-b.
$$
\end{definition}

The following theorem essentially recovers the framework of \cite{CLV}.
\begin{theorem} \label{thm:CLV}
Let $\nu$ be a probability measure on $\{-1,1\}^n$ and let $\alpha \geq 1$. Suppose that 
\begin{equation} \label{eq:SIpins}
\rho(\SI(\pin_u \nu)) \leq \alpha, ~~~  \forall u \in \{-1,0,1\}^n.
\end{equation}
Moreover, suppose that $\nu$ is $b$-marginally bounded. Then for every $\eps \in \left (\frac{8a}{bn}, \frac{1}{2} \right )$,
$$
\EF( P^{\ell-\mathrm{GD}}) (\nu) \geq \eps^{C \alpha / b}.
$$
where $\ell = \lceil \eps n \rceil$ and $C>0$ is a universal constant.
\end{theorem}
\begin{remark}
This result also holds under a weaker condition than having bounded-marginals, namely it is enough that the measure has $(1/b)$-tame marginals, see definition \ref{def:tamemar} below.
\end{remark}

In \cite{CLV}, a clever argument is used to show that when the model has bounded degrees (see \cite{CLV} for a precise definition) and $\eps$ is a small enough constant (as a function of the maximum degree), then mixing of the $\ell$-Glauber dynamics for $\ell = \eps n$ implies mixing of the $1$-Glauber dynamics. In fact, it is even easier to show that in the case of bounded degrees, one may sample directly by using the $\ell$-Glauber dynamics, as the single step has complexity bounded by $n^{o(1)}$.

\begin{proof}[Proof of Theorem \ref{thm:CLV}]

The bounded marginals assumption implies that the condition \eqref{eq:smallmarginals0} holds true for any pinning of the measure $\nu$ with $K=2/b$. Using Theorem \ref{thm:EISI}, we have that for every $u \in \{-1,0,1\}^n$, $\pin_u \nu$ is $(48 \alpha / b)$-entropically stable with respect to $H(x,y)$ defined in \eqref{eq:defh}. An application of Proposition \ref{prop:cbcent} gives that the $\ell$-Glauber dynamics has MLSI
\begin{align*}
\EF( P^{\ell-\mathrm{GD}})(\nu) & \geq \prod_{i=0}^{n-\ell-1} \left (1-  \frac{48 \alpha b^{-1}}{n-i} \right ) \\
& = \exp \left ( \sum_{i=0}^{\lceil (1-\eps) n \rceil} \log \left ( 1 - \frac{100 \alpha b^{-1}}{n-i} \right ) \right ) \\
& \stackrel{(i)}{\geq} \exp \left ( - \sum_{i=0}^{\lceil (1-\eps) n \rceil} \frac{200 \alpha b^{-1}}{n-i} \right ) \geq \eps^{C a/b},
\end{align*}
where (i) uses the assumption that $\eps > \frac{8a}{bn}$ and the fact that $\log(1-x) \geq -2x$ for $x > -1/2$. This completes the proof.
\end{proof}

\section{Annealing via a localization scheme}
A powerful tool in our framework will be to conceptually concatenate two localization schemes, running one localization scheme up to some time $t$ and then running a second scheme where $\nu_t$ is used as the initial condition. As we will see, in many cases this tool will allow us to use the first stage as an annealing procedure which tames the measure in a way that provides a good starting point for the second stage. 

Consider two localization schemes $\mathcal{L}_i, \mathcal{L}_f$ on a space $\Omega$. Let $\nu$ be a measure on $\Omega$ and consider localization the process $ \left (\nu^{(i)}_t \right )_t$ obtained from applying $\mathcal{L}_i$ to $\nu$. Let $\TTT$ be a stopping time adapted to the filtration of this process. Now, conditional on $\nu^{(i)}_\TTT$, let $\left (\nu^{(f)}_t \right )_t$ be a localization process obtained from applying $\mathcal{L}_f$ to $\nu^{(i)}_\TTT$. We may obtain a localization process $(\nu_t)_t$ by defining
$$
\nu_t := \begin{cases}
\nu^{(i)}_t & t \leq \TTT \\
\nu^{(f)}_{t - \tau} & t \geq \TTT.
\end{cases}
$$
We denote the localization scheme mapping $\nu$ to $(\nu_t)_t$ by $\mathrm{concat}(\mathcal{L}_i,\mathcal{L}_f,\TTT)$. 

The advantage of concatenating two localization schemes is demonstrated by the following theorem.
\begin{theorem} \label{thm:annealvar} (Annealing - variance).
Let $\mathcal{L}_i, \mathcal{L}_f$ be two localization schemes on $\Omega$ such that the latter is a Doob localization scheme. Let $\nu$ be a probability measure to which we assign a localization process $(\nu_t)_t$ from $\mathrm{concat}(\mathcal{L}_i,\mathcal{L}_f,\TTT)$ for a stopping time $\TTT$. Fix $\tau>0$ and let $P=P^{(\mathcal{L}_f, \tau)}(\nu_\TTT)$ be the (random) transition kernel associated to $\nu_\TTT$ via $(\mathcal{L}_f, T)$ using \eqref{eq:MC}. Suppose that,
\begin{enumerate}
\item
We have the approximate variance conservation bound
$$
\frac{\EE \left [ \Var_{\nu_\TTT}[\varphi] \right ]}{\Var_\nu[\varphi]} \geq \eps, ~~ \forall \varphi:\Omega \to \RR.
$$
\item
Almost-surely, we have
$$
\SG(P(\nu_\TTT)) \geq \delta.
$$
\end{enumerate}
Under these two assumptions, we have that the Markov chain associated to $\nu$ via $\mathcal{L}_f$ has spectral gap at least $\eps \delta$.
\end{theorem}
In the above theorem, the localization scheme $\mathcal{L}_i$ is thought of as an annealing procedure which takes a (potentially not well-behaved) measure $\nu$, and outputs a measure $\nu_\TTT$ which is well-behaved in the sense of assumption 2. Assumption 1 tells us that we did not lose much of the variance throughout the annealing process. The proof relies on the fact that the Dirichlet form associated by a Doob localization is a supermartingale - this fact was crucially used in \cite{EKZ} in the context of stochastic localization.

Next, we have a completely analogous theorem for the entropy.
\begin{theorem} \label{thm:annealent} (Annealing - entropy).
Let $\mathcal{L}_i, \mathcal{L}_f$ be two localization schemes on $\Omega$ such that the latter is a Doob localization scheme. Let $\nu$ be a probability measure to which we assign a localization process $(\nu_t)_t$ from $\mathrm{concat}(\mathcal{L}_i,\mathcal{L}_f,\TTT)$ for a stopping time $\TTT$. Fix $\tau>0$ and let $P=P^{(\mathcal{L}_f, \tau)}(\nu_\TTT)$ be the (random) transition kernel associated to $\nu_\TTT$ via $(\mathcal{L}_f, T)$ using \eqref{eq:MC}. Suppose that,
\begin{enumerate}
\item
We have the entropy conservation inequality
$$
\frac{\EE \left [ \Ent_{\nu_\TTT}[f] \right ]}{\Ent_\nu[f]} \geq \eps, ~~ \forall f:\Omega \to \RR_+.
$$
\item
Almost-surely, one has
\begin{equation}\label{eq:efanneal}
\EF(P(\nu_\TTT)) \geq \delta.
\end{equation}
\end{enumerate}
Under these two assumptions, we have that the Markov chain associated to $\nu$ via $\mathcal{L}_f$ has an MLSI coefficient of at least $\eps \delta$.
\end{theorem}	

Let $P^{(t)} = P^{(\mathcal{L}_f, \tau)}(\nu_t)$. Our main step towards proving the two theorems is the following.
\begin{proposition} \label{prop:submartingale}
For every $\varphi:\Omega \to \RR$ and $f:\Omega \to \RR_+$, the processes
$$
t \to \int_{\Omega \times \Omega} \varphi(x) \varphi(y) d P^{(t)}_x(y) d \nu^{(i)}_t(x)
$$
and
$$
t \to \int_{\Omega} P^{(t)} f(x) \log P^{(t)} f(x) d \nu^{(i)}_t(x)
$$
are submartingales.
\end{proposition}
The proof of Proposition~\ref{prop:submartingale} is deferred to Appendix~\ref{appendix:proof_of_prop48}. Let us see how theorems \ref{thm:annealvar} and \ref{thm:annealent} follow from this proposition. 
\begin{proof}[Proof of Theorem \ref{thm:annealvar}]
We write
\begin{align*} 
&\quad \frac{1}{2} \int_\Omega \int_\Omega \left (\varphi(x) - \varphi(y)\right )^2 d P^{(0)}_x(y) d \nu(x) \\
 & = \int_\Omega \varphi(x)^2 d \nu(x) - \int_{\Omega} \int_\Omega \varphi(x) \varphi(y) d P^{(0)}_x(y) d \nu(x) \\
& \stackrel{(i)}{\geq} \EE \left [ \int_\Omega \varphi(x)^2 d \nu^{(i)}_\TTT(x) - \int_{\Omega} \int_\Omega \varphi(x) \varphi(y) d P^{(\TTT)}_x(y) d \nu(x) \right ] \\
& = \EE \left [ \frac{1}{2} \int_\Omega \int_\Omega \left (\varphi(x) - \varphi(y)\right )^2 d P^{(\TTT)}_x(y) d \nu^{(i)}_\TTT(x) \right ] \\
& \stackrel{(ii)}{\geq} \delta \EE \left [ \Var_{\nu^{(i)}_\TTT}[\varphi] \right ] \\
& \stackrel{(iii)}{\geq} \delta \eps \Var_{\nu}[\varphi],
\end{align*}
where (i) uses Proposition \ref{prop:submartingale} and the optional stopping theorem, (ii) uses the second assumption of the theorem and (iii) uses the first assumption of the theorem.
\end{proof}

\begin{proof}[Proof of Theorem \ref{thm:annealent}]
We have,
\begin{align*}
\Ent_\nu[f] - \Ent_\nu[P^{(0)}f] & = \int_\Omega f(x) \log f(x) d \nu(x) - \int_\Omega P^{(0)}f(x) \log P^{(0)}f(x) d \nu(x) \\
& \stackrel{(i)}{\geq} \EE \left [\int_\Omega f(x) \log f(x) d \nu^{(i)}_\TTT(x) - \int_\Omega P^{(\TTT)}f(x) \log P^{(\TTT)}f(x) d \nu^{(i)}_\TTT(x) \right ] \\
& = \EE \left [\Ent_{\nu^{(i)}_\TTT}[f] - \Ent_{\nu^{(i)}_\TTT}[P^{(\TTT)}f] \right ] \\
& \stackrel{(ii)}{\geq} \delta \EE \left [ \Ent_{\nu^{(i)}_\TTT}[f] \right ] \\
& \stackrel{(iii)}{\geq} \delta \eps \Ent_{\nu}[f],
\end{align*}
where (i) uses Proposition \ref{prop:submartingale} and the optional stopping theorem, (ii) uses the second assumption of the theorem and (iii) uses the first assumption of the theorem.
\end{proof}

\section{Applications}
In this section we provide several applications to demonstrate ways to obtain mixing time results via our framework.

\subsection{Glauber dynamics for Ising models}
As a first example, we will prove a general theorem that provides a sufficient condition for MLSI, which as a special case recovers two results: the first one was shown in \cite{EKZ,EI1} and establishes mixing time for Ising models whose interaction matrix has an operator norm bounded by a constant and in particular is relevant to the \textbf{Sherrington-Kirkpatrick} model in high-enough temperatures. The second one concerns Ising models in the uniqueness regime and improves upon the main theorem of \cite{CFYZ21-treeuniqueness}.

Both applications will rely on the following result, which is obtained through a simple combination of several ingredients of our framework.
\begin{theorem} \label{thm:SK}
Let $\nu$ be a measure on $\{-1,1\}^n$ and let $J$ be a positive-definite $n \times n$ matrix. For any $0 \leq \lambda \leq 1$ and every $v \in \RR^n$, consider the probability measure $\mu_{\lambda, v}$ defined by
\begin{equation}\label{eq:defmulambda}
\frac{d\mu_{\lambda, v}}{d\nu}(x) \propto \exp \left ( - \lambda \langle x, J x \rangle + \langle v, x \rangle  \right ).	
\end{equation}
Suppose that, for some $\alpha:[0,1] \to \RR_+$, 
\begin{equation} \label{eq:SKasump1}
\|\COV(\mu_{\lambda,v})\|_{\OP} \leq \alpha(\lambda), ~~ \forall \lambda \in [0,1], ~ v \in \RR^n,
\end{equation}
and that for some $\eps > 0$,
\begin{equation} \label{eq:SKasump2}
\EF(P^{\mathrm{GD}}(\mu_{1, v})) \geq \eps, ~~ \forall v \in \RR^n,
\end{equation}
where $P^{\mathrm{GD}}(\cdot)$ denotes the transition kernel of the Glauber dynamics. Then, 
$$
\EF(P^{\mathrm{GD}}(\nu)) \geq \eps \exp \left (- 2 \|J\|_{\OP} \int_0^1 \alpha(\lambda) d \lambda \right ). 
$$
\end{theorem}
This theorem is particularly useful for measures with quadratic potentials, or \emph{Ising models}. For an $n \times n$ matrix $J$ and a vector $v \in \RR^n$, consider the probability measure defined by
\begin{equation} \label{eq:ising}
\nu_{J,v}(\{x\}) \propto \exp \left ( \langle x, J x \rangle + \langle x, v \rangle \right ),
\end{equation}
referred to as the Ising measure with interaction matrix $J$ and external field $v$. 

If we apply the above theorem to the measure $\nu = \nu_{J,v}$ (and with $J$ taken to be the interaction matrix), we have that the measure $\mu_{1,v}$ defined in \eqref{eq:defmulambda} is just a product measure, so that condition \eqref{eq:SKasump2} is trivially satisfied with $\eps = \frac{1}{n}$. In order to successfully apply the theorem, it therefore remains to verify condition \eqref{eq:SKasump1}. We will show how to do so in two different settings.

\begin{proof}[Proof of Theorem \ref{thm:SK}]
	Consider the localization process $(\nu_t)_t$ corresponding to $\nu$ obtained via the stochastic localization \eqref{eq:SL}, with the choice $C_t = (2J)^{1/2}$, up to time $t=1$. Consider the transition kernel $P^{\mathrm{GD}}(\nu_1)$, hence, the random transition kernel associated to the measure $\nu_1$ via the Glauber dynamics. We apply Theorem \ref{thm:annealent} in order to obtain a lower bound on $\EF(P^{\mathrm{GD}}(\nu))$. 
	
	According to equation \eqref{eq:SLform}, $\nu_t$ has the form
	$$
	\nu_t(x) \propto \exp \left ( - t  \langle x, J x \rangle + \langle y_t, x \rangle \right ) \nu(x)
	$$
	for some stochastic process $y_t$, meaning that $\nu_t = \mu_{t, y_t}$. Since in particular, $\nu_1 = \mu_{1,y_1}$, the assumption \eqref{eq:SKasump2} then amounts to the fact that $\EF(P^{\mathrm{GD}}(\nu_1)) \geq \eps$, which fulfills the second condition of Theorem \ref{thm:annealent}.
	
	Next, the assumption \eqref{eq:SKasump1} tells us that
	$$
	\|\COV(\nu_t)\|_{\OP} \leq \alpha(t), \forall t \in [0,1].
	$$
	By applying Lemma \ref{lem:SLentstab} with the choice $A = \alpha(t) \Id$ and $C = (2J)^{1/2}$, we conclude that $\nu_t$ is $\alpha$-entropically stable with respect to $\psi(x,y) = \frac{1}{2}\abss{(2J)^{1/2} (x-y)}^2$ with $\alpha = 2 \|J\|_\OP \alpha(t)$. Invoking Proposition \ref{prop:SLentdecay} gives the approximate-entropy-conservation bound
	$$
	\EE \left [ \Ent_{\nu_1}[f] \right ] \geq e^{- 2\|J\|_\OP \int_0^1 \alpha(t) dt} \Ent_\nu[f].
	$$
	Therefore, condition 1 of Theorem \ref{thm:annealent} is satisfied with $\alpha = e^{- 2\|J\|_\OP \int_0^1 \alpha(t) dt}$, and we conclude that $\EF(P) \geq \eps e^{- 2\|J\|_\OP \int_0^1 \alpha(t) dt}$. This completes the proof.
\end{proof}

\subsubsection{Mixing under a spectral condition on the interaction matrix}
It turns out that if the interaction matrix has operator norm bounded by a certain constant, then the covariance matrix of the model is also bounded. The following lemma is based on the decomposition obtained by Bauerschmidt and Bodineau \cite{BB-SK}. For completeness, we include a proof in Appendix B.
\begin{lemma} \label{lem:SKcov}
Let $J$ be a positive-definite matrix with $\|J\|_{\OP} < \frac{1}{2}$ and let $v \in \RR^n$. One has
$$
\|\COV(\nu_{J,v})\|_{\OP} \leq \frac{1}{1-2 \|J\|_{OP}}.
$$	
\end{lemma}

Combining the lemma with Theorem \ref{thm:SK}, we obtain
\begin{corollary}
Let $J$ be a positive-definite matrix with $\|J\|_{\OP} < \frac{1}{2}$ and let $v \in \RR^n$. The Glauber dynamics for $\nu = \nu_{J,v}$ mixes in time $O(n \log n)$.
\end{corollary}
\begin{proof}
An application of Lemma \ref{lem:SKcov} gives
$$
\|\COV(\mu_{\lambda, v})\|_{\OP} \preceq \frac{1}{1-2 (1-\lambda) \|J\|_{\OP} },
$$
so the assumption \eqref{eq:SKasump1} holds with $\alpha(\lambda)$ equal to the right-hand side. Assumption \eqref{eq:SKasump2} holds with $\delta = 1/n$ since $\mu_{1,v}$ is a product measure. The conclusion of the Theorem \ref{thm:SK} tells us that
\begin{align*}
\EF(P^{\mathrm{GD}}(\nu_{J,v})) & \geq \frac{1}{n} \exp \left (- \int_0^1 \frac{2}{1/\|J\|_\OP - 2(1-\lambda)} d \lambda  \right ) \\
& = \frac{1}{n} (1-2 \|J\|_\OP). 
\end{align*}
The proof is complete via fact \ref{fact:mixing}.
\end{proof}

\begin{remark}
The above corollary implies an optimal mixing bound for Glauber dynamics on the Sherrington-Kirkpatrick model up to inverse temperature $\beta = 1/4$ (see~\cite{EKZ} for a detailed discussion on the Sherrington-Kirkpatrick model). Very recently, El Alaoui, Montanari and Sellke \cite{SK-SL} gave a polynomial time sampling algorithm which is valid all the way to the critical temperature $\beta = 1/2$. The idea behind their algorithm is to simulate the stochastic localization process. To do so, one needs to estimate the center of mass of tilted measures which, in turn, is done using approximate message passing.
\end{remark}

\subsubsection{Near-critical ferromagnetic Ising models}
Next, we show how to recover a variant of the result which appears in the very recent paper of Bauerschmidt and Dagallier \cite{bauerschmidt2022log}.
Consider a measure of the form \eqref{eq:ising} which is \emph{ferromagnetic} in the sense that $J_{i,j} \geq 0$ for all $i \neq j$. For every $\lambda \in (0,1)$ define$$\chi(\lambda) := \|\COV(\nu_{\lambda J, 0})\|_\OP$$In the ferromagnetic case, following theorem, due to Ding, Song and Sun, is very helpful towards verifying assumption \eqref{eq:SKasump1} in Theorem \ref{thm:SK}.
\begin{theorem}(\cite{Ding2021ANC})
For any interaction matrix $J$ which is ferromagnetic (hence $J_{i,j} \geq 0$ for all $i \neq j$), for all $i \neq j$ and for all $v \in \RR^n$, we have
$$
\COV(\nu_{J,v})_{i,j} \leq \COV(\nu_{J,0})_{i,j}.
$$
\end{theorem}
By the above combined with the Perron-Frobenius theorem, it follows that
$$
\|\COV(\nu_{\lambda J, v})\|_\OP \leq \chi(\lambda), ~~ \forall v \in \RR^n.
$$
Therefore, invoking Theorem \ref{thm:SK} with $\mu_{\lambda, v} = \nu_{(1-\lambda) J, v}$ immediately yields the following.
\begin{corollary}
For any ferromagnetic interaction matrix $J$ which is positive definite, we have
$$
\EF(P^{\mathrm{GD}}(\nu_{J,v}) \geq \frac{1}{n} \exp \left ( - 2\|J\|_{\OP} \int_0^1 \chi(\lambda) d \lambda \right ).
$$
\end{corollary}
As noted in~\cite{bauerschmidt2022log}, the positive definiteness
can always be imposed by adding a diagonal matrix without changing the corresponding Ising model.

\subsubsection{Graphical Ising model in the uniqueness regime}
For a graph $G=(V,E)$ with $V = [n]$ and $\beta \in \RR$, the corresponding Ising model is defined as
\begin{equation}\label{eq:defIsing2}
\nu_{G,\beta}(x) \propto \exp \left (\langle x, v \rangle + \beta \sum_{(i,j) \in E} \mathbf{1}_{\{x_i=x_j\}} \right ).
\end{equation}
We say that the measure $\nu_{G,\beta}$ satisfies the tree-uniqueness condition if 
\begin{equation}\label{eq:treeuniqueness2}
\exp(|\beta|) < \frac{\Delta(G) - \delta }{\Delta(G) -2 + \delta}, 
\end{equation}
for $\delta \in (0, 1)$, where $\Delta(G) \geq 3$ is the maximum degree of the graph. The following result was proven in \cite{CLV} (see also \cite[Lemma 8.3]{CFYZ21-treeuniqueness}).
\begin{theorem} \label{thm:SIIsing}(Spectral independence for the Ising model).
Let $\nu = \nu_{G, \beta}$ be an Ising model defined as in \eqref{eq:defIsing2} which satisfies the tree-uniqueness condition \eqref{eq:treeuniqueness2}. Then for any external field $v \in \RR^n$, 
$$
\rho(\SI(\tilt_v \nu)) \leq \frac{2}{\delta}.
$$
\end{theorem}

Given a graph $G=(V,E)$, define a matrix $J=J_G$ by
$$
J_{i,j} = \frac{1}{2} \left (\mathbf{1}_{\{(i,j) \in E) \}} + \mathbf{1}_{\{i=j\}} \mathrm{deg}_G(i) \right ).
$$
It is straightforward to check that
$$
\sum_{(i,j) \in E} \mathbf{1}_{\{x_i=x_j\}} = \langle x, J_G x \rangle,
$$
and that $J_G$ is positive-definite.
\begin{corollary}
Any graphical Ising model $\nu$ of the form \eqref{eq:defIsing2} which satisfies \eqref{eq:treeuniqueness2}, one has
$$
\EF(P^{\mathrm{GD}}(\nu) ) \geq \frac{\exp(-8/\delta)}{n}.
$$
Consequently, the mixing time for Glauber dynamics is $O(n \log n)$.
\end{corollary}
\begin{proof}
First, we deal with the case $\beta > 0$. Apply Theorem \ref{thm:SK} with the choice $J=\beta J_G$. Observe that for all $\lambda \in [0,1]$, the measure $\mu_{\lambda,v}$ defined in \eqref{eq:defmulambda} is of the form \eqref{eq:defIsing2}, with a smaller value of $\beta$, and therefore $\mu_{\lambda,v}$ is also in the uniqueness regime. An application of Theorem \ref{thm:SIIsing} ensures that $\rho(\SI( \mu_{\lambda,v} )) \leq \frac{2}{\delta}$. By Fact \ref{fact:inf} we have
$$
\|\COV(\mu_{\lambda,v})\|_\OP \leq \|\COR(\mu_{\lambda,v})\|_\OP \leq \frac{2}{\delta},
$$
so \eqref{eq:SKasump1} holds with $\alpha(\lambda) = 2/\delta$. Moreover, since $\mu_{1,v}$ is a product measure, condition \eqref{eq:SKasump2} holds with $\eps = 1/n$. Therefore, we can conclude by Theorem \ref{thm:SK} that 
$$
\EF(P^{\mathrm{GD}}(\nu) ) \geq \frac{1}{n} \exp\left ( - 4\frac{ \beta \|J_G\|_\OP}{\delta} \right ) \geq \frac{1}{n} e^{-8/\delta}.
$$
The last step follows because $J_G = \Id - \frac{1}{2} L^{\text{sym}}$ where $L^{\text{sym}}$ is the symmetrically normalized Laplacian which has eigenvalues between $0$ and $2$; and $\beta < \log \parenth{\frac{\Delta(G) - \delta }{\Delta(G) -2 + \delta}} \leq 2$. \\
Second, the case $\beta = 0$ is trivial because it becomes a product measure. Finally, for the case $\beta < 0$, we observe that
\begin{align*}
	- \abss{\beta} \angles{x, J_G x} = \abss{\beta} \parenth{\angles{x , (\Id - J_G) x} - \angles{x, x}} = = \abss{\beta} \parenth{\angles{x , (\Id - J_G) x} -n}.
\end{align*}
$\Id - J_G$ is positive-definite and has operator norm bounded by $1$. Using the same argument as the case $\beta > 0$ but with $J_G$ replaced by $\Id - J_G$, we obtain the same MLSI lower bound. 
\end{proof}

\subsection{Sampling from strongly log-concave measures via a Restricted Gaussian Oracle}
Let $\nu$ be a log-concave measure on $\RR^n$, hence, a measure of the form
$$
\frac{\nu(dx)}{dx} = \exp(-V(x)), ~~ V:\RR^n \to \RR \mbox{ convex}.
$$
We assume that it is uniformly (or strongly) log-concave, namely that
\begin{equation} \label{eq:stronglc}
\nabla^2 V(x) \succeq \mu \Id, ~~ \forall x \in \RR^n.
\end{equation}
Let $(\nu_t)_t$ be the process attained from $\nu$ via the stochastic localization scheme \eqref{eq:SL} with the choice $C_t \equiv \Id$. According to equation \eqref{eq:SLform}, $\nu_t$ has the form
\begin{equation} \label{eq:formslc}
\frac{d \nu_t}{dx} \propto \exp \left (-\tilde V(x) - \frac{t}{2} |x|^2 \right )
\end{equation}
where $\tilde V(x)$ is convex (it is an exponential tilt of $V(x)$). The following well-known bound goes back to Brascamp and Lieb.
\begin{theorem} \label{thm:stronglc}
If $\rho$ is a measure on $\RR^n$ of the form $\frac{\rho(dx)}{dx} \propto \exp(-U(x))$ that satisfies the uniform convexity condition 
$$
\nabla^2 U(x) \succeq \alpha \Id, ~~~ \forall x \in \RR^n,
$$ 
then $\|\COV(\rho)\|_{\OP} \leq \frac{1}{\alpha}$.
\end{theorem}
Equations \eqref{eq:stronglc} and \eqref{eq:formslc} imply that $\nu_t$ has the form $\frac{d \nu_t(x)}{dx} = \exp(-U(x))$ with 
$$
\nabla^2 U(x) \succeq (\mu + t) \Id, ~~ \forall x \in \RR^n.
$$
Therefore, the above theorem gives 
\begin{equation}
\|\COV(\tilt_v \nu_t) \|_{\OP} \leq \frac{1}{\mu+t}, ~~~ \forall v \in \RR^n.
\end{equation}
An application of Lemma \ref{lem:SLentstab} gives that $\nu_t$ is $\alpha$-entropically-stable with respect to $\psi(x,y) = \frac{1}{2} |x-y|^2$, with $\alpha = \frac{1}{\mu+t}$. Proposition \ref{prop:SLentdecay} now gives the approximate conservation of entropy bound
$$
\EE \left [ \Ent_{\nu_T}[f] \right ] \geq \Ent_{\nu}[f] \exp \left (- \int_0^T \frac{1}{\mu+t} dt \right ) = \frac{\mu}{\mu+T} \Ent_{\nu}[f],
$$
for an arbitrary $f:\RR^n \to \RR^+$ such that $\int f d \nu > 0$. Using equation \eqref{eq:EF}, this gives
$$
\EF(P^{\mathrm{RGD}_\eta}(\nu)) \geq \frac{\mu}{\mu+1/\eta},
$$
where $P^{\mathbf{RGD}_\eta}(\nu)$ is the transition kernel of the restricted Gaussian dynamics (Definition \ref{def:RGD} above). We have proved:
\begin{theorem}
If $\nu$ is a $\mu$-strongly log-concave measure on $\RR^n$, then the associated restricted Gaussian dynamics with parameter $\eta$ satisfies $\EF(P^{\mathrm{RGD}_\eta}(\nu)) \geq \frac{\mu}{\mu+1/\eta}$.
\end{theorem}

The above theorem not only recovers Theorem 1.1 in~\cite{RGO21}, but also resolves the open problem discussed in section 1.4 of the same paper: it provides a proof to show that the Restricted Gaussian Dynamics mixes in $\eta \mu \log(n)$ steps when started from any $O(\exp(n))$-warm start in the case $\eta < 1/\mu$. As defined in~\cite{RGO21}, an initial measure $\rho$ is a $\beta$-warm start with respect to $\nu$ if $\frac{d\rho}{d\nu} (x) \leq \beta$ everywhere. The applications of the above theorem includes improving dependency of condition number dependency of Markov chains for log-concave sampling and introducing new sampling algorithms for composite log-concave distributions and log-concave finite
sums, which has been extensively discussed in~\cite{RGO21}.

\subsection{A spectral-independence based condition for fast mixing and applications to sampling from the hardcore model} \label{sec:harcore}
In this section we present the main application of the negative-fields localization, giving a sufficient condition for fast mixing. In particular, it provides an optimal mixing time bound for sampling the hardcore model via Glauber dynamics.

Recall that, for $u \in \{-1,0,1\}$, we define the $u$-pinning $\pin_u \nu$ to be the restriction of $u$ to the subcube $S_u := \left \{x \in \{-1,1\}^n; ~ x_i u_i \geq 0, ~ \forall i \in [n]  \right \}$ and that $\tilt_v \nu$ is the $v$-exponential tilt of $\nu$. 

The following definition relates the marginals of the pinned distribution to those of the original distribution and lower bounds the marginals in a one-sided fashion.
\begin{definition} \label{def:tamemar}
For $K > 1$, a measure $\nu$ on $\{-1,1\}^n$ has $K$-tame marginals if for every $i \in [n]$ and for every $u, w \in \{-1,0,1\}^n$ with $\text{supp}(u) \cap \text{supp}(w) = \emptyset$ with $u_i = w_i = 0$, we have
\begin{align}\label{eq:tamemar}
\frac{1 + \b_i(\pin_w \pin_u \nu)}{1 - \b_i(\pin_w \pin_u \nu)} &\leq K \frac{1+ \b_i(\pin_u \nu)}{1 - \b_i(\pin_u \nu)}, \text{ and } \\
\label{eq:tamemar_lower_bound}
1 - \b_i(\pin_u \nu) &\geq \frac{1}{K}.
\end{align}
\end{definition}

Our main aim is to establish the following sufficient condition for a MLSI: Given a measure $\nu$ on $\{-1,1\}^n$, if a MLSI can be established for a perturbed measured obtained after applying an external field $v$, and if the correlation matrix has bounded operator norm for any external field $\lambda v$, $\lambda \in [0,1]$ and under all pinnings, then the original measure satisfies a MLSI.
\begin{theorem} \label{thm:tiltmix}
Let $s > 0$. Let $\nu$ be a measure on $\{-1,1\}^n$ such that for any $\lambda \in [0, s]$, $\tilt_{-\lambda \ones} \nu$ has $K$-tame marginals. Define $\ones :=  (1,...,1)$. Suppose that for any $\lambda \in [0,s]$ and $u \in \{-1,0,1\}^n$, we have
\begin{equation}\label{eq:smallcov1}
\vecnorm{ \COR(\pin_u \tilt_{-\lambda \ones} \nu ) }{\OP} \leq \eta.	
\end{equation}
Moreover, assume that for all $u \in \{-1,0,1\}^n$, 
\begin{equation} \label{eq:asumpaftertilt}
\EF(P^{\mathrm{GD}}(\pin_u \tilt_{-s\ones} \nu)) \geq \delta.
\end{equation}
Then we have
$$
\EF(P^{\mathrm{GD}}(\nu)) \geq \delta e^{-c K^{4} \eta},
$$
for a universal constant $c>0$.
\end{theorem}
\begin{remark}
In the above theorem we considered the ``path'' of tilts $\{-\lambda \ones; \lambda \in [0,s]\}$. The same proof works for the more general case that we consider $\tilt_{v(\lambda)} \nu$ for an arbitrary curve $v:[0,1] \to \RR^n$. The special case $v(\lambda) = - \lambda \ones$ is sufficient for the applications that we know of, which is why we chose to stick to it for the sake of simplicity.
\end{remark}

\begin{proof}[Proof of Theorem \ref{thm:tiltmix}]
The main proof strategy is to use the negative-fields-localization (constructed in Section \ref{sec:NF}) as an annealing scheme for the measure. Namely, we apply Proposition \ref{prop:negativefieldsloc} with the choice $v(t) = - t \ones$ to obtain a localization process $(\nu_t)_t$. We concatenate this process to the coordinate-by-coordinate localization scheme at time $s$ and apply theorem \ref{thm:annealent}, which gives us a MLSI bound for the Glauber dynamics $P^{\mathrm{GD}}(\nu)$. 

According to Proposition \ref{prop:negativefieldsloc}, for every $t>0$ there exists $u(t) \in \{-1,0,1\}^n$ such that
$$
\nu_t = \tilt_{-t\ones} \pin_{u(t)} \nu.
$$
Therefore, condition 2 of theorem \ref{thm:annealent} follows immediately from equation~\eqref{eq:asumpaftertilt}. 

It remains to show that condition 1 of Theorem \ref{thm:annealent} is satisfied, which boils down to the approximate entropy conservation bound
\begin{equation} \label{eq:ntsconserv1}
\EE[ \Ent_{\nu_s}[f]] \geq e^{-64 K^2 e^{s} \eta} \Ent_{\nu}[f],
\end{equation}
for an arbitrary function $f:\{-1,1\}^n \to \RR_+$. 

Fix $t \in [0, s]$. We aim to apply Theorem \ref{thm:EISI} to the measure $\nu_t$, which should be understood as a measure on the sub-cube $S_{u(t)} = \{x \in \{-1,1\}^n; ~ x_i u(t)_i \geq 0, ~ \forall i \in [n]\}$. To check that condition \eqref{eq:condSIpinnings0} is fulfilled, let $\tilde u \in \{-1,0,1\}^n$ be such that $\tilde u_i u(t)_i = 0$ for all $i \in [n]$ (which amounts to all valid pinning for a function on $S_{u(t)}$), then
\begin{align*}
\vecnorm{\COR(\pin_{\tilde u} \nu_t)}{\OP} & = \vecnorm{ \COR(\pin_{\tilde u} \pin_{u(t)} \tilt_{-t \ones} \nu)}{\OP} \\
& = \vecnorm{\COR(\pin_{\tilde u + u(t)} \tilt_{-t \ones} \nu) }{\OP} \overset{ \eqref{eq:smallcov1} }{\leq} \eta.
\end{align*}
Define $F_t = \braces{i \in [n]; u_i(t) = 0}$, the set of coordinates which have not been pinned yet. 
In order to verify condition \eqref{eq:smallmarginals0} and~\eqref{eq:marginals_away_from_one0} for $\nu_t$, we use the fact that $\nu_t$ is obtained by pinning $\tilt_{-t\ones} \nu$ which has $K$-tame martingales. First, for all $i \in F_t$ and all $u(t) \perp \tilde u \in \{-1,0,1\}^n$ with $\tilde u_i = 0$, $1 - \b_i(\pin_{\tilde u} \nu_t) = 1 - \b_i(\pin_{\tilde u} \pin_{u(t)} \tilt_{-t \ones} \nu) \geq \frac{1}{K}$, which verifies condition~\eqref{eq:marginals_away_from_one0} for $\nu_t$.
Second, using the $K$-tame marginals of $\tilt_{-t\ones} \nu$, we obtain
\begin{align*}
	\frac{1 + \b_i(\pin_w \pin_{\tilde u} \nu_t)}{1 - \b_i(\pin_w \pin_{\tilde u} \nu_t)} &= \frac{1 + \b_i(\pin_w \pin_{\tilde u} \pin_{u(t)} \tilt_{-t \ones} \nu)}{1 - \b_i(\pin_w \pin_{\tilde u} \pin_{u(t)} \tilt_{-t \ones} \nu)} \\
	&\leq K \frac{1 + \b_i(\pin_w \nu_t)}{1 - \b_i(\pin_w \nu_t)},
\end{align*}
which verifies condition~\eqref{eq:smallmarginals0} for $\nu_t$.
Apply Theorem \ref{thm:EISI} to $\nu_t$, we obtain that $\nu_t$ is $c K^{4} \eta$-entropically stable with respect to the function $H(x,y)$, where $c$ is a universal constant. We can finally use Proposition \ref{prop:NFentdecay} which gives that for $t,h > 0$,
$$
\EE[ \Ent_{\nu_{t+h}}[f] | u(t)] \geq \Ent_{\nu_t}[f] (1- c K^4 \eta h ) + o(h).
$$
By integrating this inequality, we finally get
$$
\EE[ \Ent_{\nu_{s}}[f]] \geq \Ent_{\nu}[f] \exp \left ( -c K^4 \eta \right ).
$$
Both conditions of Theorem \ref{thm:annealent} are satisfied, and we conclude the proof.
\end{proof}

\subsubsection{Application to an optimal mixing bound for Glauber dynamics in the hardcore model}
Given a graph $G = (V, E)$, we define the set of all independent assignments as
\begin{align*}
	\mathcal{I}_G := \braces{x \in \{-1,1\}^V: ~ x_i + x_j \leq 1, ~ \forall (i,j) \in E }.
\end{align*}
If we associate a point $x \in \{-1,1\}^V$ the set $A(x) = \{v \in V; x_v = 1\}$, then $A(\mathcal{I}_G)$ is the family of independent sets in $G$.

For $\lambda \in (0, \infty)$ let $\rho_\lambda$ be the product probability measure on $\{-1,1\}^V$ defined by
$$
\rho_\lambda(\{x\}) = \prod_{v \in V} \frac{\mathbf{1}_{\{ x_v = -1 \}} + \lambda \mathbf{1}_{\{ x_v = 1 \}}  }{1+\lambda}.
$$
The \textit{hardcore} model on $G$ with \textit{fugacity} $\lambda$ is a probability measure on $\{-1,1\}^V$ defined by
$$
\nu_{G,\lambda}(\{x\}) := \frac{\rho_\lambda(\{x\}) \mathbf{1}_{\{x \in \mathcal{I}_G \}}}{\rho_\lambda(\mathcal{I}_G)}.
$$

A hardcore model with fugacity $\lambda$ is called \textbf{$\delta$-unique} if $\lambda \leq (1-\delta) \lambda_\Delta$, where the critical fugacity is
\begin{align}
  \label{eq:def_critical_fugacity}
  \lambda_\Delta = \frac{(\Delta-1)^{\Delta-1}}{(\Delta-2)^\Delta},
\end{align}
and where $\Delta$ is the maximal degree of $G$.

Our goal is to prove the following theorem, which asserts that the Glauber dynamics mixes in time $O(n \log n)$ in the uniqueness regime.
\begin{theorem}
 \label{thm:main_hardcore_GD_mixing}
Given a graph $G = (V, E)$. $\abss{V} = n$ and $\lambda > 0$, let $\nu = \nu_{G,\lambda}$ be the hardcore model on $G$ with fugacity $\lambda$. Suppose that $\nu$ is $\delta$-unique for some $\delta > 0$. Let $\mu_0$ be an arbitrary initial distribution supported on $\mathcal{I}_G$. Then there exists a universal constant $c > 0$, such that
  \begin{align*}
    t_{\mathrm{mix}}(P^{\mathrm{GD}}(\nu), \eps; \mu_0) \leq \exp\parenth{\frac{c}{\delta}} \parenth{n\log(n) + 3 n \log(1/\eps)}.
  \end{align*}
\end{theorem}

Proving mixing bounds for the hardcore model has been one of the central applications of the spectral and entropic independence frameworks. This model has inspired the first paper which put forth the notion of spectral independence as well as many of the following works. Let us summarize the progress made so far. 
\begin{itemize}
\item
Anari, Liu and Oveis Gharan were the first to introduce the notion of spectral independence in \cite{ALO-SI}, and proved an $n^{\exp(O(1/\delta))}$ mixing time.
 \item Chen, Liu and Vigoda~\cite{CLV} proved tight bounds on spectral independence, gave an improved $n^{O(1/\delta)}$ mixing time.
\item Chen, Liu and Vigoda~\cite{CLV2} extended the framework towards MLSI bounds and proved that $\Delta^{O(\Delta^2/\delta)} n \log(n)$, which in particular gives an optimal bound for constant $\Delta$.
\item Jain, Pham and Vuong~\cite{EI2} proved a mixing time of $\Delta^{O(1/\delta)} n^2$.
\item Chen, Feng, Yin and Zhang~\cite{CFYZ21-rapid}, in Theorem 1.3, obtained an optimal bound for the relaxation time, which implies a mixing time of $C(\delta) n^2 \log(\Delta)$.
\item Anari et al.~\cite[Theorem 1]{EI2}, showed that a variant of Glauber dynamics called the balanced Glauber dynamics gives an $O_\delta(n\log(n))$ mixing time.
\end{itemize}
Our result, Theorem \ref{thm:main_hardcore_GD_mixing} is the first to give an optimal mixing bound for the (usual) Glauber dynamics.

The proof combines many of ingredients developed in our framework, together with the idea of using a certain ``restricted'' type of entropy decay.

\subsubsection{Properties of the hardcore distribution}
In this subsection, we collect four properties of the hardcore distribution: spectral independence for all negative tilts and all pinnings, marginal upper and lower bounds, the fact that exponentially tilt hardcore model remains a hardcore model and MLSI coefficient for hardcore model with small fugacity. Throughout this section we fix graph $G=(V,E)$ with maximum degree $\Delta$, a fugacity $\lambda \in (0, \infty)$ and define $\nu = \nu_{G,\lambda}$ to be the corresponding hardcore model. Without loss of generality, we identify $V$ with $[n]$ for simplicity.

The first lemma establishes spectral independence of all negative tilts and all pinnings of the hardcore distribution, and was proved in Lemma 8.4 of~\cite{CFYZ21-rapid} extending the analyses in~\cite{CLV2}.
\begin{lemma}
  \label{lem:all_pinning_spectral_independence_of_hardcore}
Consider the hardcore model $\nu=\nu_{G, \lambda}$ with $\Delta = \Delta(G) \geq 3$. Furthermore, suppose that $\nu$ is $\delta$-unique on $G$. Then for all $u \in \braces{-1, 0, 1}^n$ and for every vector $v \in (- \infty, 0]^{V}$, we have
\begin{align*}
\vecnorm{\COR(\tilt_v \pin_u \nu)}{\mathrm{OP}} \leq \frac{144}{\delta}.
\end{align*}
\end{lemma}

The next lemma gives bounds for the marginals of the hardcore distribution. It was essentially proven in \cite[Proposition 50]{EI2}, but we provide a proof in Appendix \ref{appendix:HC} for completeness.
\begin{lemma}
\label{lem:marginal_bounds_of_hardcore}
Let $\nu = \nu_{G,\lambda}$ with $\Delta(G) \geq 3$. Then for any $v \in V$, for any $u \in \braces{-1, 0, 1}^n$ which sets all neighbors of $v$ to $0$ (i.e., satisfying $u_a = 0, \forall a \in N_v$), we have
  \begin{align*}
    \frac{\lambda}{1+\lambda} \parenth{\frac{1}{1+\lambda}}^{\abss{N_v}} \leq \Prob_{\sigma \sim \nu} \parenth{\sigma_v = +1 \mid \sigma_i u_i \geq 0, \forall i \in [n] } \leq \frac{\lambda}{1+\lambda}.
  \end{align*}
  Moreover, if the model is $\delta$-unique, namely $\lambda \leq (1-\delta) \lambda_\Delta$, then
  \begin{align*}
    \frac{\lambda}{1+\lambda} e^{-3e^2} \leq \Prob_{\sigma \sim \nu}\parenth{\sigma_v = +1 \mid \sigma_i u_i \geq 0, \forall i \in [n] } \leq \frac{\lambda}{1+\lambda}.
  \end{align*}
\end{lemma}

Recall that $\ones = (1,\dots,1)$. The next lemma shows that the exponentially tilted hardcore model remains a hardcore model. 
\begin{lemma}
\label{lem:tilt_still_hardcore}
For every $G$ and $\lambda$, we have for all $t>0$,
\begin{equation}
\tilt_{-t \ones} \nu_{G, \lambda} = \nu_{G, e^{-2t} \lambda}.
\end{equation}
In other words, the exponential tilt $\tilt_{-t \ones} \nu$ corresponds to a hardcore model, on the same graph, with fugacity $e^{-2t}\lambda$.
\end{lemma}
\begin{proof}[Proof of Lemma~\ref{lem:tilt_still_hardcore}]
  It is clear from the definition of the hardcore model.
\end{proof}

The final lemma states that if the fugacity of the model is small enough, then one has a bound on the MLSI coefficient. This lemma was proved in~\cite[Corollary 4.7]{erbar2017ricci} (see also \cite[Proposition 51]{EI2}).
\begin{lemma}
\label{lem:MLSI_for_small_fugacity}
Let $\nu$ be a hardcore model on $G$ of maximum degree at most $\Delta$ with fugacity $\lambda \leq \frac{1}{2\Delta}$. $\Delta \geq 3$. Then one has $\EF(P^{\text{GD}}(\nu)) \geq \frac{1}{4n}$.
\end{lemma}
\begin{proof}
Combine \cite[Proposition 51]{EI2} and \cite[Fact 3.5]{CLV}.
\end{proof}

\begin{proof}[Proof of Theorem \ref{thm:main_hardcore_GD_mixing}]
We simply verify that all the conditions of Theorem \ref{thm:tiltmix} hold. Lemma \ref{lem:all_pinning_spectral_independence_of_hardcore} verifies condition \eqref{eq:smallcov1} with $\eta = \frac{144}{\delta}$. Next, note that, since $\Delta > 3$,
\begin{align*}
e^{-2t} \lambda \leq e^{-2t} \frac{(\Delta-1)^{\Delta-1}}{(\Delta-2)^\Delta} = \frac{e^{-2t}}{\Delta - 2} \left (1 + \frac{1}{\Delta-2} \right )^{\Delta - 1} \leq \frac{9 e^{1-2t}}{\Delta}.
\end{align*}
According to Lemma \ref{lem:tilt_still_hardcore}, the condition of Lemma \ref{lem:MLSI_for_small_fugacity} is verified for $\tilt_{-2 \ones} \nu$, which gives that $\EF(P^{\mathrm{GD}}(\tilt_{-2 \ones} \nu)) \geq \frac{1}{4n}$ so that \eqref{eq:asumpaftertilt} holds true with $s=2$ and $\delta = 1/(4n)$. Lemma \ref{lem:marginal_bounds_of_hardcore} and Lemma \ref{lem:tilt_still_hardcore} together ensure that
\begin{align*}
	1 - \b_i(\tilt_{-t\ones} \nu) \geq \frac{2}{1+ e^{-2t} \lambda} \geq \frac{2}{1 + 4 }.
\end{align*}
Finally, Lemma \ref{lem:marginal_bounds_of_hardcore} ensures that $\tilt_{-t\ones}\nu$ has $K$-tame martingales with $K = e^{30}$, for any $t\in [0,s]$. Finally, we invoke Theorem \ref{thm:tiltmix} to obtain $\EF(P^{\mathrm{GD}}(\nu) \geq \exp \parenth{- \frac{c}{\delta} } \frac{1}{4n}$ where $c$ is a universal constant. Using Fact \ref{fact:mixing}, this completes the proof.
\end{proof}

\section{Entropic stability via spectral independence} \label{sec:CLV}
In this section, we prove Theorem \ref{thm:EISI}, which shows that entropic stability is implied by spectral independence for all pinnings. A priori, the formulation of this theorem has nothing to do with the negative-fields localization, however the main argument of its proof relies on a coupling argument based on the negative-fields localization.

For convenience, we repeat the formulation of Theorem \ref{thm:EISI} in the following theorem.
\begin{theorem} \label{thm:EISI2}
Let $\nu$ be a probability measure on $\{-1,1\}^n$ and let $\alpha \geq 1$. Suppose that
\begin{align} \label{eq:condSIpinnings}
\vecnorm{\COR(\pin_u \nu)}{\OP} \leq \alpha, ~~~  \forall u \in \braces{-1,0,1}^n.
\end{align}
Then $\nu$ is $8\alpha$-entropically stable with respect to $\psi(x,y) = \frac{1}{2}|x-y|^2$. \\
Furthermore, if for some constants $K \geq 1, C \geq 1$, for every $i \in [n]$ and for every $u, w \in \{-1,0,1\}^n$ with $\text{supp}(u) \cap \text{supp}(w) = \emptyset$ with $u_i = w_i = 0$, we have
\begin{align}\label{eq:smallmarginals}
	\frac{1 + \b_i(\pin_w \pin_u \nu)}{1-\b_i(\pin_w \pin_u \nu)} &\leq K \frac{1 + \b_i(\pin_u \nu)}{1- \b_i(\pin_u \nu)}, \text{ and } \\
	\label{eq:marginals_away_from_one}
	1 - \b_i(\pin_u \nu) &\geq \frac{1}{C},
\end{align}
then $\nu$ is $768 \alpha K^3C$-entropically stable with respect to the function $H(x,y)$ as defined in equation \eqref{eq:defh}.
\end{theorem}

The proof of the theorem consists of the following two main lemmas.
\begin{lemma}  \label{prop:SItoEI}
Suppose that $\nu$ satisfies the assumption \eqref{eq:condSIpinnings}. Then for all $v \in \RR^n$,
$$
\left  | \b(\tilt_v  \nu) - \b(\nu ) \right |^2 \leq 16 \alpha^2 |v|^2.
$$
If, furthermore, condition \eqref{eq:smallmarginals} and~\eqref{eq:marginals_away_from_one} hold then for all $v \in \RR^n$,
$$
\langle v, \b(\tilt_v \nu) - \b(\nu) \rangle \leq 4 \alpha \sum_{i \in [n]} K^3 C \parenth{1 + \b_i(\nu)} v_i^2 \exp(4 |v_i|).
$$
\end{lemma}

\begin{lemma} \label{lem:llentdelta}
Let $\nu$ be a measure on $\RR^n$. Suppose that 
\begin{equation}\label{eq:condtilts0}
|\b(\tilt_{v'} \nu) - \b(\nu)|^2 \leq \eps^2 |v'|^2, ~~ \forall v' \in \RR^n,
\end{equation}
then 
$$
|\b(\tilt_{v} \nu) - \b(\nu)|^2 \leq \eps \KL(\tilt_{v} \nu || \nu), ~~ \forall v \in \RR^n.
$$
Furthermore, suppose that for some constant $C \geq 1$, we have
\begin{equation}\label{eq:condtilts}
\langle v, \b(\tilt_{v} \nu) - \b(\nu) \rangle \leq \sum_{i \in [n]} \eps_i v_i^2 \exp(4 |v_i|), ~~ \forall v \in \RR^n,
\end{equation}
where $2 \leq \eps_i \leq C \parenth{1+\b_i(\nu)}$ for all $i \in [n]$. Then,
$$
H(\b_i(\tilt_v \nu), \b_i(\nu)) \leq 192 C \cdot \KL(\tilt_v \nu || \nu), ~~~ \forall v \in \RR^n.
$$
\end{lemma}
Plugging the two lemmas together immediately establishes Theorem \ref{thm:EISI2}. Before we move on to the proofs of Lemmas \ref{prop:SItoEI} and \ref{lem:llentdelta}, we need three more technical intermediate lemmas whose proofs are found in Appendix \ref{appendix:A}.

\begin{lemma} \label{lem:cormar}
Let $\nu$ be a probability measure on $\{-1,1\}^n$. If $s \in \{-1,1\}$ and $\ee_i$ is a vector of the standard basis, then we have
\begin{equation}\label{eq:barpin}
\b \left (\pin_{s \mathbf{e}_i} \nu \right ) - \b(\nu) = (1+s \b(\nu)_i)^{-1} \COV(\nu) s \ee_i.
\end{equation}
\end{lemma}

\begin{lemma} \label{lem:HPhi}
	Define $H(x,y) = \frac{1+x}{2} \log \parenth{\frac{1+x}{1+y}} + \frac{1-x}{2} \log \parenth{\frac{1-x}{1-y}}$ and $\Phi(x) = (1+x) \log(1+x) - x$. Then we have 
	\begin{equation} \label{eq:hphicomp}
	\frac{1}{2} H(x,y) \leq  (1 + y) \Phi \parenth{\frac{x- y }{1 + y}} \leq 2 H(x,y),
	\end{equation}
	for all $y \in (-1,1)$ and $x \in [-1,1]$. Moreover, for all $\eps \geq 1$, 
	\begin{equation} \label{eq:hphicomp2}
	\frac{1}{4 \eps} H(x,y) \leq \eps (1 + y) \Phi \parenth{\frac{x-y}{\eps (1 + y )} }.
	\end{equation}
	Finally, we also have
	\begin{equation} \label{eq:hphicomp3}
	\Phi(|s|) \geq \frac{1}{3} \Phi(s), ~~ \forall |s| \leq 1.
	\end{equation}
\end{lemma}

\begin{lemma} \label{lem:tiltmarginals}
Let $\rho$ be a measure on $\{-1,1\}^n$ which satisfies, for all $i \in [n]$ and for all $u \in \{-1,0,1\}^n$ such that $u_i = 0$, 
\begin{equation}\label{eq:boundedmar3}
	1 + \b_i(\pin_u \rho) \leq \delta, 
\end{equation}
Then for all $v \in \RR^n$ and all $u \in \{-1,0,1\}^n$ with $u_i = 0$, we have
\begin{equation}\label{eq:boundedmar4}
	1 + \b_i(\tilt_v \pin_u \rho) \leq \delta \exp\parenth{\max\braces{0, 2v_i}}\leq \delta \exp (2 |v_i|).
\end{equation}
Moreover, if for all $i \in [n]$ and for all $u \in \{-1,0,1\}^n$ such that $u_i = 0$, 
\begin{equation}\label{eq:boundedmar3prime}
	\delta' \leq \frac{1 + \b_i(\pin_u \rho)}{1 - \b_i(\pin_u \rho)} \leq \delta'', 
\end{equation}
Then for all $v \in \RR^n$ and all $u \in \{-1,0,1\}^n$ with $u_i = 0$, we have
\begin{equation}\label{eq:boundedmar4prime}
	\delta' \exp\parenth{\min\braces{0, 2v_i}} \leq \frac{1 + \b_i(\tilt_v \pin_u \rho)}{1 - \b_i(\tilt_v \pin_u \rho)} \leq \delta'' \exp\parenth{\max\braces{0, 2v_i}}.
\end{equation}
\end{lemma}

\begin{proof}[Proof of lemma \ref{prop:SItoEI}]
Define $\mu = \tilt_v \nu$. Consider the localization process $(\mu_t)_t$ obtained via the negative-fields localization starting from $\mu$. More specifically, we invoke Proposition \ref{prop:negativefieldsloc} with the choice $v(t) = -tv$, to obtain a process $(u(t))_t$ of pinnings, $u(t) \in \{-1,0,1\}^n$, such that the process defined by
$$
\mu_t = \pin_{u(t)} \tilt_{v(t)} \mu = \pin_{u(t)} \tilt_{(1-t)v} \nu
$$
is a martingale. Moreover, we have
\begin{equation} \label{eq:jp2}
\PP[u(t+h)_i \neq u(t)_i | u(t)] =  h (1+s_i \b(\mu_s)_i) |v_i| + o(h),
\end{equation}
where $s_i := \sign(v_i)$.

We couple this process with a process $(\nu_t)_t$ defined by
$$
\nu_t = \pin_{u(t)} \nu.
$$
We remark that $\nu_t$ is \textbf{not} a martingale. By definition of $\nu_t$ and $\mu_t$, it is evident that $\nu_1 = \mu_1$ almost surely. Since $\mu_t$ is a martingale, we have
\begin{equation}\label{eq:bnubmu}
\b(\mu) = \EE[\b(\mu_1)] = \EE[\b(\nu_1)].
\end{equation}
Define $F_t = \{i \in [n]; ~ u(t)_i = 0\}$, the set of coordinates which have not been pinned yet. Note that $\PP\bigl( \bigl. \|u(t+h) - u(t)\|_1 \geq 2 ~~\bigr| u(t) \bigr) = o(h)$. Therefore, we have
\begin{align*}
\EE[\b(\nu_{t+h}) -\b(\nu_t) | u(t)] & = \EE[\b(\pin_{u(t+h)} \nu) -\b(\pin_{u(t)} \nu) | u(t)]  \\
& = \sum_{i \in F_t} \left (\b(\pin_{u(t)+ s_i e_i } \nu) - \b(\pin_{u(t)} \nu ) \right ) \PP(u(t+h)_i \neq u(t)_i ) + o(h) \\
& \stackrel{ \eqref{eq:jp2} }{=} \sum_{i \in F_t} \left (\b(\pin_{u(t)+ s_i e_i } \nu) - \b(\pin_{u(t)} \nu ) \right ) |v_i|h (1+s_i \b(\mu_t)_i) + o(h) \\
& \stackrel{ \eqref{eq:barpin} }{=} h \sum_{i \in F_t} \Cov(\nu_t) s_i e_i \frac{1+s_i \b(\mu_t)_i}{ 1+s_i \b(\nu_t)_i } |v_i| + o(h).
\end{align*}
Note that if $i \notin F_t$, then $\b(\mu_t)_i \b(\nu_t)_i = 1$ and we can denote by convention that $\frac{0}{0} = 1$.
Since $(1+s_i \b(\nu_t)_i)^{-1} = (\COV(\nu_t)_{i,i})^{-1} (1-s_i \b(\nu_t)_i)$, we finally get
$$
\EE[\b(\nu_{t+h}) -\b(\nu_t) | u(t)] = h \SI(\nu_t) P_t Q_t v + o(h),
$$
where $S = \mathrm{diag}(\sgn(v))$, $P_t = \Id - S \mathrm{diag}(\b(\nu_t))$ and $Q_t = \Id + S \mathrm{diag}(\b(\mu_t))$. 
Integrating with respect to time gives
\begin{equation}\label{eq:finalint}
\b(\mu) - \b(\nu) \stackrel{\eqref{eq:bnubmu}}{=}
\EE \b(\nu_1) - \b(\nu_0) = \int_0^1 \EE \left [ \SI(\nu_t) P_t Q_t v \right ] dt.
\end{equation}
Since $\nu_t$ is a pinning of $\nu$, the assumption \eqref{eq:condSIpinnings} implies that $\|\COR(\nu_t)\|_{\OP} \leq \alpha$. The triangle inequality then gives
\begin{equation}\label{eq:SIEIfinal}
\left |\b(\mu) - \b(\nu) \right | \leq \alpha \int_0^1 \EE |(\Id - S \mathrm{diag}(\b(\nu_t)))  \left (\Id + S \mathrm{diag}(\b(\mu_t)) \right ) v| dt.
\end{equation}
Since $|(\Id - S \mathrm{diag}(\b(\nu_t)))  \left (\Id + S \mathrm{diag}(\b(\mu_t)) \right ) v| \leq 4 |v|$, this proves the first part of the lemma. 

For the second part of the lemma, we write
\begin{align*}
	\langle v, \b(\mu) - \b(\nu) \rangle \overset{ \eqref{eq:finalint} }{ = } \int_0^1 \EE \left [ \langle v, \SI(\nu_t) P_t Q_t v \rangle \right ] dt.
\end{align*}
Define $E_t = \diag(\Cov(\nu_t))$. We have 
\begin{align}
	\label{eq:Et_indiv_upper_bound}
	(E_t)_{ii} &= 1 - \b_i(\nu_t)^2 \notag \\
	&= (1 - \b_i(\nu_t))(1 + \b_i(\nu_t))  \leq 2 (1 + \b_i(\nu_t)).
\end{align}
Recall that by definition $\SI(\nu_t) = \Cov(\nu_t) E_t^{-1}$ and $\COR(\nu_t) = E_t^{-1/2}\Cov(\nu_t)E_t^{-1/2}$. We have
\begin{align}
	\label{eq:vSIPQv_intermediate}
	\angles{v, \SI(\nu_t) P_t Q_t v} &=  \angles{E_t^{1/2}v, \COR(\nu_t) E_t^{-1/2} P_t Q_t v} \notag \\
	&\leq \vecnorm{\COR(\nu_t)}{\OP} \cdot \abss{E_t^{1/2} v} \cdot \abss{E_t^{-1/2} P_t Q_t v}.
\end{align}
The following diagonal matrix is entry-wise bounded. 
\begin{align*}
	(E_t^{-1/2} P_t Q_t)_{ii} &= \frac{\parenth{1-s_i \b_i(\nu_t)} \parenth{1+s_i \b_i(\mu_t)}}{\parenth{1 - \b_i^2(\nu_t)}^{1/2}} \\
	&\leq 2K  \exp\parenth{2\abss{v_i}} \sqrt{ \frac{1 + \b_i(\nu_t)}{1 - \b_i(\nu_t)} }\\
	&\leq 2K^{3/2}C^{1/2}  \exp\parenth{2\abss{v_i}} \parenth{ 1 + \b_i(\nu)}^{1/2}.
\end{align*}
The last step follows because: if $i \notin F_t$, meaning that $i$-th coordinate has been pinned, by the convention on $\frac{0}{0}$, $(E_t^{-1/2} P_t Q_t)_{ii} \leq 1$; if $s_i = -1$, then we bound $\parenth{1+s_i \b_i(\mu_t)}$ from above by 2; otherwise, applying Lemma~\ref{lem:tiltmarginals} to $\nu_t$ together with condition~\eqref{eq:smallmarginals}, we obtain
\begin{align*}
	1 +  \b(\mu_t)_i =  1 + \b_i( \tilt_{(1-t) v} \nu_t) \leq 2K \frac{1+\b_i(\nu_t)}{1 - \b_i(\nu_t)} \exp\parenth{2 \abss{v_i}}.
\end{align*}
Plugging the above bound together with~\eqref{eq:Et_indiv_upper_bound} back to~\eqref{eq:vSIPQv_intermediate}, we obtain 
\begin{align*}
	&\quad \angles{v, \SI(\nu_t) P_t Q_t v} \\
	&\leq \vecnorm{\COR(\nu_t)}{\OP} \sqrt{\sum_i 2 (1 + \b_i(\nu_t)) v_i^2} \sqrt{\sum_i 4 K^3 C  \exp(4\abss{v_i}) \parenth{1 + \b_i(\nu)} v_i^2 } \\
  &\leq \vecnorm{\COR(\nu_t)}{\OP} \parenth{\sum_i 4 K^3 C \exp(4\abss{v_i}) \parenth{1 + \b_i(\nu)} v_i^2}
\end{align*}
where the last step follows from condition~\eqref{eq:smallmarginals}.
Hence,
\begin{align*}
	\langle v, \b(\mu) - \b(\nu) \rangle \leq 4 \alpha \sum_{i \in [n]} K^3 C \parenth{1 + \b_i(\nu)} v_i^2 \exp(4 |v_i|),
\end{align*}
which completes the proof.
\end{proof}

\begin{proof}[Proof of lemma \ref{lem:llentdelta}]
Define $f(v) := \log \int_{\{-1,1\}^n} \exp(\langle v, x \rangle) d \nu(x)$ for $v \in \RR^n$, the logarithmic Laplace transform of $\nu$. Define for $x \in \RR^n$
\begin{align*}
	g(x) = \max_{v \in \RR^n} \angles{v,x} - f(v),
\end{align*}
its Legendre dual. According to Lemma \ref{lem:llent}, we have $\KL(\tilt_v \nu || \nu) = g(\nabla f(v))$ and that $\b(\tilt_v \nu) = \nabla f(v)$. With this notation, the assumption \eqref{eq:condtilts0} amounts to
\begin{equation}\label{eq:assumpf0}
|\nabla f(v') - \nabla f(0)|^2 \leq \eps^2 |v'|^2, ~~~ \forall v' \in \RR^n,
\end{equation}
and the conclusion becomes
\begin{equation}\label{eq:nts1}
g(\nabla f(v)) \geq \frac{1}{2 \eps} \abss{\nabla f(v) - \nabla f(0)}^2, \forall v \in \RR^n.
\end{equation}
Define $h(v) = f(v) - \langle \nabla f(0), v \rangle$. Equation \eqref{eq:assumpf0} implies
$$
h(v) = \int_0^1 \langle \nabla h(t v), v \rangle dt \leq |v| \int_0^1 |\nabla h(tv)| dt \leq |v| \int_0^1 \eps t |v| dt = \frac{1}{2} \eps |v|^2.
$$
Since the Legendre transform is order-reversing, we have
$$
\sup_{v \in \RR^n} \angles{x, v} - h(v)  \geq \frac{1}{2\eps} |x|^2, ~~ \forall x \in \RR^n
$$
which yields \eqref{eq:nts1} by taking $x = \nabla f(v) - \nabla f(0)$ and completes the proof for the first part of the lemma. 

For the second part, note that the assumption~\eqref{eq:condtilts} amounts to
\begin{align*}
	\angles{v, \nabla f(v) - \nabla f(0)} \leq \sum_{i \in [n]} \eps_i v_i^2 \exp(4 |v_i|), ~~ \forall v \in \RR^n.
\end{align*}
Define $h(v) = f(v) - \langle \nabla f(0), v \rangle$, then
\begin{align*}
	\angles{v, \nabla h(v)} \leq \sum_{i \in [n]} \eps_i v_i^2 \exp(4|v_i|), ~~~~ \forall v \in \RR^n.
\end{align*}
We can upper bound $h$ as follows
\begin{align*}
h(v) & = h(0) + \int_0^1 \angles{v, \nabla h(tv)} dt \\
& \leq \int_0^1 \sum_{i \in [n]} \eps_i v_i^2 t \exp(4 t |v_i|) dt \\
& = \sum_{i \in [n]} \frac{1}{16} \eps_i \brackets{\exp(4 |v_i|) (4|v_i| - 1) + 1 }\\
& \leq  \frac{1}{16} \sum_{i \in [n]} \eps_i \brackets{\exp(8 |v_i|) - 8 |v_i| - 1} \\
& \leq  \frac{1}{16} \sum_{i \in [n]} \eps_i \brackets{\exp(16 |v_i|) - 16 |v_i| - 1}
\end{align*}
The Legendre transform of the function $s \mapsto \frac{\eps}{16} \left (\exp(|16s|) - |16s| - 1 \right )$ is the function $t \mapsto \frac{\eps}{16} \brackets{\parenth{1 + \frac{|t|}{\eps}} \log\parenth{1 + \frac{|t|}{\eps}} - \frac{|t|}{\eps}}$. Since the Legendre transform is order-reversing, we get
$$
h^*(x) \geq \sum_{i \in [n]} \frac{\eps_i}{16} \left ( \left (1 + \frac{|x_i|}{\eps_i }\right ) \log \left (1 + \frac{|x_i|}{\eps_i} \right ) - \frac{|x_i|}{\eps_i} \right ).
$$
Let $y = \nabla f(0)$. Taking $\eps = \frac{\eps_i}{1+y_i} \geq 1$ in equation~\eqref{eq:hphicomp2} of Lemma~\ref{lem:HPhi}, we obtain
$$
\frac{1+y_i}{4\eps_i} H(x_i + y_i,y_i) \overset{\eqref{eq:hphicomp2}}{\leq} \eps_i \Phi \parenth{\frac{x_i}{\eps_i} } \stackrel{\eqref{eq:hphicomp3}}{\leq } 3\eps_i  \Phi \left (\frac{|x_i|}{\eps_i} \right ).
$$
Finally, using the assumption $\eps_i \leq C \parenth{1+\b_i(\nu)}$ and taking $x = \nabla f(v) - \nabla f(0)$, we get
$$
\sum_{i \in [n]} H(x_i + y_i,y_i) \leq 192 C \cdot h^*(x) = 192 C \cdot g(x+y) = 192 C \cdot \KL(\tilt_v \nu || \nu).
$$
This completes the proof.
\end{proof}

\bibliographystyle{alpha}
\bibliography{bib}

\appendix
\section{Appendix A: Loose ends} \label{appendix:A}

\begin{proof}[Proof of Lemma \ref{lem:loglaplace}]
	Let $K$ and $v(x)$ be defined as in Lemma \ref{lem:llent}. Define $z = \b(\nu)$. A direct calculation shows that
	$$
	H(z,z) = 0, ~~~\nabla_x H(x,z)|_{x=z} = 0.
	$$
	Moreover, we have 
	$$
	\KL(\tilt_{v(z)} \nu || \nu) = 0, ~~~ \nabla_x \KL(\tilt_{v(x)} \nu || \nu) = 0 |_{x=z} = 0.
	$$
	Therefore, if we establish that
	\begin{equation} \label{eq:secondderdom}
	\nabla^2_x H(x,z) \preceq \alpha \nabla^2_x \KL(\tilt_{v(x)} \nu || \nu), ~~ \forall x \in K
	\end{equation}
	Then it will follow that $H(x,z) \leq \alpha \KL(\tilt_{v(x)} \nu || \nu)$ which will complete the proof. A direct calculation yields
	$$
	\frac{\partial^2}{\partial x_i^2} H(x,z) = \frac{1}{1-x_i^2}.
	$$
	Since $\nu$ is supported on $\{-1,1\}^n$, we have $\int x^{\otimes 2} d \nu(x) = \Id$, which implies that
	$$
	\diag(\COV(\tilt_{v(x)} \nu)) = \Id - x^{\otimes 2}.
	$$
	A combination of the last two displays gives
	$$
	\nabla_x^2 H(x,z) = \diag(\COV(\tilt_{v(x)} \nu))^{-1}.
	$$
	Moreover, Formula \eqref{eq:denttilt} gives
	\begin{equation}
	\nabla_x^2 \KL(\tilt_{v(x)} \nu || \nu ) = \COV(\tilt_{v(x)}\nu)^{-1}.
	\end{equation}
	Combining the last two displays with the assumption $\mathbf{\COR}(\tilt_{v(x)}) \preceq \alpha \Id$ implies \eqref{eq:secondderdom}. The proof is complete.
\end{proof}

\begin{proof}[Proof of Fact \ref{fact:inf}]
	Let $X \sim \nu$. We have 
	\begin{align*}
		2 \COV(\nu)_{i,j} & = \EE[X_i|X_j=1](1+\EE[X_j]) - \EE[X_i|X_j=-1](1-\EE[X_j]) - 2 \EE[X_i]\EE[X_j] \\
		& = (\EE[X_i|X_j=1]-\EE[X_i|X_j=-1]) + \EE[X_j] \Bigl (\EE[X_i|X_j=1]+\EE[X_i|X_j=-1] \Bigr .\\
		& ~~~~ \Bigl. - \EE[X_i|X_j=1](1+\EE[X_j]) - \EE[X_i|X_j=-1](1-\EE[X_j]) \Bigr) \\
		& = (\EE[X_i|X_j=1]-\EE[X_i|X_j=-1]) (1- \EE[X_j]^2),
	\end{align*}
	which readily implies that $\SI(\nu) = \COV(\nu) D^{-1}$ where $D := \diag(\COV(\nu))$. To see that the operator norm is that same as that of $\COR(\nu) = D^{-1/2} \COV(\nu) D^{-1/2}$, let $v$ be an eigenvector of $\COR(\nu)$ with eigenvalue $\lambda$, write $u = D^{1/2} v$, then
	$$
	\lambda v = D^{-1/2} \COV(\nu) D^{-1/2} v \Leftrightarrow \lambda D^{-1/2} u = D^{-1/2} \COV(\nu) D^{-1} u \Leftrightarrow \lambda u = \COV(\nu) D^{-1} u
	$$
	which implies that $\rho(\SI(\nu)) = \rho(\COR(\nu))$. Since $\COR(\nu)$ is symmetric, we have $\rho(\COR(\nu)) = \|\COR(\nu)\|_{\OP}$. This completes the proof.
\end{proof}

\subsection{Supermartingality of the Dirichlet form: Proof of Proposition \ref{prop:submartingale}}
\label{appendix:proof_of_prop48}
Recall that we assume that $\mathcal{L}_f$ is a Doob localization. Let $P$ be the transition kernel of the Markov chain given by $(\mathcal{L}_f, t, \mu)$, then there exists a random variable $X \sim \mu$ and a $\sigma$-algebra $\Sigma$ on $X$ such that
\begin{align*}
	\int_{\Omega \times \Omega} \varphi(x) \varphi(y) d P_x (y) d \mu(x) & = \EE \left [\int_{\Omega \times \Omega} \varphi(x) \varphi(y) \mu_t(dx) \mu_t(dy) \right ] \\
	& = \EE_{\Sigma} \left [ \EE_{X} \left [ \EE[\varphi(X) | \Sigma]^2 \right ] \right ].
\end{align*}
Moreover,
\begin{align*}
	\int_{\Omega \times \Omega} P f(x) \log P f(x) d \mu(x) = \EE_{\Sigma} \left [ \EE_{X} \left [ \EE[f(X) | \Sigma] \log \EE[f(X) | \Sigma] \right ] \right ] - \EE_\nu[f] \log \EE_\nu[f].
\end{align*}

Fix a measure $\rho$ on $\Omega$ and suppose that $\mu$ is absolutely continuous with respect to $\rho$, hence we may write $\frac{d \mu}{d \rho}(x) = h(x)$. 
\begin{lemma}
	Fix a measure $\rho$ on $\Omega$ and $\varphi: \Omega \to \RR$ and a $\sigma$-algebra $\Sigma$ on $\Omega$. For all $x \in \Omega$ there exist linear functionals $F_x,G_x:L_1(\Omega, \rho) \to \RR$ such that the following holds: For every $h:\Omega \to \RR_+$ satisfying $\int_\Omega h d \rho = 1$,
	$$
	\left [ \EE_{X} \left [ \EE[\varphi(X) | \Sigma]^2 \right ] \right ] = \int_\Omega \frac{F_x(h)^2}{G_x(h)} d \rho(x).
	$$
	Analogously, for every non-negative $f:\Omega \to \RR$, there exist families of linear functionals $F_x, G_x$ such that
	$$
	\left [ \EE_{X} \left [ \EE[f (X) | \Sigma] \log \EE[f (X) | \Sigma] \right ] \right ] = \int_\Omega F_x(h) \log \left ( F_x(h) / G_x(h) \right )
	$$
	where $X \sim \mu$ and $\frac{d \mu}{d \rho} = h$.
\end{lemma}
\begin{proof}
	The lemma follows from a standard disintegration theorem. For the sake of demystifying the proof, let us first consider the case that $\Sigma$ is a finite partition of $\Omega$ into sets $\Omega_1,...,\Omega_k$. Let $A(x)$ be the unique set $\Omega_i$ such that $x \in \Omega_i$. Then we have
	$$
	\EE[\varphi(X) | \Sigma] = \frac{\int_{A(X)} \varphi(x) d \mu(x)}{\mu(A(X))}.
	$$
	Therefore,
	\begin{align*}
		\EE_{X \sim \mu} \left [ \EE[\varphi(X) | \Sigma]^2 \right ] & = \sum_{i=1}^k \mu(\Omega_i) \EE_{X \sim \mu} \left [ \left . \frac{\int_{A(X)} \varphi(x) d \mu(x)}{\mu(A(X))} \right  | X \in \Omega_i \right ]^2 \\
		& = \sum_{i=1}^k \frac{\left (\int_{\Omega_i} \varphi(x) d \mu(x) \right )^2}{\mu(\Omega_i)} \\
		& \sum_{i=1}^k  \frac{\left (\int_{\Omega_i} \varphi(x) h(x) d \rho(x) \right )^2}{\int_{\Omega_i} h(x) \rho(dx)}.
	\end{align*}
	This completes the proof of the first part by choosing 
	$$
	F_x(h) := \frac{1}{\rho(A(x))} \int_{A(x)} \varphi(y) h(y) d \rho(y)~~~\mbox{and}~~~G_x(h) := \frac{1}{\rho(A(x))} \int_{A(x)} h(y) d \rho(y).
	$$
	The proof of the second part is very similar. Define $\psi(s) = s \log s$. We have,
	\begin{align*}
		\EE_{X \sim \mu} \left [ \EE[f(X) | \Sigma] \log \EE[f(X) | \Sigma ] \right ] & = \sum_{i=1}^k \mu(\Omega_i) \psi \left (\EE_{X \sim \mu} \left [ \left . \frac{\int_{A(X)} f(x) d \mu(x)}{\mu(A(X))} \right  | X \in \Omega_i \right ] \right ) \\
		& = \sum_{i=1}^k \nu(\Omega_i) \psi \left (\frac{\left (\int_{\Omega_i} f(x) d \mu(x) \right )}{\mu(\Omega_i)} \right ) \\
		& \sum_{i=1}^k  \left (\int_{\Omega_i} f(x) h(x) d \rho(x) \log \left ( \frac{\int_{\Omega_i} f(x) h(x)\rho(dx)}{\int_{\Omega_i} h(x) d \rho(x)}  \right ) \right ),
	\end{align*}
	and the second part is proved by choosing
	$$
	F_x(h) := \frac{1}{\rho(A(x))} \int_{A(x)} f(y) h(y) d \rho(y)~~~\mbox{and}~~~G_x(h) := \frac{1}{\rho(A(x))} \int_{A(x)} h(y) d \rho(y).
	$$
\end{proof}

To complete the proof, observe that for every linear functional $F:L_1(\Omega, \nu) \to \RR$, we have that $t \to F(\frac{d \nu_t}{d \nu})$ is a martingale. 
\begin{fact}
	If $M_t, N_t$ are martingales and $N_t \geq 0$ almost surely, then $\frac{N_t^2}{M_t}$ is a submartingale. If $M_t$ is also positive almost-surely then $N_t \log \frac{N_t}{M_t}$ is a submartingale.
\end{fact}
\begin{proof}
	The first part follows immediately from the fact that the function $(x,y) \to x^2/y$ is convex in the domain $\{y > 0\}$. The second part follows from the convexity of $(x,y) \to x \log(x/y)$ in the domain $\{x>0,y>0\}$.
\end{proof}

\subsection{Existence and entropic decay of the negative-fields localization process}

\begin{proof}[Proof of Proposition \ref{prop:negativefieldsloc}]
For $v \in \RR^n$ and $u \in \{-1,0,1\}^n$, define
$$
a(v,u) = \b( \pin_u \tilt_v \nu ).
$$
Fix $t \geq 0$ and suppose that the process $u(t)$ has been constructed up to time $t$. We will first show that the process can be extended up to some stopping time $\tau > t$. Define 
$$
F(t) := \{i \in [n]; u_t(i) = 0\},
$$ 
understood as the set of coordinates that have not been pinned. For any $i \in F_t$ consider a random variable $T_i$ defined by the formula
$$
\PP(T_i > s) = \exp \left (- \int_t^s (1 - s_i(r) a_i(v(r), u(t)) \bigr ) |v_i'(r)| dr \right ),
$$
where $s_i(r) := \sign(v_i'(r))$ (note that the right-hand side is $u(t)$-measurable), and such that the $T_i$'s are independent. Define 
$$
\tau := \min_{i \in F_t} T_i~~~\mbox{and}~~~ J := \argmin_{i \in F_t} T_i.
$$ 
We can now extend the process up to time $\tau$ by setting 
$$
u(s) = 
\begin{cases}
u(t), & t \leq s < \tau \\
u(t) - s_J(\tau) \mathbf{e}_J, & s=\tau,
\end{cases}
$$
where $\mathbf{e}_J$ is the $J$-th standard basis vector. In words, up to the stopping time $\tau$, we apply an exponential tilt to the measure according to $v(\cdot)$, and at time $\tau$ we pin the $j$-th coordinate according to the sign of $v_j'(\tau)$. We can now define the process for all times by iteratively extending it until the next stopping time. Since at any such stopping time, one of the coordinates is pinned, we only need to repeat this iteration at most $n$ times and thus the process is well-defined for $t \in [0, \infty)$.

Next, we show that the process is a martingale. First observe that the process is time-equivariant in the sense that the evolution of $(\nu_s)_{s \geq t}$ conditioned on $(\nu_r)_{r \in [0,s]}$ is the same as the evolution of $(\tilde \nu_t)_t$, the process obtained from the starting measure $\tilde \nu = \nu_s$ and using $\tilde v(t) = v(t+s) - v(s)$ as the driving curve. Indeed, this follows from the memoryless property of exponential random variables, which implies that for every $t > 0$ and every $i \in F_t$, we have that $T_i | T_i > s$ has the same distribution as $\tilde T_i$. 

In order to show that the process is a martingale, we need to show that for every $s,t \geq 0$ we have $\EE[\nu_{t+s} | \nu_s] = \nu_s$
and by the above discussion it is enough to consider the case $s=0$, and show that $\EE[\nu_{t}(x)] = \nu(x)$. We claim that it is yet enough to show that
\begin{equation}\label{eq:marh}
\EE[\nu_h(x)] = \nu(x) + o(h), ~~~ \forall h \geq 0.
\end{equation}
Indeed, if this is the case then we can write 
$$
\EE[\nu_t] = \sum_{i=1}^k \EE \left [\EE \left [\nu_{it/k} - \nu_{(i-1)t/k} | \nu_{(i-1)t/k}  \right ] \right ] \leq k o(1/k),
$$
for all $k \in \mathbb{N}$, which implies that $\EE[\nu_{t}(x)] = \nu(x)$.

Let us now prove \eqref{eq:marh}. By reflecting both $\nu$ and $v(\cdot)$ around coordinate directions, we may clearly assume without loss of generality that $v_i'(0) \geq 0$ for all $i \in [n]$. Since for all $i$ we have $\PP(T_i) \in [0,h] = O(h)$, we can write
\begin{align*}
\nu_{h} & = \tilt_{v(h)} \left (\mathbf{1}_{\{ \tau > h\} } \nu + \mathbf{1}_{\{ \tau \in [0,h] \} } \pin_{-\sign (v_j'(\tau)) \mathbf{e}_J} \nu \right ) + o(h) \\
& = \tilt_{h v'(0)} \left ( \mathbf{1}_{\{ \tau > h\} } \nu + \sum_{i \in [n]} \mathbf{1}_{\{ T_i \in [0,h] \}} \pin_{-\mathbf{e}_i} \nu \right ) + o(h) \\
& = \tilt_{h v'(0)} \left ( \nu + \sum_{i \in [n]} \mathbf{1}_{\{ T_i \in [0,h] \} } \left (\pin_{-\mathbf{e}_i} \nu - \nu \right ) \right ) + o(h).
\end{align*}
Now, note that for all $x \in \{-1,1\}^n$,
$$
\tilt_{hv'(0)} \nu (x) = \nu(x) (1+\langle x-\b(\nu), hv'(0) \rangle ) + o(h),
$$
and
$$
\pin_{-\mathbf{e}_i} \nu(x) = \nu(x) \left (1 - \frac{x_i-\b(\nu)_i}{1 - \b(\nu)_i} \right ).
$$
Combining the last three displays gives
\begin{align}
\EE[\nu_{h}(x)] & = \nu(x) \left ( 1+\langle x-\b(\nu), hv'(0) \rangle - \sum_{i \in [n]} \mathbf{1}_{\{ T_i \in [0, h] \}} \frac{x_i-\b(\nu)_i}{1 - \b(\nu)_i}  \right ) + o(h) \nonumber \\
& = \nu(x) \left ( 1+ \sum_{i \in [n]} (x_i-\b(\nu)_i) \left ( h v_i'(0)  -  \frac{\mathbf{1}_{\{ T_i \in [0,h] \}}}{1 - \b(\nu)_i}  \right ) \right ) + o(h), \label{eq:PinSL1}
\end{align}
at which point we have established the correctness of \eqref{eq:dnft}. By definition of the random variables $T_i$, we have
$$
\PP(T_i \in [0,h]) = h (1-\b(\nu)_i) v_i'(0) + o(h),
$$
and therefore
$$
\EE\left [ \left . h v_i'(0)  -  \frac{\mathbf{1}_{\{ T_i \in [0,h] \}}}{1 - \b(\nu)_i} \right . \right ] = o(h).
$$
Plugging the last display into \eqref{eq:PinSL1} implies that $\EE[\nu_{h}(x)] = \nu(x) + o(h)$ which proves \eqref{eq:marh}, so we have established that the process is a martingale.

Under the extra condition that $\lim_{t \to \infty} |v_i(t)| = \infty$, we just observe that $\nu_t$ converges to a Dirac measure almost-surely under every pinning process $u(t)$, and therefore the process is a localization process.
\end{proof}

\begin{proof}[Proof of Proposition \ref{prop:NFentdecay}]
Fix a measure $\nu$ on $\{-1,1\}^n$ and let $(\nu_t)_t$ be the process obtained via the negative-fields localization process. Fix $t>0$. According to equation \eqref{eq:SLjump1}, we can write
\begin{equation}
\nu_{t+h}(x) = \nu_t(x) \left ( 1 + \langle x - \b(\nu_t), Z \rangle \right ) + o(h),
\end{equation}
where, conditional on $\nu_t$, the random variable $Z$ has independent coordinates which satisfy
$$
Z_i = 
\begin{cases}
-h & \mbox{with probability } 1-h (1+\b(\nu_t)_i) \\
\frac{1}{1+b(\nu_t)_i} & \mbox{with probability } h (1+\b(\nu_t)_i).
\end{cases}
$$
Let $f:\{-1,1\}^n \to \RR_+$ be such that $\int f d \nu > 0$. We have,
\begin{align*}
\int f(x) d \nu_{t+h}(x) & = \int f(x) d \nu_t(x) \left ( 1 + \left \langle \left (\frac{\int x f(x) d\nu_t(x)}{\int x f(x) \nu_t(dx)} - \b(\nu_t) \right ), Z \right \rangle \right ) + o(h) \\
& = \nu_t(f) \left ( 1 + \langle V, Z \rangle \right ) + o(h),
\end{align*}
where 
$$
v := \frac{\int x f(x) \nu_t(dx)}{\int f(x) \nu_t(dx)} - \b(\nu_t) = \b(\mu_t) - \b(\nu_t)
$$
and where
$$
\frac{d \mu_t}{d \nu_t}(x) := \frac{f(x)}{\int f(x) d \nu_t(x)}.
$$
Therefore, we have
\begin{equation} \label{eq:dflogf}
\EE \brackets{ \nu_{t+h}(f) \log \nu_{t+h}(f)} - \nu_t(f) \log \nu_t(f) = \nu_t(f) \EE \brackets{ (1+\langle Z, v \rangle) \log(1+ \langle Z, v \rangle)  } + o(h).
\end{equation}
For a fixed $t$, define $b_i = \b(\nu_t)_i$. We calculate,
\begin{align*}
&\quad \EE \brackets{ (1+\langle Z, v \rangle) \log(1+ \langle Z, v \rangle)  } \\
& = - h \sum_{i=1}^n v_i + h \sum_{i=1}^n (1+b_i) \left (1+ \frac{v_i}{1+b_i} \right ) \log \left (1+\frac{v_i}{1+b_i} \right ) + o(h) \\
& = h \sum_{i=1}^n (1 + b_i) \left ( \left (1+ \frac{v_i}{1+b_i} \right ) \log \left (1+\frac{v_i}{1+b_i} \right ) - \frac{v_i}{1+b_i} \right ) + o(h) \\
& \stackrel{ \eqref{eq:hphicomp} }{\leq} 4 h H(b + v, b) + o(h),
\end{align*}
Combining the two last displays yields 
$$
\EE[ \Ent_{\nu_{t+h}}[f] | \nu_t ] \geq \Ent_{\nu_t}[f] - 4 h \nu_t(f) H(b+v, b) + o(h)
$$
The assumption that $\nu$ is $\alpha$-entropically stable with respect to $H(x,y)$ together with Lemma  \ref{lem:maxent} give
$$
H(b+v,b) \leq \alpha \KL(\mu_t || \nu_t) = \alpha \frac{\Ent_{\nu_t}[f]}{\nu_t(f)}.
$$
Combining the last two displays completes the proof.
\end{proof}

\subsection{Proofs of technical lemmas from Section \ref{sec:CLV}}

\begin{proof}[Proof of Lemma \ref{lem:cormar}]
For $k \in [n]$. Let $Z \sim \nu$ and define $X=Z_j$ and $Y=Z_i$. We have
\begin{align*}
\COV(\nu)_{ij} & = \EE[XY] - \EE[X] \EE[Y] \\
& = \EE[X|Y=1] \PP(Y=1) - \EE[X|Y=-1] \PP(Y=-1) - \EE[X] \EE[Y] \\
& = 2 \EE[X|Y=1] \PP(Y=1) - \EE[X] - \EE[X] \EE[Y] \\
& = \left (\EE[X | Y=1] - \EE[X] \right ) \left ( 1 + \EE[Y] \right ).
\end{align*}
Therefore, we have
\begin{align*}
\b \left (\pin_{\ee_i} \nu \right )_j - \b(\nu)_j & = \EE[X | Y=1] - \EE[X] = \frac{\Cov(\nu)_{ij}}{1+\b(\nu)_i}
\end{align*}
which proves the lemma for the case $s = +1$. The proof for $s=-1$ is similar.
\end{proof}

\begin{proof}[Proof of Lemma \ref{lem:HPhi}]
	We first deal with equation~\eqref{eq:hphicomp}. Fixing $y \in (-1,1)$, we calculate the first derivatives with respect to $x$
	\begin{align*}
	\frac{\partial }{\partial x} H(x, y) &= \frac{1}{2} \log\parenth{\frac{1+x}{1+y}} - \frac{1}{2} \log\parenth{\frac{1-x}{1-y}}\\
	\frac{\partial }{\partial x} (1+y)\Phi\parenth{\frac{x-y}{1+y}} &= \log\parenth{\frac{1+x}{1+y}}.
	\end{align*}
	The two functions coincide up to the first order Taylor expansion in $x$ at $x=y$. Next we look at the second derivatives
	\begin{align*}
	\frac{\partial^2}{\partial x^2}H(x,y) &= \frac{1}{(1+x)(1-x)} \\
	\frac{\partial^2}{\partial x^2} (1+y) \Phi \left (\frac{x-y}{1+y} \right ) &= \frac{1}{1+x}.
	\end{align*}
	Since $1 \leq 2 \frac{1}{1-x}$, we obtain the inequality on the right-hand side. Similarly, since $\frac{1}{1-x} \leq 2$ when $x \leq 1/2$, we also obtain the inequality on the left-hand side for the case $x \leq 1/2$. For $x > 1/2$, let
	\begin{align*}
	G(x, y) = 2 (1+y) \Phi (\frac{x-y}{1+y}) - H(x, y).
	\end{align*}
	Its first derivative is as follows
	\begin{align*}
	\frac{\partial }{\partial x} G(x, y) = \frac{1}{2} \log\parenth{\frac{(1+x)^3 (1-x)}{(1+y)^3 (1-y)}}.
	\end{align*}
	We observe that $x \mapsto (1+x)^3 (1-x)$ is monotonically decreasing on $[1/2, 1]$. Because $\frac{\partial }{\partial x} G(0, y) < 0$, $x \mapsto G(x, y)$ is either decreasing on $[1/2, 1]$ or increasing then decreasing on $[1/2, 1]$. So it is sufficient to check that $G(\frac{1}{2}, y) \geq 0$ and $G(1, y) \geq 0$ to verify $G(x, y) \geq 0$ for $x \in [1/2, 1]$. We verify that
	\begin{align*}
	G(1, y) = 3 \log\parenth{\frac{2}{1+y}} - 1 + y \geq 0,
	\end{align*}
	because $\frac{\partial}{\partial y} G(1, y) = -3/(1+y) + 1\leq 0$ and $G(1, 0) \geq 0$. And we have
	\begin{align*}
	G(\frac{1}{2}, y) = \frac{5}{2} \log\parenth{\frac{1}{1+y}} + 2y + \frac{9}{4} \log(3/2) - 1 - \frac{1}{4} \log(1/2) \geq 0,
	\end{align*}
	because $\frac{\partial}{\partial y} G(\frac{1}{2}, y) = - \frac{5}{2(1+y)} + 2 \leq 0$ and $G(\frac{1}{2}, 0) \geq 0$.
	This proves \eqref{eq:hphicomp}. To prove \eqref{eq:hphicomp2}, we note that by using \eqref{eq:hphicomp}, it is enough to show that 
	$$
	2\Phi(\delta s) \geq \delta^2 \Phi(s), ~~~ \forall \delta \in (0,1), s \in [-1,\infty),
	$$ 
	and then take $s = \frac{x-y}{(1+y)} \geq -1$ and $\delta = \eps^{-1}$. To prove the last inequality, let $F(s) = 2\Phi(\delta s) - \delta^2 \Phi(s)$. We have
	\begin{align*}
	F'(s)  &= 2\delta \Phi'(\delta s) - \delta^2 \Phi'(s) = 2\delta \log(1+\delta s) - \delta^2 \log(1+s). \\
	F''(s) &=  2\delta^2 \frac{1}{1+\delta s} - \delta^2 \frac{1}{1+s}
	\end{align*}
	For $s \geq 0$, since $F'(0) = 0$ and $F''(s) \geq 0$, we obtain that $F(s)$ is an increasing function of $s$ and hence $F(s) \geq F(0) = 0$. For $s \in [-1, 0)$, since as a function of $\delta$, $F'(s)$ is decreasing, we have $F'(s) \geq  0$. We conclude that $F(s)$ is an increasing function of $s$ on $[-1, 0)$. Hence, $F(s) \geq F(-1) = 2\Phi(-\delta) - \delta^2$. Since
	\begin{align*}
	\frac{\partial }{\partial \delta} [2\Phi(-\delta) - \delta^2] = 2 (-\delta - \log(1-\delta)) \geq 0,
	\end{align*}
	we have $F(s) \geq 2\Phi(-\delta) - \delta^2 \geq 0$.
	
	Finally, we deal with equation~\eqref{eq:hphicomp3}. For $s \geq 0$, there is nothing to prove since $\Phi(\cdot) \geq 0$ on $[-1, 1]$. For $s \in [-1, 0)$, let $J(s) = 3\Phi(-s) - \Phi(s)$. It has derivatives
	\begin{align*}
	J'(s) &= -3\log(1-s) - \log(1+s) \\
	J''(s) &= \frac{3}{1-s} - \frac{1}{1+s}.
	\end{align*}
	Since $J''(-1/2) = 0, J''(0) > 0$ and $J''(-1) < 0$, we have that $J'(s)$ is minimized at $1/2$ and that it only crosses $0$ once on $(-1, 0)$. To check $J(s) \geq 0$, it is sufficient to check $J(0) \geq 0$ and $J(-1) \geq 0$. We conclude equation~\eqref{eq:hphicomp3}.
\end{proof}

\begin{proof}[Proof of Lemma \ref{lem:tiltmarginals}]
Clearly we may assume $u = 0$ (otherwise we just need to prove the lemma on for the restriction of $\nu$ to $S_u$). Denote $v = v_i \ee_i + \tilde v$ where $\tilde v \perp \ee_i$. We have
$$
\tilt_v \nu = \tilt_{v_i \ee_i} \tilt_{\tilde v} \nu
$$
Define $\mu = \tilt_{\tilde v} \nu$. 
Note that 
\begin{align*}
	\b_i(\mu) = \Exs_{X \sim \tilt_{\tilde{v}}\nu} \brackets{\b_i(\pin_{\mathrm{Proj}_{\ee_i^\perp} X} \nu)}.
\end{align*}
According to equation~\eqref{eq:boundedmar3}, we have $1 + \b_i(\pin_{\mathrm{Proj}_{\ee_i^\perp} X} \nu) \leq \delta$. Hence
\begin{align*}
	1 + \b_i(\mu) \leq \delta.
\end{align*}
Define $f(t) = \b_i(\tilt_{t\ee_i} \mu)$. If $v_i < 0$ then $f(v_i) < f(0)$ and there is nothing to prove. Otherwise, we have
\begin{align*}
	f(0) \leq -1 + \delta, \quad f'(t) = 1 - f(t)^2.
\end{align*}
Thus,
\begin{align*}
	(1+f(t))' \leq 2 (1 + f(t)). 
\end{align*}
By Gronwall's inequality,
\begin{align*}
	1+f(t) \leq (1+ f(0)) \exp(2t) \leq \delta \exp(2t). 
\end{align*}
Therefore,
\begin{align*}
	1+\b_i(\tilt_v(\nu)) = 1+ f(v_i) \leq \delta \exp(2v_i).
\end{align*}
For the second part, let $g(t) = \frac{1+ f(t)}{1-f(t)}$. Since $x \mapsto \frac{1+x}{1-x}$ is an increasing function of $x \in (-1, 1)$, we have
\begin{align*}
	\delta' \leq g(0) \leq \delta''.
\end{align*} 
If $v_i < 0$, then $f(v_i) < f(0)$ and $g(v_i) < g(0)$.  Otherwise, we have
\begin{align*}
	g'(t) = \frac{2 f'(t)}{(1-f(t))^2} = 2 g(t).  
\end{align*}
Integrating the above equation, we obtain
\begin{align*}
	g(t) = g(0) \exp(2t) \leq \tilde{\delta} \exp(2t).
\end{align*}
The lower bound is established similarly.
\end{proof}

\section{Appendix B: Entropic stability for Ising models under a spectral condition}
\begin{proof}[Proof of Lemma \ref{lem:SKcov}]
Consider the Ising model whose density is
$$
\nu(x) \propto \exp\left ( \langle Jx,x \rangle + \langle v, x \rangle  \right )
$$
Since $v$ is arbitrary, it suffices to show that
\begin{equation} \label{eq:condSLC2}
\COV(\nu) \preceq \frac{1}{1-2\|J\|_\OP}.
\end{equation}

Set $\alpha > 2 \|J\|_{\OP}$ whose value will be chosen later. Let $M$ be a matrix such that $M^{-1} + \alpha^{-1} \Id = (\alpha \Id - 2 J)^{-1}$. By the fact that 
$$
\mathcal{N}(0, M^{-1}) + \mathcal{N}(0, \alpha^{-1} \Id) \stackrel{(d)}{=} \mathcal{N}(0, (\alpha \Id - 2 J)^{-1}),
$$
we have
\begin{align*}
\exp\left ( -\frac{1}{2} \langle (\alpha \Id - 2J)x,x \rangle \right ) & \propto \int_{\RR^n} e^{-\frac{1}{2} \langle y, M y \rangle} e^{- \frac{\alpha}{2} |x-y|^2} dy \\
& \propto  e^{- \alpha |x|^2/2} \int_{\RR^n} e^{-\frac{1}{2} \langle y, (M + \alpha \Id) y \rangle} e^{\alpha \langle x,y \rangle} dy \\
& \propto e^{- \alpha |x|^2/2} \int_{\RR^n} e^{-\frac{1}{2} \langle y, (\alpha^2 (2J)^{-1}) y \rangle} e^{\alpha \langle x,y \rangle} dy.
\end{align*}
Since $|x|$ is constant on $\{-1,1\}^n$, we get that there is a constant $C(J,v)$ such that
$$
\exp\left ( \langle (\alpha Jx,x \rangle + \langle v, x \rangle \right ) = C(J,v) \int_{\RR^n} e^{-\frac{1}{2} \langle y, (\alpha^2 (2J)^{-1}) y \rangle} e^{\langle x,v+\alpha y \rangle} dy, ~~ \forall x \in \{-1,1\}^n.
$$
It follows that
\begin{align*}
\nu(x) & \propto \int_{\RR^n} e^{-\frac{1}{2} \langle y, (\alpha^2 (2J)^{-1}) y \rangle} e^{\langle x,v+\alpha  \rangle} dy \\
& \propto  \int_{\RR^n} e^{-\frac{1}{2} \langle y, (\alpha^2 (2J)^{-1}) y \rangle + \log(Z(v+\alpha y))} \left ( Z(v+y \alpha )^{-1} e^{\langle x,v+\alpha y \rangle} \right ) dy,
\end{align*}
where $Z(w) = \int_{\{-1,1\}^n} \exp(\langle x, w \rangle) dx$. Since $Z(v+ \alpha y)^{-1} e^{\langle x,v+\alpha y \rangle}$ is the density of $\tilt_{v+ \alpha y} \mu$ where $\mu$ is the uniform measure, we have
$$
\nu = \int_{\RR^n} f(y) \tilt_{v+\alpha y} \mu dy,
$$
where
$$
f(y) \propto \exp\left (-\frac{1}{2} \langle y, (\alpha^2 (2J)^{-1}) y \rangle + \log(Z(\alpha y)) \right ).
$$
Let $X$ be the random vector with density $\frac{f}{\int f}$. Then the above formula and the law of total variance give
\begin{align}
\COV(\nu) & = \COV(\EE[\tilt_{v+\alpha X} \nu]) + \EE[ \COV(\tilt_{v+\alpha X} \mu) ] \nonumber \\
& = \COV(\tanh(v+\alpha X)) + \EE[ \COV(\tilt_{v+\alpha X} \mu) ] \nonumber \\
& \preceq \alpha^2 \COV(X) + \Id, \label{eq:covnu2}
\end{align}
where we used the fact that $\tanh(X)$ is a contraction.

Since $\log Z(\cdot)$ is the logarithmic Laplace transform of the uniform measure, we obtain that $\nabla^2 \log(Z(w)) \preceq \Id$. Therefore, defining
$$
U(y) = \frac{1}{2} \langle y, (\alpha^2 (2J)^{-1}) y \rangle - \log(Z(\alpha y)),
$$
we have 
$$
\nabla^2 U \succeq \alpha^2 (2J)^{-1} - \alpha^2 \Id \succeq  \left (\frac{\alpha^2}{2 \|J\|_\OP} - \alpha^2 \right ) \Id.
$$ 
An application of Theorem \ref{thm:stronglc} gives 
$$
1+ \|\alpha^2 \COV(X)\|_\OP \leq 1 + \frac{2 \|J\|_\OP}{1-2 \|J\|_\OP} = \frac{1}{1-2 \|J\|_\OP}.
$$
Combining with \eqref{eq:covnu2} completes the proof.
\end{proof}

\section{Appendix C: The hardcore model} \label{appendix:HC}

\begin{proof}[Proof of Lemma~\ref{lem:marginal_bounds_of_hardcore}]
  Denote $\mathfrak{E}(\sigma)$ the event $\sigma_i u_i \geq 0, \forall i \in [n]$.
  For the upper bound, we have
  \begin{align*}
    \Prob_{\sigma \sim \nu} \parenth{\sigma_v = +1 \mid \mathfrak{E}(\sigma) } &= \frac{\sum_{\sigma_v = +1, \sigma_i u_i \geq 0, \forall i} \lambda^{\abss{I_\sigma}}}{\sum_{\sigma_v = +1, \sigma_i u_i \geq 0, \forall i} \lambda^{\abss{I_\sigma}}+\sum_{\sigma_v = -1, \sigma_i u_i \geq 0, \forall i} \lambda^{\abss{I_\sigma}}} \\
    &\overset{(i)}{\leq} \frac{\sum_{\sigma_v = +1, \sigma_i u_i \geq 0, \forall i} \lambda^{\abss{I_\sigma}}}{\sum_{\sigma_v = +1, \sigma_i u_i \geq 0, \forall i} \lambda^{\abss{I_\sigma}}+\sum_{\sigma_v = +1, \sigma_i u_i \geq 0, \forall i} \lambda^{\abss{I_\sigma} - 1}} \\
    &= \frac{\lambda}{\lambda + 1}.
  \end{align*}
  Inequality (i) follows from the observation that any configuration $\sigma$ with $\sigma_v = +1$ gives a configuration $\sigma'$ with $\sigma'_v=-1$ and $\sigma'_w = \sigma_w, \forall w \in V \setminus \braces{v}$. \\
  For the lower bound, let $(u_1, \ldots, u_K)$ be the elements of $N_v$, $K = \abss{N_v}$, we have
  \begin{align*}
    &\quad \Prob_{\sigma \sim \nu} \parenth{\sigma_v = +1 \mid \mathfrak{E}(\sigma) } \\
    &\geq \Prob_{\sigma \sim \nu} \parenth{\sigma_v = +1 \text{ and } \sigma_{u_k} = -1, \forall k\in [K] \mid \mathfrak{E}(\sigma) } \\
    &= \Prob_{\sigma \sim \nu} \parenth{\sigma_{u_k} = -1, \forall k\in [K] \mid \mathfrak{E}(\sigma) } \cdot \Prob_{\sigma \sim \nu} \parenth{\sigma_v = +1 \mid \sigma_{u_k} = -1, \forall k\in [K], \mathfrak{E}(\sigma) }
  \end{align*}
  Conditioned on all neighbors of $v$ being $-1$, $\sigma_v$ can either be $-1$ or $+1$. We have
  \begin{align*}
    \Prob_{\sigma \sim \nu} \parenth{\sigma_v = +1 \mid \sigma_{u_k} = -1, \forall k\in [K], \mathfrak{E}(\sigma) } = \frac{\lambda}{1+\lambda}.
  \end{align*}
  On the other hand, we have
  \begin{align*}
    &\quad \Prob_{\sigma \sim \nu} \parenth{\sigma_{u_k} = -1, \forall k\in [K] \mid \mathfrak{E}(\sigma) } \\
		&= \prod_{i=1}^K\Prob_{\sigma \sim \nu} \parenth{\sigma_{u_{i}} = -1 \mid \sigma_{u_j} = -1, \forall j\in [i-1], \mathfrak{E}(\sigma) } \\
    &= \prod_{i=1}^K \parenth{1-\Prob_{\sigma \sim \nu} \parenth{\sigma_{u_{i}} = +1 \mid \sigma_{u_j} = -1, \forall j\in [i-1], \mathfrak{E}(\sigma) }} \\
    &\overset{(i)}{\geq} \prod_{i=1}^K \parenth{1- \frac{\lambda}{1+\lambda}} \\
    & = \parenth{\frac{1}{1+\lambda}}^{\abss{N_v}}.
  \end{align*}
  Inequality (i) follows from the upper bound of the marginal in the first part. \\
  If in addition $\lambda \leq (1-\delta) \lambda_{\Delta}$, then we have
  \begin{align*}
    \parenth{\frac{1}{1+\lambda}}^{\abss{N_v}}
    &\geq \parenth{\frac{1}{1+\lambda_\Delta}}^{\abss{N_v}} \\
    &\geq \parenth{\frac{1}{1+\lambda_\Delta}}^{\Delta} \\
    &\overset{(i)}{\geq} \parenth{\frac{1}{1+3e^2/\Delta}}^{\Delta} \\
    &\overset{(ii)}{\geq} e^{-3e^2}.
  \end{align*}
  Inequality (i) follows from the fact that for $\Delta \geq 3$,
  \begin{align*}
    \lambda_\Delta = \frac{(\Delta-1)^{\Delta-1}}{(\Delta-2)^\Delta} = \frac{1}{\Delta-2} \parenth{1 + \frac{1}{\Delta-2}}^{\Delta-1} \leq \frac{1}{\Delta-2} e^{\frac{\Delta-1}{\Delta-2}} \leq \frac{3e^2}{\Delta}.
  \end{align*}
  Inequality (ii) follows from $\log(1+x) \leq x, \text{ for } x \in (-1, 1)$.
\end{proof}

\end{document}